\documentclass[10pt]{article}

\usepackage[dvipdf]{graphicx}
\usepackage[latin1]{inputenc}
\usepackage{latexsym}
\usepackage{amsmath,amssymb,amsfonts,amsthm}
\usepackage{xcolor}
\usepackage{mathrsfs}
\usepackage{bbm}
\usepackage{bm}
\usepackage{enumitem}

\usepackage[numbers,comma,sort]{natbib}

\definecolor{db}{RGB}{0, 0, 130}
\definecolor{wildstrawberry}{rgb}{1.0, 0.26, 0.64}

\usepackage[colorlinks=true,citecolor=db,linkcolor=black,urlcolor=blue,pdfstartview=FitH]{hyperref}

\usepackage{pdfsync}

\definecolor{rp}{rgb}{0.25, 0, 0.75}
\definecolor{dg}{rgb}{0, 0.5, 0}

\newcommand{\R}{\mathbb{R}}

\newcommand{\Q}{\mathbb{Q}}

\newcommand{\N}{\mathbb{N}}
\newcommand{\EE}{\mathbb{E}}

\newcommand{\PP}{\mathbb{P}}

\newcommand{\Ccal}{\mathcal{C}}

\newcommand{\Pp}{\mathcal{P}}

\newcommand{\pK}{{\bm p}}
\newcommand{\qK}{{\bm q}}
\newcommand{\rK}{{\bm r}}
\newcommand{\rKs}{\widetilde{\bm r}}
\newcommand{\Hreg}{{\bm \zeta}}
\newcommand{\Lsol}{L^{1}\cap L^{\rK}(\R^d)}
\newcommand{\Lsolsing}{L^{1}\cap L^{\infty}(\R^d)}

\newcommand{\Hyp}{\mathbf{A}}
\newcommand{\HypK}{\mathbf{A}^K}
\newcommand{\HypKK}{\widetilde{\mathbf{A}}^K}
\newcommand{\HypKKK}{\widetilde{\widetilde{\mathbf{A}}}{}^K}

\makeatletter
\newcommand{\customlabel}[2]{%
   \protected@write \@auxout {}{\string\newlabel {#1}{{#2}{\thepage}{#2}{#1}{}}}%
   \hypertarget{#1}{#2\hspace{-0.14cm}}
}
\makeatother

\numberwithin{equation}{section}

\newtheorem{theorem}{Theorem}[section]

\newtheorem{definition}[theorem]{Definition}

\newtheorem{corollary}[theorem]{Corollary}

\newtheorem{lemma}[theorem]{Lemma}

\newtheorem{proposition}[theorem]{Proposition}
\newtheorem{remark}[theorem]{Remark}

\usepackage{etoolbox}

\makeatletter
\newcounter{author}
\renewcommand*\author[1]{%
  \stepcounter{author}%
  \ifnum\c@author=1
    \gdef\@author{#1}%
  \else
    \xdef\@author{\unexpanded\expandafter{\@author\and#1}}%
  \fi
  \csgdef{author@\the\c@author}{#1}}
\newcommand*\email[1]{%
  \csgdef{email@\the\c@author}{#1}}
\newcommand*\address[1]{%
  \csgdef{address@\the\c@author}{#1}}
\AtEndDocument{%
  \xdef\author@count{\the\c@author}%
  \c@author=1
  \print@authors}
\newcommand*\print@authors{%
  \ifnum\c@author>\author@count
  \else
    \print@author{\the\c@author}%
    \advance\c@author by 1
    \expandafter\print@authors
  \fi}
\newcommand*\print@author[1]{%
  \par\medskip
  \begin{tabular}{@{}l@{}}%
    \textsc{\csuse{author@#1}}\\
    \csuse{address@#1}\\
    \textit{E-mail address}:
    \href{mailto:\csuse{email@#1}}{\csuse{email@#1}}
  \end{tabular}}
\makeatother

\author{Christian Olivera}
\address{Departamento de Matem\'atica, Universidade Estadual de Campinas, Brazil} 
\email{colivera@ime.unicamp.br}

\author{Alexandre Richard}
\address{Universit\'e Paris-Saclay, CentraleSup\'elec, MICS and CNRS FR-3487, 91190 Gif-sur-Yvette, France}
\email{alexandre.richard@centralesupelec.fr}

\author{Milica Toma\v sevi\'c}
\address{CNRS, CMAP, Ecole Polytechnique, I.P. Paris, 91128 Palaiseau, France}
\email{milica.tomasevic@polytechnique.edu}

\usepackage{fancyhdr}
\date{\empty}

\title{ \Large{\textbf{Quantitative particle approximation of nonlinear Fokker-Planck equations with singular kernel}}}

\begin{document}

\maketitle

\begin{abstract}
In this work, we study the convergence of the empirical measure of moderately interacting particle systems with singular interaction kernels. First, we prove quantitative convergence of the time marginals of the empirical measure of particle positions towards the solution of the limiting nonlinear Fokker-Planck equation. Second, we prove the well-posedness for the McKean-Vlasov SDE involving such singular kernels and the convergence of the empirical measure towards it (propagation of chaos).

Our results only require very weak regularity on the interaction kernel, which permits to treat models for which the mean field particle system is not known to be well-defined. For instance, this includes attractive kernels such as Riesz and Keller-Segel kernels in arbitrary dimension. For some of these important examples, this is the first time that a quantitative approximation of the PDE is obtained by means of a stochastic particle system. In particular, this convergence still holds (locally in time) for PDEs exhibiting a blow-up in finite time. 
 The proofs are based on a semigroup approach combined with a fine analysis of the regularity of infinite-dimensional stochastic convolution integrals. 
\end{abstract}


\medskip

\noindent\textbf{MSC2020 subject classification: }60K35, 60H30, 35K55, 35Q84.

\setcounter{tocdepth}{2}
\renewcommand\contentsname{}
\vspace{-1cm}

\tableofcontents

\section{Introduction and main results}

In this work, we are interested in the stochastic approximation of nonlinear Fokker-Planck Partial Differential Equations (PDEs) of the form
\begin{equation}\label{eq:PDE}
\begin{cases}
&\partial_t u(t,x) = \Delta u(t,x) - \nabla \cdot \big(u(t,x) ~  K \ast_{x} u (t,x) \big),\quad t>0,~x\in\R^d,\\
&u(0,x) = u_0(x), 
\end{cases}
\end{equation}
by means of moderately interacting particle systems. The main interest here is for kernels $K$ with a singularity at the origin.

Although we will not limit ourselves to kernels that derive from a potential, a typical family of singular kernels that we will consider in this work derives from Riesz potentials, defined in any dimension $d$ as
\begin{align}\label{eq:defRiesz}
V_{s}(x) := 
\begin{cases}
|x|^{-s} & \text{ if } s\in(0,d) \\
-\log|x|  & \text{ if } s=0 
\end{cases}
, \quad x\in\R^d .
\end{align}
The associated kernel is then $K_{s}:= \pm \nabla V_{s}$, the sign deciding whether the interaction is attractive or repulsive. If $d\geq 2$ and $s=d-2$, this is the Coulomb potential that characterises electrostatic and gravitational forces (depending on the sign). 
Our results will also cover several classical models such as the $2d$ Navier-Stokes equation, which in vorticity form can be written as in \eqref{eq:PDE} with the Biot-Savart kernel, 
and the parabolic-elliptic Keller-Segel PDE in any dimension $d\geq 1$, which models the phenomenon of chemotaxis. These kernels are presented and discussed in more details in Section \ref{sec:examples}.

~

The problem of deriving a macroscopic equation from a microscopic model of interacting particles can be traced back to the original inspiration of \citet{Kac}, in the context of the Boltzmann equation. Since then, a huge literature has been devoted to interacting particle systems and their convergence to nonlinear Fokker-Planck equations such as \eqref{eq:PDE}. 
 In the case of Lipschitz continuous interaction kernels, this problem is now well-understood (see \cite{Sznit,MeleardLN,Oelschlager85,Meleard}) and the propagation of chaos holds. 
 That is to say, the empirical measure (on the space of trajectories) of the associated particle system converges in law towards the weak solution of a McKean-Vlasov SDE associated to \eqref{eq:PDE}.

There are fewer works when the interaction kernel is singular, despite the great importance 
it represents both theoretically and in applications. One can mention the early works of \citet{Marchioro} and \citet{Osada} on the $2d$ Navier-Stokes equation (see also \cite{FHM} more recently); while \citet{Stigm} and \citet{BossyTalay} worked on Burgers' equation.  \citet{CepaLepingle} studied one-dimensional electrical particles with repulsive interaction,
 and more recently \citet{FournierHauray} studied a stochastic particle system approximating the Landau equation with moderately soft potentials. 
Outside the scope of physics, interesting biological models have arisen, for instance in neuroscience and in the modelling of chemotaxis (see \cite{DIRT,FournierJourdain,CattiauxPedeches,JabirTalayTomasevic,Godinho}).

 When the interaction is attractive in addition to being singular, it may not be possible to define the particle system in mean-field interaction, and therefore to obtain propagation of chaos. Even when it is possible to define the particles, the propagation of chaos may not always hold. 
For example, in the tricky case of the $2d$ parabolic-elliptic Keller-Segel model, the mean-field particle system was shown to be well-defined \cite{FournierJourdain,CattiauxPedeches}, but the convergence (on the level of measures on the space of trajectories) is known to hold only for small values of the critical parameter of the equation (see \cite{FournierJourdain}). For the $d$-dimensional parabolic-elliptic Keller-Segel model with $d\geq3$ or for the attractive Riesz kernel with $s\in (d-2,d)$, we are not aware of any existence result for the mean-field particle system.

This work is motivated by the approximation of the PDE \eqref{eq:PDE} in cases when the kernel is attractive and singular, and the associated mean-field interacting particle system is not known to be well-defined. We thus consider the following moderately interacting particle system:
\begin{equation}\label{eq:IPS0}
dX_{t}^{i,N}= F_{A}\left( \frac{1}{N}\sum_{k=1}^{N} (K \ast V^{N})(X_{t}^{i,N} -X_{t}^{k,N})\right)\; dt + \sqrt{2} \; dW_{t}^{i}, \quad t\leq T,\quad  1\leq i \leq N,
\end{equation}
where $(W^i)_{1\leq i\leq N}$ are independant standard Brownian motions, $V^N$ is a mollifier and $F_{A}$ is a smooth cut-off function. 
 Hence,  the existence of strong solutions for \eqref{eq:IPS0} is ensured. 
  This kind of interacting particle system, introduced by~\citet{Oelschlager85}, relies on a smoothing of the interaction kernel at the scale $N^{-\alpha},\, \alpha\in [0,1]$. Observe that the interaction kernel $K\ast V^N(x)$ is very close to $K(x)$ when $|x|$ is sufficiently large compared to $N^{-\alpha}$. 
However, the difficulty is to prove that at the limit $N\to \infty$, the system \eqref{eq:IPS0} behaves as the mean-field system would. In this work, we show that:
\begin{enumerate}[label=(\alph*),leftmargin=*]
\item $\{\mu^N_{t}=\frac{1}{N} \sum_{i=1}^N \delta_{X^{i,N}_t},\, t\in [0,T]\}$, the marginals of the empirical measure of \eqref{eq:IPS0}, converge to the solution of the PDE \eqref{eq:PDE} for values of the smoothing parameter $\alpha$ which can be up to $(\frac{1}{d})^-$ for some of the models we consider, 
which is the typical distance between the particles in problems coming from statistical physics with repulsive kernels. 
We get a rate of convergence of order $N^{-\varrho}$, with $\varrho$ which can be up to $(\frac{1}{d+2})^-$. Moreover, one can take $F_{A}(x)\equiv x$ in \eqref{eq:IPS0}, at the price of a weaker form of convergence.

\item The system \eqref{eq:IPS0} propagates chaos towards the following nonlinear equation (without the cut-off and the mollifier):
\begin{equation}\label{eq:McKeanVlasov}
\begin{cases}
& d X_t = K \ast u_t (X_t)   \, dt +\sqrt{2} dW_t , \quad t\leq T,\\
& \mathcal{L}(X_t) = u_t, ~\mathcal{L}(X_0) = u_0. 
\end{cases}
\end{equation}
\end{enumerate}
Notice that it is not \emph{a priori} clear that \eqref{eq:McKeanVlasov} is well-posed, due to the singularity of $K$.  Hence, we obtain first a well-posedness result for \eqref{eq:McKeanVlasov}. 

As an application, one can use the rate in (a) and a time discretization of \eqref{eq:IPS0} to propose a numerical approximation of the PDE \eqref{eq:PDE}. We leave this line of investigation for a future work.

~

In the rest of this introduction, we will detail the notations and the assumptions necessary for our study. Then, we will present our main results in Subsections \ref{subsec:MainResult1} and \ref{subsec:MainResult2}. The organisation of the paper is then given in Subsection \ref{subsec:orga}.

~

\subsection{Notations}

\begin{itemize}[leftmargin=*]

\item  For any  $\beta\in \R$ and $p\geq1$,  we denote by
$H^{\beta}_{p}(  \mathbb{R}^{d})  $ the \emph{Bessel potential space}
\[H_{p}^{\beta}(  \mathbb{R}^{d}):= \Big\{ u \text{ tempered distribution; }  \;  \mathcal{F}^{-1}\Big( \big(1+|\cdot|^{2}\big)^{\frac{\beta}{2}}\; \mathcal{F} u(\cdot) \Big) \in  L^p(\mathbb R^d)\Big\},  \]
where $\mathcal Fu$ denotes the \emph{Fourier transform} of $u$. This space is endowed with the norm
 \begin{equation*}
 \| u \|_{\beta,p} = \left\| \mathcal{F}^{-1}\Big( \big(1+|\cdot|^{2}\big)^{\frac{\beta}{2}
}\; \mathcal{F} u(\cdot) \Big) \right\|_{L^p(\R^d)}. 
 \end{equation*}
 In particular, note that 
 \begin{equation}\label{eq:ineqBesselnorms}
  \left\Vert u\right\Vert _{0,p}=\left\Vert u\right\Vert_{L^p(\mathbb{R}^d)} \quad \text{and for any } \beta \leq \gamma, \quad \left\Vert u\right\Vert _{\beta,p} \leq \left\Vert u\right\Vert_{\gamma,p}.
 \end{equation}
The space $H_{p}^\beta(\R^d)$ is 
associated to the fractional operator $(I-\Delta)^\frac{\beta}{2}$ defined as (see e.g. \cite[p.180]{Triebel} for more details on this operator):
\begin{align}\label{eq:defFracLaplacian}
(I-\Delta)^\frac{\beta}{2} f := \mathcal{F}^{-1}\left((1+|\cdot|^2)^{\frac{\beta}{2}} \mathcal{F}f \right) .
\end{align}
For $p=2$, these are Hilbert spaces when endowed with the scalar product
\begin{equation}\label{eq:Hilbertbeta}
\langle u,v\rangle_{\beta} := \int_{\R^d} \left( 1+|\xi|^{2}\right)^{\beta} 
 \mathcal{F} u(\xi) ~ \overline{\mathcal{F} v}(\xi) ~d\xi
\end{equation} 
and the norm simply denoted by $\left\Vert u\right\Vert _{\beta} := \left\Vert u\right\Vert _{\beta,2} = \sqrt{\langle u,u\rangle_{\beta}}$.

In addition, for $p\geq 1$ we will denote by $p'$ its conjugate, i.e. $p'\geq 1$ such that $\frac{1}{p}+\frac{1}{p'}=1$.

 \item In this paper, $(e^{t\Delta })_{t\geq 0 }$ is the heat  semigroup. That is, for $f \in {L}^p(\R^d)$,
\[
\left(  e^{t\Delta}f\right)  \left(  x\right)  =\int_{\mathbb{R}^{d}} g_{2t}(x-y) \,  f\left(  y\right)  dy ,
\]
where $g$ denotes the usual $d$-dimensional Gaussian density function:
\begin{align}\label{eq:def-heat-kernel}
g_{\sigma^2}(x)= \frac{1}{(2\pi \sigma^2)^{\frac{d}{2}}} e^{-\frac{|x|^2}{2\sigma^2}} .
\end{align}

\item For $\mathcal{X}$ some normed vector space, the space $\Ccal([0,T];\mathcal{X})$ of continuous functions from the time interval $[0,T]$ with values in $\mathcal{X}$ is classically endowed with the norm
\begin{align*}
\|f\|_{T,\mathcal{X}} = \sup_{s\in [0,T]} \|f_{s}\|_{\mathcal{X}}.
\end{align*}

\item Let us denote by $\mathcal{N}_{\delta}$ the H\"older seminorm of parameter $\delta\in(0,1]$, that is, for any function $f$ defined over $\R^d$:
\begin{equation}\label{eq:HolderNorm}
\mathcal{N}_{\delta}(f) := \sup_{x\neq y \in \R^d} \frac{|f(x) - f(y)|}{|x-y|^\delta}.
\end{equation}
The set of continuous and bounded functions on $\R^d$ which have finite $\mathcal{N}_{\delta}$ seminorm is the H\"older space $\Ccal^\delta(\R^d)$.

\item If $u$ is a function or stochastic process defined on $[0,T]\times \R^d$, we will most of the time use the notation $u_{t}$ to denote the mapping  $x\mapsto u(t,x)$.

\item Depending on the context, the brackets $\langle\cdot , \cdot \rangle$ will denote either the scalar product in some $L^2$ space or the duality bracket between a measure and a function.

\item The ball centred in the origin that has radius $1$ is denoted by $\mathcal{B}_1$.

\item Finally, for functions from $\R^d$ to $\R$, the space of $n$-times ($n\in\N$) differentiable functions is denoted by $\Ccal^n(\R^d)$; the space of $n$-times ($n\in\N$) differentiable functions with bounded derivatives of any order between $0$ and $n$, denoted by $\Ccal_{b}^n(\R^d)$; and the space of $n$-times ($n\in\N$) differentiable functions with compact support, denoted by $\Ccal_{c}^n(\R^d)$. For $n=0$, we will denote the space of continuous (resp. bounded continuous, and continuous with compact support) by $\Ccal(\R^d)$ (resp. $\Ccal_{b}(\R^d)$ and $\Ccal_{c}(\R^d)$). 

\item Throughout the paper, we denote by $C$ a constant that may change from line to line.

\end{itemize}

\subsection{Setting}\label{subsec:assumptions}

We start this section with a precise definition of the particle system \eqref{eq:IPS0}. 

Let  $A>0$ and $f_A: \R \to \R$ be a $\Ccal_b^2(\R)$ function such that
\begin{enumerate}[label=(\roman*)]
\item $f_A(x)=x$, for $x\in [-A,A]$,
\item $f_A(x)= A$, for $x> A+1$ and $f_A(x)= -A$, for $x< -(A+1)$,
\item $\|f_{A}'\|_\infty \leq 1$ and $\|f_{A}''\|_\infty <\infty$.
\end{enumerate}
As a consequence, $\|f_A\|_\infty\leq A+1$.
Now $F_{A}$ is given by
\begin{equation}\label{eq:defF0}
 F_{A}  : (x_{1},\dots, x_{d})^T \mapsto \left(f_{A}(x_{1}) ,\dots, f_{A}(x_{d})\right)^T .
\end{equation}
Let $V:\R^d \to \R_+$ be a smooth probability density function, and assume further that $V$ is 
compactly supported. For any $x \in \mathbb R^d$, define
\begin{equation}\label{eq:VN}
V^{N}(x):=N^{d\alpha}V(N^{\alpha}x), \qquad \text{for some } \alpha \in [0,1].
\end{equation}
Below, $\alpha$ will be restricted to some interval $(0,\alpha_{0})$, see Assumption \eqref{C04}.

Let $T>0$. For each $N\in\mathbb{N}$, the particle system \eqref{eq:IPS0} reads more precisely:
\begin{equation}\label{eq:IPS}
\begin{cases}
dX_{t}^{i,N}=  F_{A}\bigg( \frac{1}{N}\displaystyle\sum_{k=1 }^{N} (K \ast V^{N})(X_{t}^{i,N} -X_{t}^{k,N})\bigg)\; dt + \sqrt{2} \; dW_{t}^{i}, \quad t\leq T,~  1\leq i \leq N,  \\
X_{0}^{i,N},~ 1\leq i \leq N, \quad \text{are independent of } \{W^i,~1\leq i \leq N\},
\end{cases}
\end{equation}
where  $\{(W_{t}^{i})_{t\in[0,T]}, \; i\in\mathbb{N}\}$ is a family of independent standard $\R^d$-valued Brownian motions defined on a filtered probability space $\left(\Omega,\mathcal{F},(\mathcal{F}_{t})_{t\geq 0},\PP\right)$.

Let us denote the empirical measure on $\mathcal{C}([0,T],\R^d)$ of $N$ particles by 
\begin{align}\label{eq:empiricalMeasure}
\mu^N  =  \frac{1}{N}\sum_{i=1}^N \delta_{X^{i,N}} ~,
\end{align}
and the mollified empirical measure by
$$u^N_{t} := V^N\ast \mu^N_{t} ~, \quad t\in[0,T].$$

The following properties of the kernel will be assumed:

~

\noindent \begin{minipage}{0.06\linewidth}
(\customlabel{HK}{$\HypK$}):
\end{minipage}
\hspace{-0.5cm}
\begin{minipage}[t]{0.94\linewidth}
\raggedright
 \hspace{0.8cm} 
\begin{enumerate}[label= ]
\item(\customlabel{HK1}{$\HypK_{i}$})   $K\in L^\pK (\mathcal{B}_{1})$, for some $\pK\in[1,+\infty]$;

\item(\customlabel{HK2}{$\HypK_{ii}$}) 
 $K\in L^{\qK}(\mathcal{B}_{1}^c)$, for some  $\qK\in[1,+\infty]$;
 
\item  (\customlabel{HK3}{$\HypK_{iii}$})  There exists  $\rK \geq \max(\pK',\qK')$, $\Hreg \in (0,1]$ and $C>0$ such that for any $f\in L^{1}\cap L^{\rK}(\R^d)$, one has
$$\mathcal{N}_\Hreg (K\ast f)\leq C  \|f\|_{L^{1}\cap L^{\rK}(\R^d) }. $$

\end{enumerate}
\end{minipage}

\vspace{0.3cm}

\noindent Throughout the paper, it will be supposed that the parameter $\rK$ is such that $\rK \geq \max (\pK', \qK')$, where $(\pK, \qK)$ are given in \eqref{HK}. Hence, even if some statements do not require the regularity assumption \eqref{HK3}, the parameter $\rK$ (when it appears) will always be considered  larger than  $\max (\pK', \qK')$.

  The Assumption \eqref{HK} is rather mild, 
 and we provide in Section \ref{app:CZ} a sufficient condition which is easier to check in concrete examples. In Section \ref{sec:examples}, we show that Riesz potentials (see \eqref{eq:defRiesz}) up to $s\leq d-2$, whether repulsive or attractive, satisfy \eqref{HK}. 
 Note that the case of more singular kernels, e.g. Riesz kernels with $s\in(d-2,d-1)$, will be treated separately in Section \ref{sec:sing}.

 ~
 
The restriction with respect to the key parameters is given by the following assumption:

~

\noindent 
\begin{minipage}{\linewidth}
(\customlabel{C04}{$\Hyp_{\alpha}$}): \hspace{0.35cm} The parameters $\alpha$ and $\rK$ (which appear respectively in \eqref{eq:VN} and \eqref{HK3}) satisfy 
$$ 0<\alpha<\frac{1}{d + 2 d (\frac{1}{2} - \frac{1}{\rK})\vee 0} .$$

\end{minipage}

\vspace{0.4cm}

Notice that if the integrability of the kernel is such that $\rK$ in \eqref{HK3} could be chosen in the interval $[1,2]$, then $\alpha$ could be arbitrarily close to, but smaller than, $\frac{1}{d}$ (although this will not yield the best possible rate of convergence, see Remark \ref{rk:optimalrate}). This scaling of order $N^{-(\frac{1}{d})^-}$ is close to the typical interparticle distance in physical models with repulsive interactions. This is what we get for instance in the case of $2d$ Navier-Stokes equation and $2d$ Coulomb potential.\\ However, in the case of attractive interaction, the typical interparticle distance is not necessarily $N^{-\frac{1}{d}}$ as the particles may collide or even agglomerate. Indeed, in the case of the parabolic-elliptic Keller-Segel equation, the associated mean-field particles  collide with positive probability (see \cite{FournierJourdain}), whilst in the fully parabolic case they seem to agglomerate (numerically), see \cite[Fig. 7.1, p. 120]{MilicaThesis}.

~
 
Finally, let us state the assumptions on the initial conditions of the system:

~

\noindent 
\begin{minipage}{0.05\linewidth}
(\customlabel{H}{$\Hyp$}):   
\end{minipage}
\begin{minipage}[t]{0.93\linewidth}
\raggedright
Fix $\rK$ from \eqref{HK3}. For any $m\geq1$, $\displaystyle\sup_{1\leq i\leq N,\, N\in \N} \EE|X_{0}^{i,N}|^m<\infty$ and
$$
 \sup_{N \in \N} \EE\left[  \left\|  \mu^N_{0} \ast V^N \right\|_{L^\rK(\R^d)}^m   \right] \ 
<\infty .
$$
\end{minipage}

\vspace{0.4cm}

\noindent A sufficient condition for \eqref{H} to hold is that particles are initially i.i.d. with a law which is in $L^\rK$ and $\alpha d<1$. 

~

We aim to prove the convergence of the mollified empirical measure to the PDE \eqref{eq:PDE}. As \eqref{eq:PDE} preserves the total mass $M:= \int_{\R^d} u_0 (x) \, dx$, we will assume throughout the paper that $M=1$, see Remark \ref{remark:propertiesEDP}. 

Solutions to \eqref{eq:PDE} will be understood in the following mild sense:
\begin{definition}\label{def:defMild}
Given $K$ satisfying \eqref{HK1}-\eqref{HK2}, $u_0 \in \Lsol$ with $\rK \geq \max (\pK', \qK')$ and $T>0$,
 a function $u$ on $[0,T] \times \R^d$ is said to be a mild solution to \eqref{eq:PDE} on $[0,T]$ if
\begin{enumerate}[label=(\roman*)]
\item $u\in \Ccal([0,T]; \Lsol) $; 
\item $u$ satisfies the integral equation
\begin{equation}
\label{eq:mildKS}
u_{t} =  e^{t\Delta} u_0 -  \int_0^t \nabla \cdot ( e^{(t-s)\Delta }  (u_{s}\,  K \ast u_{s}) )\, ds, \quad 0 \leq t \leq T.
\end{equation}
\end{enumerate}
A function $u$ on $[0,\infty) \times \R^d$ is said to be a global mild solution to \eqref{eq:PDE} if it is a mild solution to \eqref{eq:PDE} on $[0,T]$ for all $T>0$.
\end{definition}

In Section \ref{subsec:PDE} the following result about the (local) well-posedness of the PDE \eqref{eq:PDE} is established: 
\begin{proposition}\label{prop:local_existence}
There exists $T>0$ such that the PDE \eqref{eq:PDE} 
admits a 
mild solution $u$ in the sense of Definition \ref{def:defMild} on $[0,T]$. In addition, this mild solution is unique.
\end{proposition}

From here we can explicit a cut-off $A$ (for the function $F_{A}$ defined in \eqref{eq:defF0}) to define the particle system. 
 For a local mild solution $u$ on $[0,T]$, we will use the following cut-off: 
\begin{align}\label{eq:AT'}
A_{T}:= C_{K,d} \, \|u\|_{T,\Lsol}, 
\end{align}
where $C_{K,d}$ depends only on $K$ and $d$ and is given in Lemma \ref{lem:unifbound}. 
 
Finally, as the kernel $K$ may be singular, the PDEs we are interested in may blow up in finite time. For this reason, we will denote by $T_{max}$ the maximal time of existence of a solution to \eqref{eq:PDE} in the sense of Definition \ref{def:defMild}. This means that for any $T<T_{max}$, the PDE admits a mild solution on $[0,T]$. If there exists a global mild solution to our PDE, then $T_{max}=\infty$.

~

\subsection{First main result: Rate of convergence to the PDE}\label{subsec:MainResult1}

Our first main result is the following claim, whose proof is detailed in Section \ref{subsec:proofRate2}:
\begin{theorem}\label{th:rate2}  
Assume that the initial conditions $\{\mu^N_{0}\}_{N\in\N}$ satisfy \eqref{H} and that the kernel $K$ satisfies \eqref{HK}.  Moreover, let \eqref{C04} hold true. 
 Let $T_{\max}$ be the maximal existence time for \eqref{eq:PDE} and fix $T\in(0,T_{\max})$. 
In addition, let the dynamics of the particle system be given by \eqref{eq:IPS} with $A$ greater than $A_{T}$ defined in \eqref{eq:AT'}.

Then, for any $\varepsilon>0$ and any $m\geq 1$, there exists a constant $C>0$ such that for all $N\in\N^*$,
\begin{align*}
\left\| \| u^N-u \|_{T,\Lsol} \right\|_{L^m(\Omega)}
&\leq C \left\| \|u^N_0- u_0\|_{\Lsol} \right\|_{L^m(\Omega)}  + C N^{-\varrho + \varepsilon} ,
\end{align*}
where
\begin{align}\label{eq:def_rho}
\varrho  =  \min \left(\alpha \Hreg,\,\frac{1}{2}\left(1- \alpha (d+d(1-\frac{2}{\rK})\vee 0)\right)  \right) .
\end{align}
\end{theorem}

It is clear from the definition \eqref{eq:def_rho} of $\varrho$ that there is a trade-off between choosing $\alpha$ close to $\frac{1}{d}$ (as imposed by \eqref{C04}, provided $\rK\leq 2$) so as to have more physical particles; and choosing $\alpha$ smaller (such that $\alpha \Hreg = \frac{1}{2}\left(1- \alpha d \right)$, for a given $\Hreg$ and  assuming again that $\rK\leq 2$) so as to maximize the rate of convergence. The latter case could be of importance in numerical applications. We discuss the possible choices of the parameters in more details on various examples in Section \ref{sec:examples}.

 In Corollary~\ref{cor:rateEmpMeas}, we obtain the same rate for the genuine empirical measure of the system \eqref{eq:IPS0}:
\begin{equation*}
\bigg\| \sup_{t\in[0,T]}  \|\mu_{t}^N - u_{t} \|_{0} \bigg\|_{L^m(\Omega)}  \leq C  \left\|  \|u^N_0- u_0\|_{\Lsol} \right\|_{L^m(\Omega)} + C\, N^{-\varrho + \varepsilon},
\end{equation*}
where $ \|\cdot\|_{0}$ denotes the Kantorovich-Rubinstein metric (see Subsection \ref{subsec:CorRate}).

In the next result, we remove the cutoff $F_{A}$ from the drift of the particle system (i.e. we choose $F_{A} = \text{Id}$ in \eqref{eq:IPS0}, or equivalently $A=\infty$), and the convergence is now in probability.

\begin{corollary}\label{cor:wo-cutoff}
Let the same assumptions as in Theorem~\ref{th:rate2} hold. For each $N\in\N^*$, consider the solution $(X^{i,N})_{1\leq i \leq N}$ to \eqref{eq:IPS0} with $A=\infty$ (i.e. $F_{A} = \text{Id}$). We still denote by $\mu^N_{t}$ the empirical measure of this particle system and $u^N_{t} = \mu^N_{t}\ast V^N$.  
Let $\varrho$ be as in Theorem \ref{th:rate2}. 
Then for any $\varepsilon \in (0,\varrho)$, any $\eta>0$ and any $m\geq1$, there exists $C>0$ such that, for any $N\in \N^*$,
\begin{align*}
\PP \left(  \|u_{t}^N - u_{t} \|_{T,\Lsol} \geq \eta  \right) \leq \frac{C}{\eta^m} \left( \left\|  \|u^N_0- u_0\|_{\Lsol} \right\|_{L^m(\Omega)}  + C N^{-\varrho + \varepsilon}\right)^m .
\end{align*}
\end{corollary}
Note that for fixed $N\in \N^*$, the drift term in Equation \eqref{eq:IPS0} is smooth and bounded, hence the particle system with $F_{A} = \text{Id}$ has a unique strong solution. The proof of this result is presented in Subsection \ref{subsec:CorRate}.

\begin{remark}\label{rk:optimalrate}
Unlike particle systems in mean field interaction, we cannot expect here a $N^{-1/2}$ rate of convergence, as for regular interactions \cite{CattiauxGuillinMalrieu}. This is due to the short range of interaction of the particles which is of order $N^{-\alpha}$. At the macroscopic level, we also observe that the distance between a finite measure $\mu$ and its regularisation $V^N\ast \mu$ is of order $N^{-\alpha}$ too. Hence it is reasonable to expect no better than an $N^{-\alpha}$ rate of convergence. 

Assuming that the kernel $K$ is such 
that $\Hreg$ can be chosen equal or close to 1, the rate in \eqref{eq:def_rho} is  determined by the minimum between $\alpha$ and  $\frac{1}{2}\left(1-\alpha (d+ d(1-\frac{2}{\rK})\vee 0 )\right)$, so choosing $\alpha$ accordingly does lead to a rate of convergence of ``optimal'' order $N^{-\alpha}$.
\end{remark}

\begin{remark}
Finally, let us consider the application of Theorem~\ref{th:rate2} to several classes of models (the details are given in Section \ref{sec:examples}):

$\bullet$ For Coulomb-type kernels, which includes the Biot-Savart kernel in dimension $2$, the Riesz kernel with $s=d-2$ and the Keller-Segel kernel, 
the convergence happens for any $\alpha <\frac{1}{2(d-1)}$ (note that in dimension $2$, this accounts for choosing $\alpha = (\frac{1}{2})^-$); the best possible rate of convergence is $\varrho = \left(\frac{1}{2(d+1)}\right)^-$ for the choice $\alpha = \left(\frac{1}{2(d+1)}\right)^+$.

$\bullet$ The $2d$ Keller-Segel model (see kernel \eqref{eq:KS-kernel}). We recall that the PDE has a global solution whenever the critical parameter $\chi$ satisfies $\chi<8\pi$, and explodes in finite time otherwise (see~\cite{Nagai11}). 
In Theorem~\ref{th:rate2}, we get a rate for any value of $\chi$  which is almost $\frac{1}{2(d+1)}$. This result holds even if the PDE explodes in finite time ($\chi >8\pi$). In that case, one works on $[0,T]$ for any $T<T_{\text{max}}$. 

$\bullet$ The Riesz kernels with $s>d-2$ do not satisfy Assumption \eqref{HK3}, see Section \ref{sec:examples}. However, by imposing more regularity on the initial conditions and smaller values of $\alpha$ ($\alpha<\frac{1}{2d}$), we present a rate of convergence for singular Riesz kernels with $s\in (d-2,d-1)$, see Theorem \ref{th:rate2sing}.
\end{remark}

\paragraph{Related works.} 
The general question of quantifying the convergence of interacting particle systems (whether in mean field or moderate interaction) towards the PDE in non-singular framework has been addressed thoroughly in the literature. See for example \cite{Sznit, Oelschlager87,JourdainMeleard} or, more recently, \citet{Font} in the case of  homogeneous Boltzmann  equation for  Maxwell  molecules. 

However, in the singular case there are fewer results in the literature:  when the particle system is somehow regularised (in a wider sense than \eqref{eq:IPS0}, i.e. with moderate interactions or a more general cut-off in the kernel), 
 \citet{MeleardNS2d} obtained a rate on the density of one particle for the $2d$ Navier-Stokes equation, and \citet{BossyTalay} got a rate for Burger's equation. \\
   When the kernel is not regularised, let us mention that  \citet{FournierMischler} obtained a rate for the empirical measure of the so-called Nanbu particle system approaching the Boltzmann equation, and
   \citet{FournierHauray} approximated the Landau equation with moderately soft potentials. 
   Working at the level of the Liouville equations associated to the mean field particle system, \citet{JabinWang} obtained a quantitative convergence  with a $N^{-1/2}$ rate for some kernels including the Biot-Savart kernel. Recently, \citet{BJW2020} also proved a rate of convergence between the entropy solutions of the Liouville equations and the PDE for the tricky case of the $2d$ Keller-Segel model. 
 Let us finally mention the related problem of singular non-diffusive systems (i.e. with deterministic particles), which has seen significant progress recently with the work of \citet{Serfaty}, who introduced the modulated energy method (see also  \cite{Duerinckx} and \cite{Rose3}).
 
The present work is inspired by the new semigroup approach developed by \citet{FlandoliLeimbachOlivera}, which allows one to approximate nonlinear PDEs by smoothed empirical measures in strong functional topologies. More precisely, 
  the convergence of the mollified empirical measure of the moderately interacting particle system is obtained, and the approach was initially proposed for the FKPP equations. 
 It has already found many applications:   see \citet{FlandoliLeocata} for a PDE-ODE system related to aggregation phenomena;  \citet{Simon} for non-local  conservation laws;  \citet{FlandoliOliveraSimon} for the $2d$ Navier-Stokes equation; and \citet{ORT} for the parabolic-elliptic Keller-Segel systems.  However, in these recent works this convergence was not quantified and  the propagation of chaos was not considered.

\subsection{Second main result: Propagation of chaos}\label{subsec:MainResult2}

We now tackle the question of propagation of chaos of \eqref{eq:IPS0} towards \eqref{eq:McKeanVlasov}.

To ensure that \eqref{eq:McKeanVlasov} admits a unique weak solution, we will solve the associated nonlinear martingale problem. By classical arguments, one can then pass from a solution to this martingale problem to the existence of a weak solution to \eqref{eq:McKeanVlasov}. Hence, consider the following nonlinear martingale problem related to \eqref{eq:McKeanVlasov}:
\begin{definition}
\label{def:MP2-cutoff}
Given $K$ satisfying \eqref{HK1}-\eqref{HK2}, $u_0 \in \Lsol$ with $\rK \geq \max (\pK', \qK')$ and $T>0$, consider the canonical space $\mathcal{C}([0,T];\R^d)$ equipped with its canonical filtration. Let $\Q$ be a probability measure on the canonical space and denote by $\Q_t$ its one-dimensional time marginals. We say that $\Q$ solves the nonlinear 
martingale problem (\customlabel{MP}{$\mathcal{MP}$}) if:
\begin{enumerate}[label=(\roman*)]
\item $\Q_0 = u_0$;
\item For any $t\in (0,T]$, $\Q_t$ has a density $q_t$ w.r.t. Lebesgue measure on $\R^d$. In addition, it satisfies $q\in \Ccal([0,T];\Lsol)$;
\item For any $f \in \Ccal_{c}^2(\R^d)$, the process $(M_t)_{t \in[0, T]}$ 
defined as
$$M_t:=f(w_t)-f(w_0)-\int_0^t \Big[ \Delta f(w_s)+  \nabla f(w_s)\cdot (K \ast q_s (w_s))\Big] ds$$
is a $\Q$-martingale, where $(w_t)_{t \in[0, T]}$ denotes the canonical process.  
\end{enumerate}
\end{definition}

We will see that for $f\in \Lsol$, $K\ast f$ is bounded, hence $(M_t)_{t\geq 0}$ is well-defined. The following claim establishes the well-posedness of the martingale problem and is proven in Section~\ref{subsec:proofNMP}:
\begin{proposition}\label{prop:wpMP}
Let $T< T_{max}$. 
Then, the martingale problem \eqref{MP} admits a unique solution in the sense of Definition \ref{def:MP2-cutoff}. 
\end{proposition}

This result comes from the combination of the fact that the marginal laws of the process are uniquely determined as solutions of \eqref{eq:PDE} with the fact that the linear version of \eqref{eq:McKeanVlasov} admits a unique weak solution. 
  We choose this approach as we have \emph{a priori} information about the PDE.

By Proposition \ref{prop:wpMP}, McKean-Vlasov Equation \eqref{eq:McKeanVlasov} admits a unique weak solution, so let us consider $N$ independent copies $(\widetilde{X}^i)_{i\in\{1,\dots,N\}}$  and denote by $(W^i)_{i\in \{1,\dots,N\}}$ the associated Brownian motions. Consider the particles $(X^i)_{i\in\{1,\dots,N\}}$ defined as the (strong) solutions of \eqref{eq:IPS} driven by the same Brownian motions $(W^i)_{i\in \{1,\dots,N\}}$ and with initial conditions $(\widetilde{X}^i_{0})_{i\in\{1,\dots,N\}}$. 
 A simple consequence of Theorem \ref{th:rate2} is that the particles $(X^i)_{i\in\{1,\dots,N\}}$ (with mollified interaction kernel and cut-off $F_{A}$) are uniformly close to the non-regularised McKean-Vlasov particles.  More precisely,
\begin{equation}\label{eq:gapMV}
 \left\| \max_{i\in \{1,\dots, N\}} \sup_{t\in [0,T]} |X^{i,N}_{t} - \widetilde{X}^{i}_{t}| \right\|_{L^m(\Omega)} \leq C \,  N^{-\varrho+\varepsilon}, \quad \forall N\in \N^*.
\end{equation}

~

~

Now, we present our second main result. It will be proven in Section \ref{subsec:proofPropChaos}.
\begin{theorem}
\label{th:propagation-of-chaos}
Let the hypotheses of Theorem \ref{th:rate2} hold. In particular, recall that ${u\in \mathcal{C}([0,T],\Lsol)}$. Assume further that the family of random variables $\{X^i_{0},~i\in\N\}$ is identically distributed and that $\langle u^N_0, \varphi\rangle  \to \langle u_{0}, \varphi\rangle $ in probability, for any $\varphi \in \mathcal{C}_{b}(\R^d)$. Then, the empirical measure $\mu^N_.$ (defined in \eqref{eq:empiricalMeasure} as a probability measure on $\mathcal{C}([0,T];\R^d)$)  converges in probability towards $\Q$, which is the law of the  unique weak solution of \eqref{eq:McKeanVlasov}. 
\end{theorem}

We emphasize here that without the convergence of $u^N$ in the convenient functional framework, it would not be possible to obtain the propagation of chaos in this singular setting. Hence, the result of Theorem \ref{th:rate2} is very much related to the propagation of chaos and should be considered as the most important ingredient when proving  Theorem \ref{th:propagation-of-chaos}.

\begin{remark}
 As in the previous section, the case of the singular Riesz kernels is not covered by Proposition \ref{prop:wpMP} and Theorem \ref{th:propagation-of-chaos}. However, we are able to give very similar statements of well-posedness of the McKean-Vlasov equations for Riesz kernels with $s\in (d-2,d-1)$ in Section~\ref{subsec:chaossing}, provided that the initial conditions are smoother and $\alpha<\frac{1}{2d}$.
\end{remark}

\paragraph{Related works.} Recently, the well-posedness of McKean-Vlasov SDEs (or distribution-dependent SDEs) has gained much attention in the literature (see e.g.  \cite{Chaudru,deraynal2019, BauerEtAl,lacker2018, mishura2020, rockner2019} and the references therein). The authors analyse well-posedness when the diffusion coefficient is also distribution-dependent and when the dependence on the law is not necessarily as in~\eqref{eq:McKeanVlasov}. The main difficulties there are to treat coefficients that may not be continuous (in the measure variable) w.r.t the Wasserstein distance and eventually to treat singular drift coefficients (in the space variable). As both of these difficulties appear in  the specific distribution dependence in  \eqref{eq:McKeanVlasov}, our well-posedness result gives a new perspective on the matter as in the above works  one cannot treat singular interactions in the drift term with kernels $K$ of the order $\frac{1}{|x|^{s+1}}$ for $s\in(0,d-1)$, in any dimension.

Let us focus on the Keller-Segel model, for which the kernel is given in \eqref{eq:KS-kernel}. 
In  \citet{FournierJourdain}, the well-posedness of the associated McKean-Vlasov SDE for a value of the sensitivity parameter $\chi<2\pi$ in dimension $2$ is proven. The authors also proved tightness and consistency result for the associated particle system (one cannot properly speak about propagation of chaos as the PDE might not have a unique solution in their functional framework). Our result provides, assuming more regularity on the initial condition, global (in time) well-posedness of the McKean-Vlasov SDE whenever the PDE has a global solution, and local well-posedness whenever the solution of the PDE exhibits a blow-up. 
In particular in dimension $2$, we get the global well-posedness of \eqref{eq:McKeanVlasov} whenever $\chi<8\pi$, and local well-posedness of \eqref{eq:McKeanVlasov} when $\chi \geq 8\pi$. In both cases, we obtain the propagation of chaos for the moderately interacting particle system.

\subsection{Organisation of the paper}\label{subsec:orga}
 
In Section~\ref{sec:proofRate2}, we prove Theorem \ref{th:rate2} and its corollaries. In Section \ref{sec:chaos}, the existence and uniqueness of the solution of the
martingale problem \eqref{MP} is proven, 
as well as the propagation of chaos for the empirical measure of \eqref{eq:IPS0} 
(Theorem \ref{th:propagation-of-chaos}). In Section~\ref{sec:sing} we extend the results of Sections~\ref{sec:proofRate2} and~\ref{sec:chaos}.
Finally, we present some examples and applications of our results in Section~\ref{sec:examples}. In the Appendix,  the existence and uniqueness of the Fokker-Planck equation \eqref{eq:PDE} and its cut-off version are studied in Section \ref{sec:PDE}. Then, one may find the time and space estimates for some stochastic convolution integrals in Section \ref{app:martingale}, and the proof of a result about the boundedness of the mollified empirical measure in $L^2 \left([0,T]; H^\beta_{\rK}(\R^d)\right)$ in Section \ref{app:tightness}.

\paragraph{Acknowledgment.} The authors are very grateful to an anonymous referee for his careful reading of the manuscript, and in particular for suggesting us the result concerning the particle system without cutoff (Corollary \ref{cor:wo-cutoff}).\\
C.O. is partially supported by CNPq through the grant 426747/2018-6 and 
FAPESP by the grants 2020/15691-2 and 2020/04426-6. C.O. and A.R. acknowledge the support of the Math-Amsud program, through the grant 19-MATH-06. M.T. acknowledges the support of FMJH.

\section{Rate of convergence}\label{sec:proofRate2}

In this section, we prove our first main result (Theorem \ref{th:rate2}). More precisely, in Subsection \ref{subsec:ItogN} we establish the equation satisfied by $u^N$ and present the boundedness result for it that is proved in the appendix.
 Then in Subsection \ref{subsec:proofRate2}, we prove Theorem \ref{th:rate2}. In Subsection \ref{subsec:CorRate}, we present the corollaries of Theorem \ref{th:rate2}.

First, we give a  preliminary lemma that is essential for our calculations and for the analysis of the mild form of \eqref{eq:PDE} and the mild form of the cut-off PDE \eqref{eq:cutoffPDE}.
\begin{lemma}\label{lem:unifbound}
Let $K$ be satisfying Assumptions \eqref{HK1} and \eqref{HK2} and recall that $\rK \geq \max (\pK', \qK')$. There exists $C_{K,d}>0$ (which depends on $K$ and $d$ only) such that for any  $f \in \Lsol$, 
\begin{equation*}
 \|K \ast f\|_{L^\infty(\R^d)} \leq C_{K,d} \, \|f\|_{\Lsol}.
\end{equation*}
\end{lemma}
\begin{proof}
Recall that $\qK'$ denotes the conjugate exponent of the paremeter $\qK$ from \eqref{HK2}. In view of Assumptions \eqref{HK1} and \eqref{HK2},  H\"older's inequality yields
\begin{align*}
|K\ast f(x)| &\leq \int_{\mathcal{B}_{1}} |K(y)|  |f(x-y)| ~dy  +  \int_{\mathcal{B}_{1}^c} |K(y)|  |f(x-y)| ~dy \nonumber\\
&\leq \|K\|_{L^\pK(\mathcal{B}_{1})} \|f\|_{L^{\pK'}(\R^d)} + \|K\|_{L^\qK(\mathcal{B}_{1}^c)} \|f\|_{L^{\qK'}(\R^d)}. 
\end{align*}
The conclusion follows from the interpolation inequality $ \|f\|_{L^{\qK'}(\R^d)} \lesssim \|f\|_{\Lsol}$.
\end{proof}

~

We recall here some classical properties of the heat kernel that will be used throughout this section.
Applying the  convolution inequality \cite[Th. 4.15]{Brezis} for $p\geq 1$ and using the equality \newline
$\big\Vert \nabla \frac
{1}{   (4\pi t)^{\frac{d}{2}}  }e^{- \frac{\vert \cdot\vert ^{2}}{4t}}\big\Vert _{L^{1}(\R^d)}=\frac{C}{\sqrt{t}}$, it comes that
\begin{equation}
\label{eq:heat-op-norm}
\big\Vert \nabla e^{ t \Delta}\big\Vert _{L^{p}
\rightarrow L^{p}  }\leq \frac{C}{\sqrt{t}}.
\end{equation}
By explicit computations in the Fourier space, we get that 
$$\|g_{2t} \|_{\beta,1} = \big\Vert (I-\Delta)^\frac{\beta}{2} \frac
{1}{   (4\pi t)^{\frac{d}{2}}  }e^{- \frac{\vert \cdot\vert ^{2}}{4t}}\big\Vert _{L^{1}(\R^d)}\leq C\, t^{-\frac{\beta}{2}}, $$
 hence the inequality \eqref{eq:heat-op-norm} extends to 
\begin{equation}
\label{eq:heat-op-norm-frac}
\big\Vert (I-\Delta)^{\frac{\beta}{2}} e^{ t \Delta}\big\Vert _{L^{p}
\rightarrow L^{p}  }\leq C\, t^{-\frac{\beta}{2}}.
\end{equation}

\subsection{Properties of the regularised empirical measure}\label{subsec:ItogN}

The proof of the rate of convergence relies on the following mild formulation of the mollified empirical measure:
\begin{equation}
\label{eq:mildeq}
\begin{split}
u^N_t(x)=e^{t\Delta }u^N_0(x) -  \int_0^t \nabla \cdot e^{(t-s)\Delta } \langle \mu_s^N,  V^N ( x-\cdot) F\big(  K\ast u^N_s (\cdot)\big) \rangle \ ds \\
- \frac{1}{N} \sum_{i=1}^N \int_0^t  e^{(t-s)\Delta} \nabla V^N (x-X_s^{i,N})\cdot dW^i_s, 
\end{split}
\end{equation}
and the following property: For $q \geq 1$, one has 
\begin{equation}\label{eq:boundednessuN}
\sup_{N\in\N^*} \mathbb{E}\left[ \sup_{t\in[0,T]} \left\Vert  u^N_{t} \right\Vert _{L^\rK(\R^d)}^{q}\right]  <\infty.
\end{equation}
This estimate will be proven in Proposition \ref{prop:bound1} (that we apply here for $\beta=0$).

Now, let us establish Equation \eqref{eq:mildeq}. 
Consider the mollified empirical measure 
\[u^N_t= V^{N} \ast \mu_t^{N}: x \in \R^d \mapsto \int_{\R^d} V^N(x-y) d\mu_t^N(y)=\frac{1}{N}\sum_{k=1}^{N} V^N(x-X_{t}^{k,N}).\] 
Using this definition, we rewrite the particle system in  \eqref{eq:IPS} as
\begin{equation*}
dX_{t}^{i,N}=  F\big(  K\ast u^N_t (X_{t}^{i,N})\big)\; dt + \sqrt{2} \; dW_{t}^{i}, \quad t\in[0, T], \quad 1\leq i \leq N. \label{PS2}
\end{equation*}
Fix $x\in \R^d$ and $1\leq i\leq N$. Apply It\^o's formula to the function $V^N(x-\cdot)$ and the particle $X^{i,N}$. Then, sum for all $1\leq i\leq N$ and divide by $N$. It comes
\begin{equation}
\label{eq:almost-mild}
\begin{split}
u^N_t(x)=&u^N_0(x) - \frac{1}{N} \sum_{i=1}^N \int_0^t \nabla V^N (x- X_s^{i,N}) \cdot F\big(  K\ast u^N_s (X_{s}^{i,N})\big)\ ds \\
&- \frac{1}{N} \sum_{i=1}^N \int_0^t \nabla V^N (x- X_s^{i,N})\cdot dW^i_s +  \frac{1}{N} \sum_{i=1}^N \int_0^t \Delta V^N (x- X_s^{i,N}) \ ds.
\end{split}
\end{equation}
Notice that
$$ \frac{1}{N} \sum_{i=1}^N \int_0^t \nabla V^N (x- X_s^{i,N}) \cdot F\big(  K\ast u^N_s (X_{s}^{i,N})\big)\ ds =\int_0^t \langle \mu_s^N,  \nabla V^N (x- \cdot) \cdot F\big(  K\ast u^N_s (\cdot)\big) \rangle \ ds$$
and 
$$ \frac{1}{N} \sum_{i=1}^N \int_0^t \Delta V^N (x- X_s^{i,N}) \ ds= \int_0^t \Delta u^N_s(x)\ ds.$$
The preceding equalities combined with \eqref{eq:almost-mild} lead to 
\begin{equation}
\label{eq:almost-mild2}
\begin{split}
u^N_t(x) &= u^N_0(x) - \int_0^t \langle \mu_s^N, \nabla V^N (x-\cdot) \cdot F\big(  K\ast u^N_s (\cdot)\big) \rangle \ ds \\
&\quad - \frac{1}{N} \sum_{i=1}^N \int_0^t  \nabla V^N (x-X_s^{i,N})\cdot dW^i_s + \int_0^t \Delta u^N_s(x)\ ds.
\end{split}
\end{equation}

The following mild form is immediately derived:
\begin{equation}
\label{eq:mild}
\begin{split}
u^N_t(x)=e^{t\Delta }u^N_0(x) - \int_0^t e^{(t-s)\Delta } \langle \mu_s^N, \nabla V^N (x-\cdot) \cdot F\big(  K\ast u^N_s(\cdot)\big) \rangle \ ds \\- \frac{1}{N} \sum_{i=1}^N \int_0^t  e^{(t-s)\Delta} \nabla V^N (x-X_s^{i,N})\cdot dW^i_s.
\end{split}
\end{equation}
Finally, developing the scalar product, one has 
$$ \langle \mu_s^N, \nabla V^N (x-\cdot) \cdot F\big(  K\ast u^N_s(\cdot)\big) \rangle=  \nabla_x \cdot  \langle \mu_s^N, V^N (x-\cdot) F\big(  K\ast u^N_s(\cdot)\big) \rangle. $$
Combining the latter with the fact that $e^{t\Delta } \nabla \cdot f=\nabla \cdot e^{t\Delta} f$, we deduce the mild form \eqref{eq:mildeq} from \eqref{eq:mild}.

\subsection{Proof of Theorem \ref{th:rate2}}\label{subsec:proofRate2}

Let $u$ the unique mild solution to \eqref{eq:PDE} on $[0, T]$. By assumption $A$ is large enough, so  $u$ also satisfies Equation~\eqref{eq:mildKS-cutoff}.

In view of the mild formulas \eqref{eq:mildeq} for $u^N$ and \eqref{eq:mildKS-cutoff} for $u$, it comes
\begin{align*}
u^N_{t}(x) - u_{t}(x) &= e^{t\Delta} (u^N_{0} - u_{0})(x) - \int_{0}^t \nabla \cdot e^{(t-s)\Delta } \left( \langle \mu_s^N,  V^N (x-\cdot) F\big(  K\ast u^N_s (\cdot)\big) \rangle - u_{s}  F(K\ast u_s) (x) \right) \, ds\\
&\quad - \frac{1}{N} \sum_{i=1}^N \int_0^t  e^{(t-s)\Delta} \nabla V^N (x-X_s^{i,N})\cdot dW^i_s  .
\end{align*}
Now we subtract and add the term $\int_{0}^t \nabla\cdot e^{(t-s)\Delta}   \langle \mu^N_{s},  V^N (x-\cdot) F\big( K\ast u^N_s(x)\big)\rangle \ ds$, which is also equal to $\int_{0}^t \nabla\cdot e^{(t-s)\Delta} \,  u^N_{s}(x)\, F\big( K\ast u^N_s(x)\big) \ ds$. Hence we get
\begin{align*}
u^N_{t}(x) - u_{t}(x) &= e^{t\Delta} (u^N_{0} - u_{0})(x) +  \int_0^t  \nabla \cdot e^{(t-s)\Delta } \left(u_{s}  F(K\ast u_s) - u^{N}_{s}  F(K\ast u^{N}_s)\right)(x) ~ ds \\
&\quad + E_{t}(x) - M^N_{t}(x), 
\end{align*}
where we have set
\begin{align}\label{eq:defM}
E_{t}(x)&:= \int_0^t \nabla \cdot e^{(t-s)\Delta } \langle \mu^N_{s},  V^N (x-\cdot) \left( F\big( K\ast u^N_s(x)\big)-F\big( K\ast u^N_s(\cdot)\big)\right)\rangle \ ds,\nonumber\\
M^N_{t}(x) &:=\frac{1}{N} \sum_{i=1}^N \int_0^t  e^{(t-s)\Delta} \nabla V^N (x-X_s^{i,N})\cdot dW^i_s. 
\end{align}

For $p\in\{1, \rK\}$, in view of the estimate \eqref{eq:heat-op-norm}, one has
\begin{equation*}
\begin{split}
\| u^N_t-u_{t} \|_{L^{p}(\R^{d})} &\leq \|e^{t\Delta }(u^N_0- u_0 )\|_{L^{p}(\R^{d})} \\
&\quad + C \int_0^t \frac{1}{\sqrt{t-s}} \|(u_{s}  F(K\ast u_s) - u^{N}_{s}  F(K\ast u^{N}_s) ) \|_{L^{p}(\R^{d})} ds\\
&\quad + \| E_{t}\|_{L^{p}(\R^{d})}  + \|  M^{{N}}_{t}\|_{L^{p}(\R^{d})}. 
\end{split}
\end{equation*}
We will use the bound $\|e^{t\Delta }(u^N_0- u_0 )\|_{L^{p}(\R^{d})}\leq \|u^N_0- u_0\|_{L^{p}(\R^{d})}$. 
Observe also that
$$\|(u_{s}  F(K\ast u_s) - u^{N}_{s}  F(K\ast u^{N}_s) ) \|_{L^{p}(\R^{d})} \leq \|(u_{s} ( F(K\ast u_s) - F(K\ast u^{N}_s) ) \|_{L^{p}(\R^{d})} + \|( F(K\ast u_s^N)(u_{s}  - u^{N}_{s} ) ) \|_{L^{p}(\R^{d})}.$$
Hence, by the Lipschitz property of $F$ on one term and the boundedness of $F$ for the other,  it follows that 
\begin{equation*}
\begin{split}
\| u^N_t-u_{t} \|_{L^{p}(\R^{d})} &\leq \|u^N_0- u_0\|_{L^{p}(\R^{d})} + C 
\int_0^t \frac{1}{\sqrt{t-s}} \|F\|_{L^\infty(\R^d)} \| u^{N}_{s}-u_{s} \|_{L^{p}(\R^{d})} ds\\
&\quad +  C \int_{0}^t \frac{1}{\sqrt{t-s}} \|u_{s}\|_{L^p} \|F\|_{\text{Lip}} \|K\ast (u_s - u^{N}_{s}) \|_{L^{\infty}(\R^{d})} ds  \\
&\quad + \| E_{t}\|_{L^{p}(\R^{d})}  + \|  M^{{N}}_{t}\|_{L^{p}(\R^{d})} .
\end{split}
\end{equation*}
Besides, by  Proposition \ref{regularity} for $\beta= 0$, it comes that for some $C>0$ which depends on $\|u\|_{T,L^1\cap L^\rK(\R^d)}$, $\|F\|_{L^\infty(\R^d)}$ and $\|F\|_{\text{Lip}}$,
\begin{equation*}
\begin{split}
\| u^N_t-u_{t} \|_{L^{p}(\R^{d})} &\leq \| u^N_0- u_0 \|_{L^{p}(\R^{d})} + C \int_0^t \frac{1}{\sqrt{t-s}}  \| u^{N}_{s}-u_{s} \|_{L^{p}(\R^{d})} ds\\
&\quad +  C \int_{0}^t \frac{1}{\sqrt{t-s}} \|K\ast (u_s - u^{N}_{s}) \|_{L^{\infty}(\R^{d})} ds  \\
&\quad + \| E_{t}\|_{L^{p}(\R^{d})}  + \|  M^{{N}}_{t}\|_{L^{p}(\R^{d})} .
\end{split}
\end{equation*}
Finally we apply Lemma \ref{lem:unifbound} and obtain
\begin{equation}\label{lp}
\begin{split}
\| u^N_t-u_{t} \|_{L^{p}(\R^{d})} &\leq \| u^N_0- u_0 \|_{L^{p}(\R^{d})} + C
 \int_0^t \frac{1}{\sqrt{t-s}}  \| u^{N}_{s}-u_{s} \|_{L^{p}(\R^{d})} ds\\
&\quad +  C \int_{0}^t \frac{1}{\sqrt{t-s}} \|u_s - u^{N}_{s} \|_{\Lsol} ds  \\
&\quad + \| E_{t}\|_{L^{p}(\R^{d})}  + \|  M^{{N}}_{t}\|_{L^{p}(\R^{d})} .
\end{split}
\end{equation}

Therefore, considering \eqref{lp} for both $p=1$ and $p=\rK$, we deduce that
\begin{equation*}
\begin{split}
\| u^N_t-u_{t} \|_{\Lsol} &\leq \|u^N_0- u_0 \|_{\Lsol} + C \int_0^t \frac{1}{\sqrt{t-s}}  \| u^{N}_{s}-u_{s} \|_{\Lsol} ds\\
&\quad + \| E_{t}\|_{\Lsol}  + \|  M^{{N}}_{t}\|_{\Lsol} .
\end{split}
\end{equation*}
Using the Gr\"onwall lemma for convolution integrals (see e.g. \cite[Lemma 7.1.1]{Henry}), we obtain
\begin{equation}\label{ineu}
\| u^N-u \|_{t,\Lsol} \leq C \left(\|u^N_0- u_0 \|_{\Lsol} + \| E\|_{t,\Lsol}  + \|  M^{{N}}\|_{t,\Lsol} \right) .
\end{equation}

$\bullet$ \emph{Now, we estimate the moments of $\| E\|_{t, L^{q}(\R^{d}) }$ for $1\leq q\leq \rK$}. Using \eqref{eq:heat-op-norm} for $p=q$, we observe that 
\begin{multline*}
\| E_{t}\|_{L^{q}(\R^{d}) } \leq C 
\int_0^{t}   \frac{1}{(t-s)^{1/2}} \left(\int_{\R^{d}} \left|\langle \mu^N_{s},  V^N (x-\cdot) (F\big( K\ast u^N_s(\cdot)\big)- F\big( K\ast u^N_s(x)\big))\rangle \right|^{q} dx\right)^{\frac{1}{q}} \ ds.
\end{multline*}
By using the positivity of $V^N$, we get
\begin{align*}
\| E_{t}\|_{L^{q}(\R^{d}) } 
&\leq C 
\int_0^{t}   \frac{1}{(t-s)^{\frac{1}{2}}} \left(\int_{\R^{d}} \langle \mu^N_{s},  V^N (x-\cdot) \left|F\big( K\ast u^N_s(\cdot)\big)- F\big( K\ast u^N_s(x)\big)\right|\rangle^{q} \, dx\right)^{\frac{1}{q}} \, ds.
\end{align*}
Using the Lipschitz regularity of $F$ and the $\Hreg$-H\"older continuity of $K\ast u^N$ from Assumption \eqref{HK3}, we get
\begin{align*}
\| E_{t}\|_{L^{q}(\R^{d}) } &\leq C 
\int_0^{t}   \frac{\| u_s^{N}\|_{\Lsol}}{(t-s)^{\frac{1}{2}}} \left(\int_{\R^{d}} \langle \mu^N_{s},  V^N (x-\cdot) \left|\cdot-x\right|^{\Hreg}\rangle^{q} \, dx\right)^{\frac{1}{q}} \, ds.
\end{align*}
Since $V$ is compactly supported (without loss of generality, assume that the support of $V$ is included in the unit ball), we have that $V^N (x-y) \left|y-x\right|^{\Hreg} \leq N^{-\alpha\Hreg} V^N (x-y)$. Thus
\begin{align*}
\| E_{t}\|_{L^{q}(\R^{d}) }
&\leq  \frac{C}{N^{\alpha\Hreg}} 
\int_0^{t}  \frac{1}{(t-s)^{\frac{1}{2}}} \| u_s^{N}\|_{\Lsol} \| u_s^{N}\|_{L^{q}(\R^{d})} \, ds \\
&\leq \frac{C}{N^{\alpha\Hreg}} 
\int_0^{t}  \frac{1}{(t-s)^{\frac{1}{2}}} \| u_s^{N}\|_{\Lsol}^2 \, ds,
\end{align*}
where the last inequality holds by interpolation.
Apply H\"older's inequality with $p=\frac{3}{2}$ to obtain
\begin{align*}
\| E_{t}\|_{L^{q}(\R^{d}) } \leq  \frac{C}{N^{\alpha\Hreg}} \left(\int_{0}^t (t-s)^{-\frac{3}{4}}\, ds\right)^{\frac{2}{3}}  \left(\int_{0}^t  \| u_s^{N}\|_{\Lsol}^6\, ds\right)^\frac{1}{3} .
\end{align*}
Finally, we have from Jensen's inequality that for $m\geq 3$,
\begin{align*}
\left\| \| E\|_{t,L^{q}(\R^{d})} \right\|_{L^m(\Omega)} \leq \frac{C}{N^{\alpha\Hreg}} \left(\int_{0}^t \EE\left[ \| u_s^{N}\|_{\Lsol}^{2m} \right] \, ds\right)^{\frac{1}{m}}
\end{align*}
and the bound \eqref{eq:boundednessuN} permits to conclude that
\begin{align}\label{eq:errorL1}
\left\| \| E\|_{t, L^{q}(\R^{d})} \right\|_{L^m(\Omega)} \leq   \frac{C}{N^{\alpha\Hreg}} .
\end{align} 
This inequality immediately extends to $1\leq m<3$.

~

$\bullet$ \emph{We turn to the moments of $\|M^N_{t}\|_{L^{1}(\R^d)}$.} Note that $ \left(\|M^N_{s}\|_{L^{1}(\R^d)}\right)_{s\geq 0}$ is not a martingale, but a stochastic convolution integral. We explain how to deal with this term in Appendix~\ref{app:martingale}.
By Proposition \ref{prop:martingale-bound}, we have: for any $\varepsilon>0$ arbitrary small, there exists $C>0$ such that
\begin{equation}\label{Eq}
\left\| \sup_{s\in [0,t] } \|M^N_{s}\|_{L^{1}(\R^{d})} \right\|_{L^m(\Omega)}  \leq C\,  N^{-\frac{1}{2}\left(1-\alpha d\right) + \varepsilon} , \quad \forall N\in\N^* .
\end{equation}

~

$\bullet$ \emph{It remains to estimate the moments of $\| M_{t}^N\|_{L^\rK(\R^d)}$.} By Proposition \ref{prop:martingale-bound} for $p=\rK$, we have: for any $\varepsilon>0$, there exists $C>0$ such that
\begin{equation}\label{Einf}
\left\| \sup_{s\in[0,t]} \|M^N_{s}\|_{L^{\rK}(\R^{d})} \right\|_{L^m(\Omega)}  \leq C\;  N^{- \frac{1}{2}\left(1-\alpha (d+ \varkappa_{\rK}) \right) + \varepsilon} ,
\end{equation}
where $\varkappa_{\rK} =  d(1-\tfrac{2}{\rK})\vee 0$. By the assumption \eqref{C04}, the above exponent in $N$ is indeed negative.

~

$\bullet$ \emph{Conclusion}. ~ 
Plugging Inequalities \eqref{eq:errorL1}-\eqref{Einf} 
in \eqref{ineu}, we conclude that 
for any $\varepsilon>0$ small enough, there exists $C>0$ such that for any $N\in\N^*$,
\begin{align*}
\left\| \| u^N-u \|_{t,\Lsol} \right\|_{L^m(\Omega)}
&\leq C \left( \left\|  \|u^N_0- u_0\|_{\Lsol} \right\|_{L^m(\Omega)}   +  N^{-\alpha \Hreg}   +  N^{- \frac{1}{2}\left(1-\alpha (d+ \varkappa_{\rK}) \right) + \varepsilon} \right). 
\end{align*}

\subsection{Corollaries of Theorem \ref{th:rate2}}\label{subsec:CorRate}
In view of the previous result, we obtain a rate of convergence for the genuine empirical measure, which can be interpreted as propagation of chaos for the marginals of the empirical measure of the particle system. Following \cite[Section 8.3]{BogachevII}, let us introduce the Kantorovich-Rubinstein metric which reads, for any two probability measures $\mu$ and $\nu$ on $\R^d$,
\begin{equation}\label{eq:defWasserstein}
\|\mu - \nu \|_{0} = \sup \left\{ \int_{\R^d} \phi \, d(\mu-\nu) \, ; ~ \phi \text{ Lipschitz  with } \|\phi\|_{L^\infty(\R^d)}\leq 1 \text{ and } \|\phi\|_{\text{Lip}} \leq 1 \right\} .
\end{equation}
Note that this distance metrizes the weak convergence of probability measures (\cite[Theorem 8.3.2]{BogachevII}).
\begin{corollary}\label{cor:rateEmpMeas}
Let the same assumptions as in Theorem~\ref{th:rate2} hold. Let $\varrho$ be as in Theorem \ref{th:rate2}. 
Then for any $\varepsilon \in (0,\varrho)$, there exists $C>0$ such that, for any $N\in \N^*$,
\begin{align*}
\bigg\| \sup_{t\in[0,T]}  \|\mu_{t}^N - u_{t} \|_{0} \bigg\|_{L^m(\Omega)}  \leq 
 C \left(  \left\|  \|u^N_0- u_0\|_{\Lsol} \right\|_{L^m(\Omega)} + \, N^{-\varrho + \varepsilon} \right) .
\end{align*}
\end{corollary}

\begin{proof}
Let $t\in (0,T_{max})$. Let us observe first that 
there exists $C>0$ such that for any Lipschitz continuous function $\phi$ on $\R^d$, one has
 \begin{align}\label{eq:rateunmun}
\left|\langle u^N_{t},\phi\rangle - \langle\mu^N_{t},\phi\rangle \right| \leq \frac{C \|\phi\|_{\text{Lip}}}{N^\alpha} \quad a.s.
 \end{align}
 Indeed, 
\begin{align*}
|\langle \mu_{t}^{N}, \phi\rangle - \langle u^N_t, \phi \rangle| &=| \langle \mu_{t}^{N}, (\phi-\phi\ast V^{N})\rangle |\\
&\leq \left\langle \mu_{t}^{N}, \int_{\R^{d}} V(y)~ |\phi(.)-  \phi( \frac{y}{N^{\alpha}}-.) |   dy \right\rangle \\
&\leq \frac{C \|\phi\|_{\text{Lip}}}{N^{\alpha}}.
\end{align*}

Recalling the definition \eqref{eq:defWasserstein} of the Kantorovich-Rubinstein distance, it comes
\begin{align*}
\left\| \sup_{t\in[0,T]} \| \mu_{t}^N - u_{t} \|_{0}  \right\|_{L^m(\Omega)} &\leq \left\| \sup_{t\in[0,T]} \| \mu_{t}^N - u^N_{t}\|_{0} \right\|_{L^m(\Omega)} +  \left\| \sup_{t\in[0,T]} \sup_{ \|\phi\|_{L^\infty}\leq 1} \langle u_{t}^N - u_{t},\phi \rangle \right\|_{L^m(\Omega)} \\
&\leq \left\| \sup_{t\in[0,T]} \| \mu_{t}^N - u^N_{t}\|_{0} \right\|_{L^m(\Omega)} +  \left\| \sup_{t\in[0,T]}  \| u_{t}^N - u_{t}\|_{L^1(\R^d)} \right\|_{L^m(\Omega)} .
\end{align*}
Now applying Inequality \eqref{eq:rateunmun} to the first term on the right-hand side of the above inequality, and Theorem \ref{th:rate2} to the second term, we obtain the inequality of Corollary \ref{cor:rateEmpMeas}.
\end{proof}

We now carry out the proof of Corollary \ref{cor:wo-cutoff}, for the particle system without cutoff.
\begin{proof}[Proof of Corollary \ref{cor:wo-cutoff}]
Let us introduce some notations to distinguish the particle systems with and without cutoff. 
Recall that $A_{T}$ is given in \eqref{eq:AT'}, and fix $A>A_{T}$ to be precisely chosen later. The particles which solve \eqref{eq:IPS0} with cutoff $A$ will be denoted here by $(X^{i,N,(A)})_{1\leq i\leq N}$ and those which solve \eqref{eq:IPS0} with $A=\infty$ will be denoted by $(X^{i,N,(\infty)})_{1\leq i\leq N}$. Accordingly, we denote by $\mu^{N,(A)}$ the empirical measure of $(X^{i,N,(A)})_{1\leq i\leq N}$, $\mu^{N,(\infty)}$ the empirical measure of $(X^{i,N,(\infty)})_{1\leq i\leq N}$, as well as $u^{N,(A)} = \mu^{N,(A)}\ast V^N$ and $u^{N,(\infty)} = \mu^{N,(\infty)}\ast V^N$. 
We assume these particles have the same initial conditions and are driven by the same family of independent $\R^d$-valued  Brownian motions $(W^i)_{i\in \N^*}$. 
For any $N\in \N^*$, define
\begin{align*}
\Omega_{N} = \Bigg\{ \sup_{\substack{t\in [0,T] \\ i\in \{1,\dots,N\}}} \frac{1}{N}\bigg| \sum_{k=1}^{N} (K \ast V^{N})(X_{t}^{i,N,(A)} - X_{t}^{k,N,(A)})\bigg| \leq A\Bigg\} 
\end{align*}
and observe that on $\Omega_{N}$, we have $X^{i,N,(A)}_{t} = X^{i,N,(\infty)}_{t}$ for all $t\in [0,T]$ and all $i\in \{1,\dots,N\}$.  Since $\frac{1}{N} \sum_{k=1}^{N} (K \ast V^{N})(X_{t}^{i,N,(A)} - X_{t}^{k,N,(A)}) = K \ast u^{N,(A)}_{t}(X^{i,N,(A)}_{t})$, we also get that on $\Omega_{N}$, 
$u^{N,(A)}_{t} = u^{N,(\infty)}_{t}$ for all $t\in [0,T]$ and all $i\in \{1,\dots,N\}$. Hence
\begin{align*}
\PP \left(  \|u_{t}^{N,(\infty)} - u_{t} \|_{T,\Lsol} \geq \eta \right) &= \PP \left(\Omega_{N}^c \cap \{ \|u_{t}^{N,(\infty)} - u_{t} \|_{T,\Lsol} \geq \eta \}\right) \\
&\hspace{1cm}+\PP \left(\Omega_{N} \cap \{ \|u_{t}^{N,(\infty)} - u_{t} \|_{T,\Lsol} \geq \eta \}\right) \\
&\leq \PP \left(\Omega_{N}^c \right) +\PP \left( \| u_{t}^{N,(A)} - u_{t} \|_{T,\Lsol} \geq \eta \right) .
\end{align*}
Now by Lemma \ref{lem:unifbound} we have that $|K \ast u^{N,(A)}_{t}(X^{i,N,(A)}_{t})| \leq C_{K,d} \|u^{N,(A)}_{t}\|_{\Lsol}$. Thus we get that for $A= C_{K,d} \left(\eta + \|u\|_{T,\Lsol}\right)$,
\begin{align*}
\PP \left(\Omega_{N}^c \right) &\leq \PP \left(\|u^{N,(A)}\|_{T,\Lsol} >\frac{A}{C_{K,d}}\right)\\
&\leq \PP \left(\|u\|_{T,\Lsol} + \|u^{N,(A)} - u\|_{T,\Lsol} > \frac{A}{C_{K,d}}\right)\\
&\leq \PP \left( \|u^{N,(A)} - u\|_{T,\Lsol} >\eta\right).
\end{align*}
Hence 
\begin{align*}
\PP \left( \|u_{t}^{N,(\infty)} - u_{t} \|_{T,\Lsol} \geq \eta \right) \leq 2 \PP \left( \|u^{N,(A)} - u\|_{T,\Lsol} >\eta\right).
\end{align*}
Now using Markov's inequality and Theorem \ref{th:rate2}, we obtain the desired result.
\end{proof}

\section{Propagation of chaos}
\label{sec:chaos}
In this section we study the well-posedness of the nonlinear SDE \eqref{eq:McKeanVlasov} and then the propagation of chaos of the particle system \eqref{eq:IPS}.

\subsection{Proof of Proposition \ref{prop:wpMP}}\label{subsec:proofNMP}
Let $T< T_{max}$ and let $u$ be the unique mild solution to \eqref{eq:PDE} up to $T$.

 The proof is organized as follows. Assuming there is a solution to the martingale problem, we study the mild equation of its time-marginals. We will see that this equation admits a unique solution in a suitable functional space (\textbf{Step 1}). This will enable us to study a linear version of the  martingale problem \eqref{MP} (\textbf{Step 2}). Analysing this linear martingale problem, we will get the uniqueness and existence for~\eqref{MP} (\textbf{Steps 3 and 4}).
 
\textbf{Step 1.} Assume $\Q$ is a solution to \eqref{MP}. Notice first that as the family of marginal laws $(q_t)_{t\leq T}$ belongs to $q\in \Ccal([0,T];\Lsol) $, one has according to Lemma~\ref{lem:unifbound} that
\begin{equation}
\label{eq:drift-bound-MP}
\sup_{t\leq T} \|K \ast q_t\|_{L^\infty(\R^d)} \leq C_{K,d}\, \sup_{t\leq T} \|q_t\|_{\Lsol}.
\end{equation}

\noindent To obtain the equation satisfied by $(q_t)_{t\leq T}$, one derives the mild equation for the marginal distributions of the corresponding nonlinear process. This is done in the usual way as the drift component is bounded (see e.g. \cite[Section 4]{TalayTomasevic}). 
  One has
 $$q_{t} =  e^{t\Delta} u_0 -  \int_0^t \nabla \cdot e^{(t-s)\Delta }  (q_{s}(K \ast q_{s}))\ ds, \quad 0\leq t \leq T.$$
This equation is exactly   \eqref{eq:mildKS} and we know it admits a unique solution in the sense of Definition~\ref{def:defMild} up to time $T$. Meaning, as $q\in \mathcal{C}([0,T]; \Lsol)$, the one-dimensional time marginals of $\Q$ are uniquely determined.

\textbf{Step 2.} Define the corresponding linear martingale problem $(\mathcal{MP}_{\text{lin}})$ in the same way as \eqref{MP}, except that in the definition of the process $(M_t)_{t\leq T}$ from \eqref{MP}-$(iii)$, fix $q$ to be the unique mild solution $u$ 
to \eqref{eq:PDE}.
By Girsanov transformation, the equation
$$Y_t = X_{0} + \sqrt{2} W_t+ \int_0^t (K \ast u_s)(Y_s) \ ds$$
admits a weak solution. In addition, weak uniqueness holds. 
 Let us show that the probability measure $\PP:=\mathcal{L}(Y)$ solves $(\mathcal{MP}_{\text{lin}})$.  Assume for a moment that the family $(\PP_t(dx))_{t\geq 0}= (p_t(x) dx)_{t\geq 0}$ (absolute continuity follows from bounded drift and Girsanov transformation) belongs to $\Ccal([0,T]; \Lsol)$. We will prove this fact in \textbf{Step 4}. Then, as $\PP$ is $\mathcal{L}(Y)$, all the requirements of $(\mathcal{MP}_{\text{lin}})$ are satisfied. In addition, this solution is unique.

\textbf{Step 3.} The previous step immediately yields the uniqueness of solutions to \eqref{MP} (as any solution to \eqref{MP} is also a solution to $(\mathcal{MP}_{\text{lin}})$  by the uniqueness of the one-dimensional time marginals). We now turn to the question of existence.\\
 A candidate for a solution to the problem \eqref{MP} is the probability measure $\PP$ defined above. To prove the latter solves \eqref{MP}, we need to ensure that the family of marginal laws $(\PP_t)_{0\leq t\leq T}$ is exactly the family $(u_t)_{0\leq t\leq T}$ we used to define the drift in $(\mathcal{MP}_{\text{lin}})$.

\noindent To do so, for $0<t\leq T$, one derives the mild equation for $\PP_t(dx)= p_t(x) dx$. Following the same arguments as in \cite[Section~4]{TalayTomasevic}, as the drift is bounded,
we have that for a.e. $x\in \R^d$,
$$p_{t} =  e^{t\Delta} u_0 -  \int_0^t \nabla \cdot e^{(t-s)\Delta }  (p_{s} (K \ast u_{s}))\ ds, \quad 0 \leq t \leq T .$$
Assume again that $p\in \Ccal([0,T]; \Lsol)$. 
The previous equation is a linearized version of Eq.~\eqref{eq:mildKS}  and, 
 by the same arguments as in Proposition~\ref{prop:local_existence}, it admits a unique solution in $\Ccal([0,T]; \Lsol)$.   Since both $u$ and $p$ solve this equation, they must coincide and we have the desired result : $(p_t)_{t\in [0,T]}=( u_t)_{t\in [0,T]}$.

\textbf{Step 4.} It only remains to prove that $p\in \Ccal([0,T]; \Lsol)$. Obviously, as we work with a family of probability density functions, we only need to prove that $p\in \Ccal([0,T]; L^\rK(\R^d))$. Performing the same calculations as in the proof of Proposition \ref{prop:cutPDE_unique0}, we get that 
$$\|p_t\|_{L^{\rK}(\R^d)}\leq \|u_0\|_{L^{\rK}(\R^d)} + C\int_0^t \frac{\|p_s\|_{L^{\rK}(\R^d)}}{\sqrt{t-s}} \|K\ast u_s\|_{L^\infty(\R^d)} \ ds.  $$
In view of Lemma \ref{lem:unifbound}, one has
\begin{equation*}
\| p_t\|_{L^\rK(\R^d)} \leq \|u_0\|_{L^\rK(\R^d)}
+  C \|u\|_{T,\Lsol} \int_0^t \frac{\|p_s\|_{L^\rK(\R^d)}}{\sqrt{t-s}}\ ds.  
\end{equation*}
Gr\"onwall's lemma implies that $p_t \in L^\rK(\R^d)$. Repeat the above computations for $p_{t}-p_{s}$ instead of $p_t$ to conclude that $p \in \Ccal([0,T]; L^\rK(\R^d))$. 
This concludes the proof.

\subsection{Proof of Theorem \ref{th:propagation-of-chaos}}\label{subsec:proofPropChaos}

To prove Theorem \ref{th:propagation-of-chaos}, we will show that $\mu^N$ converges to the unique solution $\Q$ of the martingale problem \eqref{MP}. To do so, we will first prove the convergence towards an auxiliary martingale problem which is identical to \eqref{MP} except that in the point $(iii)$ the process $(M_t)_{t\leq T}$ is the following: 
$$M_t:=f(w_t)-f(w_0)-\int_0^t \Big[ \Delta f(w_s)+  \nabla f(w_s)\cdot F_A(K \ast q_s (w_s))\Big]ds.$$
Then, we will lift the cut-off $F_A$ as $A$ will be chosen large enough. Let us call this auxiliary martingale problem $(\mathcal{MP}_A)$ and denote its unique solution by $\Q$ by a slight abuse of notation.

A usual way to prove that $\mu^N$ converges to $\Q$ consists in proving the tightness of the family  $\Pi^N:=\mathcal{L}(\mu^N)$ in the space $\Pp(\Pp(\mathcal{C}([0,T];\R^d)))$ and then, in proving that any limit point $\Pi^\infty$ of $\Pi^N$ is $\delta_\Q$. The latter is done by showing that under $\Pi^\infty$ a certain quadratic function of the canonical measure in $\Pp(\mathcal{C}([0,T];\R^d))$ is zero. The form of this function depends on the form of the process $(M_t)_{t\leq T}$ specified in the definition of the martingale problem. Moreover, one must analyse this function under $\Pi^N$ and use the convergence of $\Pi^N$ to $\Pi^\infty$ to get the desired result. This is where $\mu^N$ and the particle system appear.

However, here the situation is slightly modified. Namely, at the level of $\Pi^N$, we need to keep track not just of $\mu^N$, but also of the mollified empirical measure $u^N$ that appears in the definition of the particle system. That is why we will need to use the convergence of $u^N$ towards $u$ proved before and keep track of the couple $(\mu^N,u^N)$. This random variable lives in the product space
\begin{align*}
\mathcal{H}:=\Pp(\Ccal([0,T];\R^d))\times \mathcal{Y}
\end{align*}
endowed with the weak topology of $\Pp(\Ccal([0,T];\R^d))$ and the topology of  ${\mathcal{Y}}$, where 
\begin{equation}\label{eq:defY}
  \mathcal{Y}=  \mathcal{C}([0,T]; \Lsol) .
\end{equation}
 We will denote by $({\boldsymbol{\mu}},\mathbf{u})$ the canonical projections in $\mathcal{H}$.

Now for $N\geq 1$, we denote by $\tilde{\Pi}^N$ the law of the random variables $(\mu^N, u^N)$ that take values in $\mathcal{H}$. The sequence $(\tilde{\Pi}^N, N\geq 1)$ is tight if and only if $(\tilde{\Pi}^N\circ {\boldsymbol{\mu}}, N\geq 1)$ and $(\tilde{\Pi}^N\circ \mathbf{u}, N\geq 1)$ are tight. The tightness of $(\tilde{\Pi}^N\circ {\boldsymbol{\mu}}, N\geq 1)$ is classical, as the drift of the particle system is bounded. As for $(\tilde{\Pi}^N\circ \mathbf{u}, N\geq 1)$, we have already proven the convergence of $(u^N, N\geq 1)$ in $ \mathcal{Y}$ (see Theorem \ref{th:rate2}).

Once we have the tightness of $(\tilde{\Pi}^N, N\geq 1)$, let $\tilde{\Pi}^\infty$ be a limit point of $(\tilde{\Pi}^N, N\geq 1)$. By a slight abuse of notation, we denote the subsequence converging to it by $(\tilde{\Pi}^N, N\geq 1)$ as well. We will study the support of $\tilde{\Pi}^\infty$ in order to describe the support of $\Pi^\infty:= \tilde{\Pi}^\infty \circ {\boldsymbol{\mu}}$.

The following lemma shows that the marginals of ${\boldsymbol{\mu}}$ and $\mathbf{u}$ coincide under the limit probability measure. This will be extremely useful to obtain that the support of $\Pi^\infty$ is concentrated around~$\Q$.
\begin{lemma}
\label{lemma:chaos-marginals}
$\tilde{\Pi}^\infty$-almost surely, ${\boldsymbol{\mu}}_t$ is absolutely continuous w.r.t. the Lebesgue measure and its density is ${\boldsymbol{\mu}}_t(dx)=\mathbf{u}_t(x)dx$.
\end{lemma}
\begin{proof}
This is a consequence of Inequality \eqref{eq:rateunmun}. Take a test function $\varphi \in \Ccal_c^\infty([0,T]\times \R^d)$ and define a functional $\phi(t,x)= \varphi(t,x_t)$, for $x\in \Ccal([0,T];\R^d)$. Then, 
\begin{align*}
\EE_{\tilde{\Pi}^\infty}|\langle \mathbf{u},\varphi\rangle -\langle dt\otimes {\boldsymbol{\mu}}, \phi \rangle|&= \lim_{N\to \infty} \EE_{\tilde{\Pi}^N}|\langle \mathbf{u},\varphi\rangle -\langle dt\otimes {\boldsymbol{\mu}}, \phi \rangle|= \lim _{N\to \infty} \EE |\langle u^N,\varphi\rangle -\langle dt\otimes \mu_t^N, \varphi \rangle|\\
& \leq \lim _{N\to \infty} \EE \int_0^T  |\langle u^N_t,\varphi(t,\cdot) \rangle - \langle \mu_t^N, \varphi(t,\cdot) \rangle| dt\\
&\leq C_T \sup_{t\in[0,T]} \|\varphi(t,\cdot)\|_{\text{Lip}} \times \lim_{N\to \infty} \frac{1}{N^{\alpha}},
\end{align*}
where the last inequality comes from \eqref{eq:rateunmun}. Thus, we obtain that $\tilde{\Pi}^\infty$-a.s. the following measures on $\R^d\times [0,T]$ are equal:
$$ \mathbf{u}_t(x)dx \ dt={\boldsymbol{\mu}}_t(dx) dt,$$
hence $\tilde{\Pi}^\infty$-a.s., for almost all $t\in[0,T]$,
$$ \mathbf{u}_t(x)dx={\boldsymbol{\mu}}_t(dx) .$$
\end{proof}

The following proposition will be the last ingredient needed for the proof of Theorem \ref{th:propagation-of-chaos}.
\begin{proposition}
\label{prop:chaos}
Let $p\in \N$, $f\in \Ccal_b^2(\R^d), ~\Phi \in \Ccal_b(\R^{dp})$ and $ 0< s_1 < \dots < s_p\leq s<t\leq T$. Define $\Gamma$  as the following function on $\mathcal{H}$ :
\begin{multline*}
\Gamma({\boldsymbol{\mu}},\mathbf{u})=\int_{\Ccal([0,T];\R^d)} \Phi(x_{s_1}, \dots, x_{s_p}) \bigg[f(x_{t})-f(x_{s})-\int_s^t  \Delta f(x_\sigma) d\sigma \\+\int_s^t F_A(K \ast \mathbf{u}_\sigma(x_\sigma))\cdot \nabla f(x_\sigma) d\sigma \bigg] d{\boldsymbol{\mu}}(x).
\end{multline*}
Then 
$\EE_{\tilde{\Pi}^\infty}(\Gamma^2)=0.$
\end{proposition}
\begin{proof}\textbf{Step 1.} Notice that $\lim_{N\to \infty} \EE_{\tilde{\Pi}^N}(\Gamma^2)=0.$ Indeed,  by It\^o's formula applied on $\frac{1}{N}\sum_{i=1}^N (f(X^i_t)-f(X^i_s))$, one has
$$\EE_{\tilde{\Pi}^N}(\Gamma^2)= \EE (\Gamma(\mu^N, u^N)^2)=\EE \left( \frac{1}{N} \sum_{i=1}^N \int_s^t \nabla f(X^i_\sigma)\cdot dW_\sigma^i\right)^2= \frac{1}{N^2} \sum_{i=1}^N \EE\left(\int_s^t \nabla f(X^i_\sigma) \cdot dW_\sigma^i\right)^2 \leq \frac{C}{N}.$$

\textbf{Step 2.}  We prove that $\Gamma$ is continuous on $\mathcal{H}$. Let $({\boldsymbol{\mu}}^n,\mathbf{u}^n)$ be  a sequence converging in $\mathcal{H}$ to $({\boldsymbol{\mu}},\mathbf{u})$. Let us prove $\lim_{n\to \infty}|\Gamma({\boldsymbol{\mu}}^n,\mathbf{u}^n)- \Gamma({\boldsymbol{\mu}},\mathbf{u})|=0$.

We decompose 
$$|\Gamma({\boldsymbol{\mu}}^n,\mathbf{u}^n)- \Gamma({\boldsymbol{\mu}},\mathbf{u})|\leq |\Gamma({\boldsymbol{\mu}}^n,\mathbf{u}^n)- \Gamma({\boldsymbol{\mu}}^n,\mathbf{u})| + |\Gamma({\boldsymbol{\mu}}^n,\mathbf{u})- \Gamma({\boldsymbol{\mu}},\mathbf{u})|=: I_n + II_n.$$

Notice that
\begin{align}
\label{eq:proofH1}
I_n &\leq \|\Phi\|_\infty \|\nabla f\|_\infty \langle {\boldsymbol{\mu}}^n,\int_s^t |F(K \ast \mathbf{u}^n_\sigma (\cdot_{\sigma}) )-F(K \ast \mathbf{u}_\sigma (\cdot_{\sigma}))| d\sigma \rangle \nonumber \\
& \leq C \int_s^t \langle {\boldsymbol{\mu}}^n_\sigma,|K \ast (\mathbf{u}^n_\sigma - \mathbf{u}_\sigma)|\rangle ~ d\sigma \leq C \int_s^t \|K \ast (\mathbf{u}^n_\sigma - \mathbf{u}_\sigma)\|_\infty d\sigma.
\end{align}
In view of Lemma~\ref{lem:unifbound}, one has 
\begin{align*}
\|K \ast (\mathbf{u}^n_\sigma- \mathbf{u}_\sigma) \|_{L^\infty(\R^d)} 
& \leq  C_{K,d} \|\mathbf{u}^n_\sigma- \mathbf{u}_\sigma\|_{\Lsol}.
\end{align*}
 Hence $I_{n}$ converges to $0$.

To prove that $II_n$ converges to zero, as ${\boldsymbol{\mu}}^n$ converges weakly to ${\boldsymbol{\mu}}$, we should prove the continuity of the functional $G: C([0,T];\R^d)\to \R$ defined by
$$G(x)= \Phi(x_{s_1}, \dots, x_{s_p})[f(x_{t})-f(x_{s})-\int_s^t  \Delta f(x_\sigma) d\sigma -\int_s^t F_A(K \ast \mathbf{u}_\sigma(x_\sigma))\cdot \nabla f(x_\sigma) d\sigma]. $$
Let $(x^n)_{n\geq 1}$ a sequence converging in $\Ccal([0,T];\R^d)$ to $x$. To prove $G(x^n) \to G(x)$ as $n\to \infty$, having in mind the properties of $f$ and $\Phi$, we should only concentrate on the term  ${\int_s^t F_A(K\ast \mathbf{u}_\sigma(x^n_\sigma))\cdot \nabla f(x^n_\sigma) ~d\sigma}$. Here we use the continuity property \eqref{HK3} 
 to deduce that $K \ast \mathbf{u}_\sigma(x^n_\sigma)$ converges to  $K \ast \mathbf{u}_\sigma(x_\sigma) $ and by dominated convergence,
$$\int_s^t F_A(K \ast \mathbf{u}_\sigma(x^n_\sigma))\cdot \nabla f(x^n_\sigma) ~d\sigma \to \int_s^t F_A(K \ast \mathbf{u}_\sigma(x_\sigma))\cdot \nabla f(x_\sigma) ~d\sigma, \quad \text{as } n \to \infty.$$
\textbf{Conclusion.} Combine Step 1 and Step 2 to finish the proof.
\end{proof}

We have all the elements in hand to finish the proof of Theorem \ref{th:propagation-of-chaos}. 
By Lemma \ref{lemma:chaos-marginals} and Proposition \ref{prop:chaos}, we get that ${\boldsymbol{\mu}} \in \text{supp} (\Pi^\infty)$ solves the nonlinear martingale problem $(\mathcal{MP}_A)$. Choose $A> A_T:= C_{k,D} \|q\|_{T,\Lsol}$ and lift the cut-off (see \eqref{eq:drift-bound-MP}). Then,  ${\boldsymbol{\mu}}$ solves the nonlinear martingale problem \eqref{MP}. As we have the uniqueness for \eqref{MP}, we get that there is only one limit value of the sequence $\Pi^N$ which is $\delta_\Q$.

\section{Riesz-like singular kernels}\label{sec:sing}

In this section, we deal with kernels that are too singular to be covered by Assumption \eqref{HK}, typically deriving from Riesz potentials \eqref{eq:defRiesz} with $s\in(d-2,d-1)$ (see Section \ref{sec:examples} for more details).

The main results are the analogous of Theorem \ref{th:rate2} (see Theorem \ref{th:rate2sing}) and Theorem \ref{th:propagation-of-chaos} (see Theorem \ref{th:propagation-of-chaos-sing}).

The following alternative to Assumption \eqref{HK} will now be considered:

~

\noindent \begin{minipage}{0.06\linewidth}
(\customlabel{HKsing}{$\HypKK$}):
\end{minipage}
\hspace{-0.5cm}
\begin{minipage}[t]{0.94\linewidth}
\raggedright
 \hspace{0.8cm} 
\begin{enumerate}[label= ]
\item(\customlabel{HK1sing}{$\HypKK_{i}$})   $K\in L^1 (\mathcal{B}_{1})$;

\item(\customlabel{HK2sing}{$\HypKK_{ii}$}) 
 $K\in L^{\qK}(\mathcal{B}_{1}^c)$, for some  $\qK\in[1,+\infty]$;
 
\item  (\customlabel{HK3sing}{$\HypKK_{iii}$})  There exists  $\rKs\in (d,+\infty)$, $\beta\in ( \tfrac{d}{\rKs},1)$, $\Hreg \in (0,1]$ and $C>0$ such that for any $f\in L^{1}\cap H_{\rKs}^{\beta}(\R^d)$, one has
$$\mathcal{N}_\Hreg (K\ast f)\leq C  \|f\|_{L^{1}\cap H_{\rKs}^{\beta}(\R^d) }. $$

\end{enumerate}
\end{minipage}

\vspace{0.3cm}
 
~

The restriction with respect to the key parameters in this setting is given by the following assumption:

~

\noindent 
\begin{minipage}{\linewidth}
(\customlabel{C04sing}{$\widetilde{\Hyp}_{\alpha}$}): \hspace{0.35cm} The parameters $\alpha$, $ \beta$ and $\rKs$ (which appear respectively in \eqref{eq:VN} and \eqref{HK3sing}) satisfy 
$$ 0<\alpha<\frac{1}{d+2 \beta + 2 d (\frac{1}{2} - \frac{1}{\rKs})\vee 0} .$$
\end{minipage}

\vspace{0.4cm}

 \noindent A kernel which would satisfy Assumption \eqref{HK} also satisfies \eqref{HKsing} (by H\"older's inequality in \eqref{HK1sing} and a Sobolev embedding in \eqref{HK3sing}), hence this new assumption is more general. However \eqref{HKsing}-\eqref{C04sing} impose more restrictions on the choice of the parameters $\alpha$, $\beta$ and $\rKs$. For instance, we are no longer able to choose $\alpha = (\frac{1}{d})^-$, but instead we have $\alpha = (\frac{1}{2d})^-$ at best.

~
 
Finally, let us now state the assumptions on the initial conditions of the system:

~

\noindent 
\begin{minipage}{0.05\linewidth}
(\customlabel{Hsing}{$\widetilde{\Hyp}$}):
\end{minipage}
\begin{minipage}[t]{0.94\linewidth}
\raggedright
Let $\rKs\in (d,+\infty)$ and $\beta\in ( \tfrac{d}{\rKs},1)$. For any $m\geq1$, $\displaystyle\sup_{1\leq i\leq N,\, N\in \N} \EE|X_{0}^{i,N}|^m<\infty$ and
$$
 \sup_{N \in \N} \EE\left[  \left\|  \mu^N_{0} \ast V^N \right\|_{\beta,\rKs}^m   \right] \ 
<\infty .
$$
\end{minipage}

\vspace{0.4cm}

\noindent A sufficient condition for \eqref{Hsing} to hold is that particles are initially i.i.d. with a density that is smooth enough (see \cite[Lemma 2.9]{FlandoliLeocata} for a related result). The reader may also find interesting comments on a similar assumption in \cite[Remark 1.2]{FlandoliOliveraSimon}.

~

We adopt a definition of mild solutions to \eqref{eq:PDE} in $\Ccal([0,T]; \Lsolsing)$ which is analogous to Definition \ref{def:defMild}. We get the following result, which is an immediate adaptation of Proposition \ref{prop:local_existence}. The proof is omitted.
\begin{proposition}\label{prop:local_existence_sing}
Assume that the kernel $K:\R^d \to \R^d$ satisfies \eqref{HK1sing} and \eqref{HK2sing} and that the initial condition is $u_{0}\in \Lsolsing$. 
Then there exists $T>0$ such that the PDE \eqref{eq:PDE} 
admits a 
mild solution $u$ in $\Ccal([0,T]; \Lsolsing)$. In addition, this mild solution is unique.
\end{proposition}
The above proposition and the theorem below rely in particular on the Lemma \ref{lem:unifbound} applied with $\rK=\infty$.

\begin{theorem}\label{th:rate2sing}  
Assume that the initial conditions $\{\mu^N_{0}\}_{N\in\N}$ satisfy \eqref{Hsing} and that the kernel $K$ satisfies \eqref{HKsing}.  Moreover, let \eqref{C04sing} hold true. 
 Let $T_{\max}$ be the maximal existence time for \eqref{eq:PDE} in the space $\mathcal{C}\left([0,T],\Lsolsing\right)$ and fix $T\in(0,T_{\max})$. 
In addition, let the dynamics of the particle system be given by \eqref{eq:IPS} with $A$ greater than $A_{T}$ (defined in \eqref{eq:AT'} with the $\Lsolsing$ norm instead of the $\Lsol$ norm). 

Then, for any $\varepsilon>0$ and any $m\geq 1$, there exists a constant $C>0$ such that for all $N\in\N^*$,
\begin{align*}
\left\| \| u^N-u \|_{T,\Lsolsing} \right\|_{L^m(\Omega)}
&\leq  C \left(\left\| \|u^N_0- u_0\|_{\Lsolsing} \right\|_{L^m(\Omega)}  +  N^{-\varrho + \varepsilon} \right),
\end{align*}
where
\begin{align*}
\widetilde{\varrho} = \min \left(\alpha \Hreg , \, \frac{1}{2}-  \alpha d  \right)  .
\end{align*}
\end{theorem}

\paragraph{Elements of proof of Theorem \ref{th:rate2sing}.}
We do not detail the proof of this theorem, as it follows the same lines as the proof of Theorem \ref{th:rate2}. We only mention how the key estimates evolve. 

First, the following embeddings justify that we assume $\beta>\frac{d}{\rKs}$ in \eqref{HK3sing}. 
Since $\beta-\frac{d}{\rKs}>0$, $H^{\beta}_{\rKs}(\R^d)$ is continuously embedded into $\mathcal{C}^{\beta-\frac{d}{\rKs}}(\R^d)$ (see \cite[p.203]{Triebel}). In particular $H^{\beta}_{\rKs}$ is continuously embedded into $L^{\rKs} \cap L^\infty$. That is, there exists $C>0$ such that 
\begin{equation}\label{eq:GeneralMorrey}
\|f\|_{L^{\rKs}\cap L^\infty(\R^d)} \leq C \|f\|_{\beta,\rKs} \quad \text{and} \quad \|f\|_{\mathcal{C}^{\beta-\frac{d}{\rKs}}(\R^d)} \leq C \|f\|_{\beta,\rKs}, \quad \forall f\in H^{\beta}_{\rKs}(\R^d) .
\end{equation}
Then by interpolation,  $L^1\cap H^{\beta}_{\rKs}(\R^d)$ is continuously embedded into $\Lsolsing$. 
That is,  
there exists $C_{d,\beta,\rKs}>0$ such that
\begin{equation}\label{eq:embed1}
\|f\|_{\Lsolsing} \leq C_{d,\beta,\rKs} \|f\|_{L^1\cap H^{\beta}_{\rKs}(\R^d)}, \quad \forall f\in L^1\cap H^{\beta}_{\rKs}(\R^d) .
\end{equation}

The Assumption \eqref{C04sing}, which strengthens \eqref{C04}, is also partially understood in view of these embeddings. This assumption yields the boundedness result for $u^N$ given in Proposition \ref{prop:bound1} (it  generalises to $\beta>0$ the result given in \eqref{eq:boundednessuN}) and reads, for any $q \geq 1$,
$$\sup_{N\in\N^*}  \mathbb{E}\left[ \sup_{t\in[0,T]} \left\Vert  u^N_{t} \right\Vert _{\beta,\rK}^{q}\right]  <\infty.$$

It is then possible to control $\EE \|K\ast u^N_{s}\|_{L^\infty(\R^d)}^m$ by $\EE \|u^N_{s}\|_{\Lsolsing}^m$ thanks to Lemma \ref{lem:unifbound} (applied with $\rK=\infty$), then using the embedding \eqref{eq:embed1} and finally the previous bound ensures that
\begin{equation*}
\sup_{N\in\N^*} \sup_{s\in[0,T]} \,\EE \|K\ast u^N_{s}\|_{L^\infty(\R^d)}^m <\infty .
\end{equation*}
A similar control is obtained on $\sup_{N\in\N^*} \sup_{s\in[0,T]} \, \|K\ast u_{s}\|_{L^\infty(\R^d)}$ using Proposition \ref{regularity} instead of Proposition \ref{prop:bound1}.

Therefore, we obtain a bound analogous to \eqref{ineu}:
\begin{equation*}
\begin{split}
&\left\| \| u^N-u\|_{t,\Lsolsing} \right\|_{L^m(\Omega)} \leq \left\| \|u^N_0- u_0\|_{\Lsolsing} \right\|_{L^m(\Omega)} \\
& +   C \int_0^t \frac{1}{(t-s)^{\frac{1}{2}}} 
 \left\|  \| u^{N}-u \|_{s,\Lsolsing} \right\|_{L^m(\Omega)} ds +  \left\|  \| E\|_{t,\Lsolsing} \right\|_{L^m(\Omega)} \\&+ \left\|  \| M^N\|_{t,\Lsolsing} \right\|_{L^m(\Omega)} ,
\end{split}
\end{equation*}
and to \eqref{eq:errorL1} for the new error term:
\begin{align*}
\left\| \| E\|_{t, \Lsolsing} \right\|_{L^m(\Omega)} \leq   \frac{C}{N^{\alpha\Hreg}} .
\end{align*} 
It remains to bound the stochastic integral term in $\Lsolsing$ norm, which is given immediately by Proposition \ref{prop:martingale-bound}: for any $\varepsilon>0$, there exists $C>0$ such that
\begin{equation*}
\left\| \sup_{s\in [0,t] } \|M^N_{s}\|_{\Lsolsing} \right\|_{L^{m}(\Omega)}  \leq C\,  N^{- \frac{1}{2}\left(1-2\alpha d) \right) + \varepsilon} .
\end{equation*}
Using the three previous inequalities, Gr\"onwall's lemma gives the desired result.
\qed

~

\begin{remark}\label{rk:bestrate}
Assuming that the kernel $K$ is such 
that $\Hreg$ can be chosen equal or close to 1, the rate is really determined by the minimum between $\alpha$ and $\frac{1}{2} -  \alpha d$, under the constraint that $\alpha<(d+2\beta+2d(\frac{1}{2} - \frac{1}{\rKs}))^{-1}$ (see \eqref{C04sing}) and $\beta<1$ (see \eqref{HK3sing}). We will see in Section \ref{subsec:RieszKernel} that the best rate for a Riesz kernel of parameter $s\in (d-2,d-1)$ is 
\begin{align*}
\widetilde{\varrho} = \frac{1}{2(d+1)} ,
\end{align*}
which is obtained for $\alpha = \frac{1}{2(d+1)}$ and initial conditions satisfying \eqref{Hsing} with any $\beta\in (\frac{d}{\rKs},1)$ and $\rKs> \frac{d}{\beta-2 +d-s}$.
\end{remark}

In Section \ref{subsec:proofCor}, using  an interpolation inequality between the results of Proposition \ref{prop:bound1} and Theorem \ref{th:rate2}, we obtain the following rate of convergence with respect to Sobolev norms:
\begin{corollary}\label{cor:forpropchaossing}
Let the assumptions of Theorem~\ref{th:rate2sing} hold and assume further that ${u_{0}\in H_{\rKs}^\beta(\mathbb{R}^d)}$. 
Then, for any $\varepsilon>0$ and any $m\geq 1$, there exists a constant $C>0$ such that for all $N\in\N^*$,
\begin{align*}
\left\| \sup_{t\in[0,T]}\| u^N_{t}-u_{t} \|_{\gamma,\rKs-\delta} \right\|_{L^m(\Omega)}
&\leq C \left(\left\|   \|u^N_0- u_0\|_{\Lsolsing} \right\|_{L^m(\Omega)}  + N^{-\widetilde{\varrho} + \varepsilon}\right)^\frac{\gamma}{\beta} ,
\end{align*}
for $\delta\in(0,1)$ and $\gamma = \beta \frac{\rKs(\rKs-1-\delta)}{(\rKs-\delta)(\rKs-1)}$.
\end{corollary}

\begin{remark}
It is clear that $\gamma<\beta$. It will also be important (in particular for the propagation of chaos in the Section \ref{subsec:chaossing}) to ensure that $\gamma>\frac{d}{\rKs-\delta}$ so as to have an embedding in a space of H\"older continuous functions. This is indeed the case if $\delta$ is chosen small enough, see condition \eqref{eq:conddelta}.

In addition, we get by the classical embedding recalled in \eqref{eq:GeneralMorrey} that the result of Corollary \ref{cor:forpropchaossing} imply the same rates in $\eta$-H\"older norm, with $\eta = \gamma - \frac{d}{\rKs-\delta}$, provided that this quantity is positive (see condition \eqref{eq:conddelta}).
\end{remark}

Observing that a rate of convergence was obtained in $H^\gamma_{\rKs-\delta}$, one could wonder if the convergence also happens in $H^\beta_{\rKs}$. Corollary \ref{cor:weakConv} answers positively, thus extending the main convergence results in \cite{FlandoliOliveraSimon,ORT} to general kernels.
\begin{corollary}\label{cor:weakConv}
 Let the same assumptions as in Corollary~\ref{cor:forpropchaossing} hold, with $\gamma, \delta$ as in Corollary \ref{cor:forpropchaossing}. Let $m\geq 1$ and assume further that
\begin{align}\label{eq:ICconvergence}
\bigg\|\|u^N_0- u_0\|_{\Lsolsing} \bigg\|_{L^m(\Omega)} \longrightarrow 0 .
\end{align}
 Then the sequence of mollified empirical measures ${\{u^N_{t} , ~t\in[0,T]\}_{N\in\N}}$ converges in probability, as $N\to \infty$, towards the unique mild solution $u$ on $[0,T]$ of the PDE \eqref{eq:PDE}, in the following sense: 
 \begin{equation*}
 \forall \varphi \in L^2\left([0,T]; H_{\rKs^{\prime}}^{-\beta}(\R^d)\right), \quad \int_{0}^{T} \langle u^N_{t},\varphi_{t}\rangle_{\beta} \, dt \overset{\PP}{\longrightarrow} \int_{0}^{T} \langle u_{t},\varphi_{t}\rangle_{\beta} \, dt,
 \end{equation*}
 where $\rKs' = \frac{\rKs}{\rKs-1}$ is the conjugate exponent of $\rKs$.
\end{corollary}
The proof is given in Section \ref{subsec:proofConv}. Theorem \ref{th:rate2sing} and Corollary \ref{cor:weakConv} cover the convergence results for the mollified empirical measures obtained in a case by case basis in \cite{FlandoliOliveraSimon,ORT}.  We present here an alternative approach to the one presented initially in \cite{FlandoliLeimbachOlivera} to prove such convergence which, in addition, enables us to quantify it. We believe that this new approach could lead to rates of convergence in slightly different models such as those presented in  \cite{FlandoliLeimbachOlivera,FlandoliLeocataRicci0,FlandoliLeocata}.

~

We also notice without giving the full statement, that as in Corollary \ref{cor:wo-cutoff}, a rate of convergence can be obtained for the particle system without cutoff.

~

Finally, it remains to give the results about the propagation of chaos in the context of the hypothesis \eqref{HKsing}. \\
Firstly, one must redefine the martingale problem \eqref{MP}. Let us denote by (\customlabel{MP'}{$\mathcal{\widetilde{MP}}$}) a martingale problem analogous to \eqref{MP} except that in point $(ii)$ we replace $\Lsol$ by $\Lsolsing$. 

Then, adapt the arguments of Proof of Proposition \ref{prop:wpMP} do this functional setting by using the well-posedness of the mild solution given by Proposition \ref{prop:local_existence_sing} and Lemma \ref{lem:unifbound} (with $\rK=\infty$ as mentioned above). This yields the following result:
\begin{proposition}\label{prop:wpMPsingular}
Let $T< T_{max}$. Assume that $u_0$ is a probability density function belonging to $L^\infty(\R^d)$ and that the kernel $K$ satisfies \eqref{HKsing}. 
Then, the martingale problem \eqref{MP'} 
admits a unique solution.
\end{proposition} 
Now, we are in the position to state the propagation of chaos result. 
\begin{theorem}
\label{th:propagation-of-chaos-sing}
Let the hypotheses of Theorem \ref{th:rate2sing} hold. Assume further that the family of random variables $\{X^i_{0},~i\in\N\}$ is identically distributed, that $u_{0}\in H_{\rKs}^\beta(\mathbb{R}^d)$ and that $\langle u^N_0, \varphi\rangle  \to \langle u_{0}, \varphi\rangle $ in probability, for any $\varphi \in \mathcal{C}_{b}(\R^d)$. Then, the empirical measure $\mu^N_.$  converges in probability towards $\Q$, which is the law of the  unique weak solution of \eqref{eq:McKeanVlasov}. 
\end{theorem}
The scheme of proof is the same as for Theorem \ref{th:propagation-of-chaos}, but there are some additional technical difficulties that need to be adressed. Hence we will give the sketch of the proof in Section~\ref{subsec:chaossing}.

\subsection{Rate in Sobolev norm: Proof of Corollary \ref{cor:forpropchaossing}}\label{subsec:proofCor}

This result relies on an interpolation inequality for Bessel potential spaces, and our previous results of rate of convergence.

Let us establish first the interpolation inequality that we shall use: let $\delta\in(0,1)$ and $\gamma$ such that
\begin{align}\label{eq:gammadelta}
\gamma = \beta \frac{\rKs(\rKs-\delta-1)}{(\rKs-\delta)(\rKs-1)}.
\end{align}
The interpolation theorem for Bessel potential spaces, see \cite[p.185]{Triebel}, gives that for any $f\in H^0_{1}\cap H^\beta_{\rKs}(\R^d)\,  (\equiv L^1\cap H^\beta_{\rKs}(\R^d))$,
\begin{align}\label{eq:interpolationBessel}
 \| f \|_{\gamma,\rKs-\delta}  \leq  \| f \|_{0,1}^{\theta} \,  \| f \|_{\beta,\rKs}^{1-\theta} ,
\end{align}
where $\theta = \frac{\gamma}{\beta}$.

Hence it follows from \eqref{eq:interpolationBessel} that for any $m\geq 1$,
\begin{align*}
\EE \sup_{s\in[0,T]} \| u^N_s-u_{s}\|_{\gamma,\rKs-\delta}^m  \leq  \EE \sup_{s\in[0,T]} \| u^N_s-u_{s}\|_{0,1}^{\theta m} \,  \| u^N_s-u_{s}\|_{\beta,\rKs}^{(1-\theta) m} ,
\end{align*}
and we deduce from H\"older's inequality that
\begin{align*}
 \EE \sup_{s\in[0,T]} \| u^N_s-u_{s}\|_{\gamma,\rKs-\delta}^m \leq  \left( \EE \sup_{s\in[0,T]} \| u^N_s-u_{s}\|_{L^1(\R^d)}^m \right)^{\theta}  \left( \EE \sup_{s\in[0,T]} \| u^N_s-u_{s}\|_{\beta,\rKs}^m  \right)^{1-\theta}. 
\end{align*}
In view of the previous inequality and using Theorem \ref{th:rate2sing}, Proposition \ref{regularity} and Proposition \ref{prop:bound1}, we deduce the rate of convergence in $L^m \left(\Omega; L^{\infty}([0,T], H_{\rKs-\delta}^{\gamma}(\R^d))\right)$. 

Finally, note that it is always true that $\gamma<\beta$. Besides, it will be important to ensure that $\gamma>\frac{d}{\rKs-\delta}$ to have an embedding in the space of H\"older continuous functions (see \eqref{eq:GeneralMorrey}). For this, it suffices to choose $\delta$ which satisfies:
\begin{align}\label{eq:conddelta}
\frac{\rKs-\delta-1}{\rKs-1} > \frac{d/\rKs}{\beta}.
\end{align}

\subsection{Convergence in $H_{ \texorpdfstring{\rKs}{r}}^{\beta}$: Proof of Corollary  \ref{cor:weakConv}}
\label{subsec:proofConv}

Let us introduce the space 
\begin{equation*}
\mathcal{X} = L^2 \left([0,T]; H_{\rKs}^{\beta}(\R^d)\right)
\end{equation*}
with the strong topology, and denote by 
 $\mathcal{X}_{w}$ the same space endowed with the weak topology. Notice first that $L^{2}\big( [0,T]; H_{\rKs}^{\beta}(\R^d) \big)$ is a reflexive Banach space for $1<\rKs<\infty$ (see \cite[p.198-199]{Triebel}). Hence by the Banach-Alaoglu theorem, it is compactly embedded in $\mathcal{X}_{w}$.  

Now the Chebyshev inequality ensures that 
\[
\PP\big( \|u^N \|_{\mathcal{X}}^{2} > R\big)\leq \frac{\EE \big[ \big\| u^N\big\|_{\mathcal{X}}^{2}\big]}{R}  \leq \frac{\sup_{t\in [0,T]}  \EE \big[ \big\| u_{t}^N\big\|_{ \beta,\rKs }^{2}\big]}{R}, \quad \text{for any } R>0.
\]
Thus by Proposition \ref{prop:bound1}, we obtain 
\[
\PP\big( \big\|u^N \big\|_{\mathcal{X}}^{2} > R\big)\leq \frac{C}{R}, \quad \text{for any } R>0 ~\text{ and } ~N \in \N.
\]
Let $\PP_{N}$ be the law of $u^N$ in $\mathcal{X}$. The last inequality implies that for any $\epsilon>0$, there exists a bounded set $B_{\epsilon}\in \mathcal{X}$ 
such that $\PP_{N}(B_{\epsilon})< 1-\epsilon$ for all $N$, and therefore by the preliminary remark, there exists a compact set $
\mathcal{K}_{\epsilon}\in \mathcal{X}_{w}$ such  that $\PP_{N}(\mathcal{K}_{\epsilon})< 1-\epsilon$. That is, $(\PP_N)_{N\geq1}$ is tight on $\mathcal{X}_{w}$.

Here, we cannot apply the usual Prokhorov Theorem, since $\mathcal{X}_{w}$ is not metrisable. However, $\mathcal{X}$ is reflexive and $\mathcal{X}'$ is separable, therefore any weak compact set in $\mathcal{X}_{w}$ is metrisable (see e.g. \cite[Thm~3.16]{RudinFA}). In addition, $\mathcal{X}_{w}$ is completely regular in the sense of \cite[Def. 6.1.2]{BogachevII} (as is any Hausdorff topological vector space). Hence, by Theorem 8.6.7 of \cite{BogachevII}, there exists a subsequence of probability measures $\PP_{\epsilon_{N}}$ that converges weakly to some $\PP_{\infty}$ on $\mathcal{X}_{w}$. In particular, if $F\in \mathcal{C}^0_{b}(\R)$ and $\varphi\in L^2 \left([0,T]; H_{\rKs^{\prime}}^{-\beta}(\R^d)\right)$, we get that
\begin{equation*}
\EE F \left(\langle u^{\epsilon_{N}},\varphi \rangle \right) \rightarrow \EE F \left(\langle u^{\infty},\varphi \rangle \right) ,
\end{equation*}
for some $\mathcal{X}_{w}$-valued random variable $u^\infty$ with law $\PP_{\infty}$.

On the other hand, from Theorem \ref{th:rate2sing} with assumption \eqref{eq:ICconvergence}, we know that 
$u^N$ converges to $u$ in $L^m \left(\Omega;\, L^{\infty}( [0,T], \Lsolsing)\right)$. Hence $u^N$ converges almost surely to $u$ in $L^{\infty}( [0,T], \Lsolsing)$, up to a subsequence. Without loss of generality, we still denote by $u^N$ this subsequence. By testing against any function $\varphi\in \mathcal{C}^\infty \left([0,T]\times \R^d\right)$, we deduce that ${\EE F \left(\langle u^{N},\varphi \rangle \right) \rightarrow F \left(\langle u,\varphi \rangle \right)}$ for any $F\in \mathcal{C}^0_{b}(\R)$. Thus by uniqueness of the limit, we get that $\EE F \left(\langle u^{\infty},\varphi \rangle \right) = F \left(\langle u,\varphi \rangle \right)$.
 Since $\mathcal{C}^\infty \left([0,T]\times \R^d\right)$ is dense in $L^2 \left([0,T]; H_{\rKs^{\prime}}^{-\beta}(\R^d)\right)$, it follows that $\PP_{\infty}$ is a Dirac measure at $u$.

We have thus obtained that any limit point of $(\PP_{N})_{N\geq1}$ is the Dirac measure at $u$, therefore $(u^{N})_{N\geq1}$ converges in law to $u$ in $\mathcal{X}_{w}$. Since the limit is deterministic, the convergence also holds in probability, in the sense of Corollary \ref{cor:weakConv}.

\subsection{Sketch of the proof of Theorem \ref{th:propagation-of-chaos-sing}}

\label{subsec:chaossing}
We will show that $\mu^N$ converges to the unique solution $\Q$ of the martingale problem \eqref{MP'}. 
The proof closely follows the proof of Theorem \ref{th:propagation-of-chaos} (see Section \ref{subsec:proofPropChaos}) in a new functional framework. Namely, we redefine the space $\mathcal{Y}$ (see \eqref{eq:defY}) in the following way:
$${  \mathcal{Y}}=  \mathcal{C}([0,T]; L^1(\R^d)) \cap \mathcal{X}_{w} .$$
Recall from Section \ref{subsec:proofConv} that $\mathcal{X}_{w}$ is the space  $L^2([0,T]; H^\beta_{\rKs}(\R^d))$ endowed with the weak topology. This choice of a new functional framework will be explained at the end of this section.

Adopting the notation of Section \ref{subsec:proofPropChaos}, the sequence $(\tilde{\Pi}^N, N\geq 1)$  
is tight. Indeed, the convergence of Theorem \ref{th:rate2sing} gives the tightness in $\mathcal{C}([0,T]; L^1(\R^d))$ and the discussion in Section \ref{subsec:proofConv} gives the tightness in $\mathcal{X}_{w}$. Besides, despite the fact that $\mathcal{X}_{w}$ is not metrisable, we explained in Section \ref{subsec:proofConv} why it is still possible to extract a converging subsequence out of $(\tilde{\Pi}^N, N\geq 1)$.

Now repeat the arguments that  follow line by line to obtain that an analogue of Lemma \ref{lemma:chaos-marginals} holds. It remains to obtain the analogue of Proposition \ref{prop:chaos}. Step 1 remains unchanged. In Step 2, when  dealing with the term $I_n$ one must do the following: 
In view of Lemma~\ref{lem:unifbound} (with $\rK=\infty$ now), one has 
\begin{align*}
\|K \ast (\mathbf{u}^n_s - \mathbf{u}_s) \|_{L^\infty(\R^d)} 
& \leq  C_{K,d} \|\mathbf{u}^n_s - \mathbf{u}_s\|_{\Lsolsing}.
\end{align*}
Now recall $\rKs$ and $\beta$ are fixed in \eqref{HK3sing}, and let $\gamma$ and $\delta$ satisfy \eqref{eq:gammadelta} and \eqref{eq:conddelta}, so that $\frac{d}{\rKs-\delta} <\gamma<\beta$. Then, use 
the Sobolev embedding $L^1\cap H_{\rKs-\varepsilon}^{\gamma}(\R^d) \subset \Lsolsing$ to get
$$\|K \ast (\mathbf{u}^n_s - \mathbf{u}_s) \|_{L^\infty(\R^d)} \leq C_{K,d} \|\mathbf{u}^n_s - \mathbf{u}_s\|_{ L^1\cap H_{\rKs-\varepsilon}^{\gamma}(\R^d)}.$$
Plug the latter in \eqref{eq:proofH1} to obtain
$$I_n \leq  C_{K,d}\int_s^t \|\mathbf{u}^n_\sigma - \mathbf{u}_\sigma\|_{ L^1\cap H_{\rK-\varepsilon}^{\gamma}(\R^d)} \, d\sigma. $$
By the interpolation inequality \eqref{eq:interpolationBessel}, 
\begin{align*}
\|\mathbf{u}^n_s - \mathbf{u}_s\|_{ L^1\cap H_{\rKs-\varepsilon}^{\gamma}(\R^d)} \leq \|\mathbf{u}^n_s - \mathbf{u}_s\|_{ L^1(\R^d)} + \|\mathbf{u}^n_s - \mathbf{u}_s\|_{ L^1(\R^d)}^\theta \, \|\mathbf{u}^n_s - \mathbf{u}_s\|_{ H_{\rKs}^{\beta}(\R^d)}^{1-\theta} ,
\end{align*}
for $\theta$ as in Section \ref{subsec:proofCor}. 
Now since $\mathbf{u}^n$ converges in $\mathcal{Y}$, and converges in particular weakly in $L^2 \left([0,T],H^\beta_{\rKs}(\R^d)\right)$, the uniform boundedness principle tells us that it is bounded in this space. Gathering this fact with the convergence in $\mathcal{C}([0,T]; L^1(\R^d))$ (by assumption), the previous inequality yields the convergence of $\mathbf{u}^n$ in $L^2 \left([0,T],L^1\cap H_{\rKs-\varepsilon}^{\gamma}(\R^d)\right)$. Hence $I_{n}$ converges to $0$.

For $II_n$ the arguments remains the same using the hypothesis \eqref{HK3sing} instead of \eqref{HK3} and having in mind that as $\mathbf{u}^n$ converges in $\mathcal{Y}$, $\mathbf{u}^n_t$ is uniformly bounded w.r.t. $n$ in $\Lsolsing$.  Then, the proof is finished. 

Here is a good place to explain why we could not define $\mathcal{Y}$ to be just $\mathcal{C}([0,T];\Lsolsing)$. Indeed, the computations for $I_n$ would have been straightforward in that case (as in the original proof). However,  this space is incompatible with the assumption \eqref{HK3sing}. Hence, when dealing with $II_n$, nothing would guarantee that $K\ast \mathbf{u}^n_t$ is continuous if $\mathbf{u}$ was in $\mathcal{C}([0,T];\Lsolsing)$. To ensure the latter, we must be in the position to apply  \eqref{HK3sing}. That is $\mathbf{u}^n_t$ must be in $L^{1}\cap H_{\rKs}^{\beta}(\R^d)$.  The reason we took the space $\mathcal{X}_w$ and not directly a space of type $L^2([0,T]; H^\beta_{\rKs}(\R^d))$ is that, by the proof of Corollary \ref{cor:weakConv}, we only have the convergence of $u^N$ in $\mathcal{X}_w$ and not in the strong topology of $L^2([0,T]; H^\beta_{\rKs}(\R^d))$. Hence, we would not have the tightness of $(\tilde{\Pi}^N, N\geq 1)$ in such space.

\section{Examples}\label{sec:examples}
In this section we first focus on the set of assumptions \eqref{HK} and we present a sufficient condition for them to be satisfied that will be used in practice. Then, we turn to specific singular kernels of Riesz type and discuss how our results can be applied to them. 

\subsection{A stronger, easier-to-check condition on the kernel}\label{app:CZ}
\label{sec:appendixKernel}
The first two points of Assumption \eqref{HK} are simple technical conditions and may not require specific comments, except that it would be interesting to lift the first integrability condition in order to be able to consider more singular kernels. The third assumption is much more interesting. We give with the following lemma an easier-to-check condition that in practice replaces \eqref{HK3}.
\begin{lemma} Assume that $K$ satisfies \eqref{HK1} and \eqref{HK2}. Assume further that \\
\noindent \begin{minipage}{0.06\linewidth}
\end{minipage}
\hspace{-0.5cm}
\begin{minipage}[t]{0.94\linewidth}
\raggedright
 \hspace{0.8cm} 
\begin{enumerate}[label= ]
\item(\customlabel{HK3'}{$\HypKKK_{iii}$}) ~ There exists  $\rK \geq \max(\pK',\qK')$ and $z\in [\pK\vee \qK,+\infty]\cap (d,+\infty]$ such that the matrix-valued kernel $\nabla K$ defines a convolution operator which is bounded component-wise from $\Lsol$ to $L^z(\R^d)$.
\end{enumerate}
\end{minipage}
\vspace{0.3cm}

Then, $K$ satisfies \eqref{HK3} with the same parameter $\rK$ and with $\Hreg = 1-\tfrac{d}{z}$.
\end{lemma}
\begin{proof}
First, we will make use of \eqref{HK1} and \eqref{HK2}.  
Young's convolution inequality states that for $\pK_{z}' := (1+\frac{1}{z}-\frac{1}{\pK})^{-1}$ and $\qK_{z}' := (1+\frac{1}{z}-\frac{1}{\qK})^{-1}$,  for any $f\in L^{\pK'_{z}}\cap L^{\qK'_{z}}(\R^d)$ we have
\begin{align*}
\|K \ast f \|_{L^z(\R^d)} &\leq \left\|\left(\mathbbm{1}_{\mathcal{B}_{1}}K \right) \ast f \right\|_{L^z(\R^d)} + \left\|\left(\mathbbm{1}_{\mathcal{B}_{1}^c} K\right) \ast f \right\|_{L^z(\R^d)} \\ 
&\leq \left\|\mathbbm{1}_{\mathcal{B}_{1}} K\right\|_{L^{\pK}(\R^d)} \|f\|_{L^{\pK'_{z}}(\R^d)} + \left\|\mathbbm{1}_{\mathcal{B}_{1}^c} K \right\|_{L^{\qK}(\R^d)} \|f\|_{L^{\qK'_{z}}(\R^d)}\\
&\leq C_{K} \|f\|_{L^{\pK'_{z}} \cap L^{\qK'_{z}}(\R^d)} .
\end{align*}
In particular the previous inequality holds true if $f\in \Lsol$, because $\pK'_{z} \leq \pK' \leq \rK$ (and similarly $\qK'_{z}\leq \rK$) and then by interpolation, $f$ is in $L^{\pK'_{z}} \cap L^{\qK'_{z}}(\R^d)$.
 Now in view of the previous fact and using the property \eqref{HK3'} of $\nabla K$, one deduces that if $f\in \Lsol$, then $K \ast f \in H^1_{z}(\R^d)$. Hence it follows from Morrey's inequality \cite[Th. 9.12]{Brezis} that there exists $C_{z,d}>0$ such that for any $f\in \Lsol$,
\begin{align*}
\mathcal{N}_{\eta}\left( K \ast f\right) &\leq  C  \| K\ast f\|_{H_{z}^{1}(\R^d)}  \nonumber\\
&\leq C_{z,d,K}  \|f\|_{\Lsol} ,
\end{align*}
where $\eta = 1-\tfrac{d}{z}$. Hence the desired result.
\end{proof}
A simple example that satisfies \eqref{HK3'} is a kernel $K$ such that $\nabla K$ is integrable. Then $\nabla K$ defines a convolution operator and by a convolution inequality, this operator is bounded in any $L^z(\R^d),~z\in[1,+\infty]$. Hence, by interpolation inequality, if $\rK\geq z$, one has 
$$\|\nabla K \ast f\|_{L^z(\R^d)} \leq C \|f\|_{L^z(\R^d)}\leq C \|f\|_{L^1\cap L^\rK(\R^d)}.$$
 As a consequence, $\nabla K$ satisfies \eqref{HK3'} for any $z\in [\pK\vee \qK,+\infty]\cap (d,+\infty]$ and any $\rK \geq \max(\pK',\qK',z)$.
 Hence it satisfies \eqref{HK3} with $\Hreg = 1-\frac{d}{z}$ and $\rK = \max(\pK',\qK',z)$.
 
 Nevertheless, for kernels $K$ such that $\nabla K$ is not integrable, it may be possible to define the convolution operator of kernel $\nabla K$ as the Principal Value integral acting on the space of smooth, rapidly decaying functions (i.e. the Schwartz space), thus defining a tempered distribution. In the next section we will see some interesting examples of such kernels.

\subsection{Riesz potentials}\label{subsec:RieszKernel}

In their general form, the Riesz potentials were defined in \eqref{eq:defRiesz}. 
 We denote the associated kernel by $K_{s}:=\pm  \nabla V_{s}$.  $K_{s}$ satisfies Assumption \eqref{HK}, provided that $d\geq 2$ and $s\in[0 , d-1)$: Indeed, in the order of the hypotheses in \eqref{HK}, we have that
\begin{itemize}
\item $K_{s}$ $\in L^1(\mathcal{B}_1)$ if and only if $0\leq s<d-1$. Moreover, $K_{s}\in L^p(\mathcal{B}_{1})$ for any $p<\tfrac{d}{s+1}$. Then one can choose $\pK = \left(\frac{d}{s+1}\right)^-$. 
\item $K_{s}\in L^q(\mathcal{B}_{1}^c)$ for any $q>\tfrac{d}{s+1}$, so one can choose $\qK = \left(\frac{d}{s+1}\right)^+$. 
\item 
\begin{itemize}
\item If $d\geq3$ and $s<d-2$, then $\nabla K_{s}$ is not bounded in any $L^p$ but it is  bounded from $L^{\tilde{z}}$ to $L^{z}$ whenever $\tilde{z}\in(1,\frac{d}{d-(s+2)})$ and $\tfrac{1}{z} = \tfrac{1}{\tilde{z}} + \tfrac{s+2}{d}-1$ (see \cite[Theorem 25.2]{SKM}). By letting $\tilde{z}$ be close enough to $\frac{d}{d(s+2)}$, one can always choose $z\geq \pK\vee \qK$ as large as desired, in particular one can find $z\geq \pK\vee \qK$, $z>d$ and $\rK \geq \max(\tilde{z},\pK',\qK')$
 to ensure property \eqref{HK3'} is verified. 
  Then \eqref{HK3} holds for $\Hreg = 1-\frac{d}{z}$ (see Section \ref{app:CZ}). Hence all our results can be applied to this kernel.

\item If $s=d-2$, then $\nabla K_{s}$ is a typical kernel satisfying the conditions of \cite[Chapter 4.4]{Grafakos}, and therefore it defines a bounded operator in any $L^z(\R^d),~z\in (1,\infty)$. Hence all our results apply to this particular kernel (see the previous discussion in Section \ref{app:CZ}). The choice of parameters will then be $\pK = \left(\frac{d}{d-1}\right)^-$, $\qK = \left(\frac{d}{d-1}\right)^+$, $\rK = z$ and $\Hreg = 1-\frac{d}{z}$, for some $z$ to be chosen in $(d,+\infty)$.

\item If $s\in [d-2,d-1)$, then one can verify (see e.g. \cite[Lemma 2.5]{Duerinckx}) that $\nabla K_{s}$ defines a convolution operator from $L^1\cap L^\infty\cap  \mathcal{C}^\sigma(\R^d)$ to $L^\infty(\R^d)$, with $\sigma\in (2-d+s,1)$:
\begin{align*}
 \|\nabla K_{s} \ast f\|_{L^\infty(\R^d)} &\lesssim \|f\|_{L^1(\R^d)} +\|f\|_{L^\infty(\R^d)} + \mathcal{N}_{\sigma}(f) \\
 &\lesssim \|f\|_{\Lsolsing} + \| f\|_{\beta,\rKs}\\
 &\lesssim \|f\|_{L^1\cap H^\beta_{\rKs}(\R^d)}  ,
\end{align*}
for some $\beta$ and $\rKs$ such that $\sigma = \beta-\frac{d}{\rKs}$, using the embedding \eqref{eq:GeneralMorrey} in the second inequality above, and \eqref{eq:embed1} in the third. Hence it follows that $\mathcal{N}_{1}(K\ast f) \leq  \|\nabla K_{s} \ast f\|_{L^\infty(\R^d)} \lesssim \|f\|_{L^1\cap H^\beta_{\rKs}(\R^d)}$, i.e. Assumption \eqref{HKsing} is satisfied for $\Hreg=1$, for any $\rKs>d$ and any $\beta\in (\frac{d}{\rKs},1)$ such that $\beta-\frac{d}{\rKs}\in (2-d+s,1)$.
\end{itemize}
\end{itemize} 

Hence, Theorem \ref{th:rate2} and its Corollaries \ref{cor:rateEmpMeas} and \ref{cor:wo-cutoff} are applicable for $s\leq d-2$ and Theorem~\ref{th:rate2sing} and its corollaries are aplicable for $s\in (d-2,d-1)$. In particular,
\begin{itemize}
\item For $s=d-2$ we have the following:
\begin{itemize}
\item Convergence holds for the widest possible range of parameter $\alpha$ by choosing $\rK$ as small as possible. This means here choosing $z=\rK=d^+$.
Under the constraint \eqref{C04}, one can therefore obtain the convergence for any $\alpha< \frac{1}{2(d-1)}$;
\item On the other hand, if one wishes to maximize the rate of convergence, then one must choose $z$ very large and $\rK = +\infty$, then the rate will be in $L^1\cap L^\infty$ norm and is $\varrho= (\frac{1}{2(d+1)})^-$ for $ \alpha = (\frac{1}{2(d+1)})^+$ and  $z$ very large. More precisely, the rate is given by
\begin{equation*}
\varrho = \min \left((1-\frac{d}{z})\alpha, \frac{1}{2} - \alpha d \right) ,
\end{equation*}
which gives $\varrho = \frac{1-d/z}{2 \left(d+1 - \frac{d}{z}\right)}$ for $\alpha = \frac{1}{2 \left(d+1 - \frac{d}{z}\right)}$, whose supremum is $\frac{1}{2(d+1)}$ for $z\to +\infty$.
\end{itemize}
\item For $s \in (d-2,d-1)$  we have the following:
\begin{itemize}
\item Convergence holds for the widest possible range of parameter $\alpha$ by maximizing the constraint $(d+2\beta + 2d(\frac{1}{2}-\frac{1}{\rKs}))^{-1} = (2(d+\sigma))^{-1}$ from the Assumption \eqref{C04sing}. Hence, letting $\sigma = (2-d+s)^+$, one can therefore obtain the convergence for any $\alpha< \frac{1}{2(s+2)}$;
\item On the other hand, if one wishes to maximize the rate of convergence, it is clear that one must choose $\alpha$ such that $\alpha\Hreg = \frac{1}{2}-\alpha d$ (see the definition of $\widetilde{\varrho}$). Since $\Hreg=1$, this gives $\widetilde{\varrho}\leq \frac{1}{2(d+1)}$
 and this maximum is attained for $\alpha = \frac{1}{2(d+1)}$.
\end{itemize}
\end{itemize}

Besides obtaining rates of convergence, Proposition \ref{prop:wpMP} proves the well-posedness of the McKean-Vlasov SDE \eqref{eq:McKeanVlasov} for all Riesz kernels with $s\in (0,d-1)$, which is new for these values, and most notably for the largest values $s\geq d-2$. The trajectorial propagation of chaos (Theorem \ref{th:propagation-of-chaos}) is also new for this whole class of particle systems.

\subsection{Classical kernels}
We present here the following important kernels that enter in the above framework for ${s=d-2}$: Coulomb, Keller-Segel and Biot-Savart kernels. In particular,  for these kernels we get the convergence for $\alpha < \frac{1}{2(d-1)}$ and the best possible rate is $\varrho= (\frac{1}{2(d+1)})^-$ for $ \alpha = (\frac{1}{2(d+1)})^+$. At the end of this section we present a fourth example of an attractive-repulsive kernel.

\paragraph{Coulomb kernel.}  The Coulomb interaction kernel is given by
$$K_{C} := -\nabla V_{d-2}, $$
where $V_{d-2}$ is defined in \eqref{eq:defRiesz}.
It is a generalisation in any dimension of the classical Coulomb force and an example of repulsive kernel, in the sense that
\begin{align}\label{eq:repulsive}
x\cdot K_C(x) \geq 0 , \quad \text{on the domain of definition of }K_C.
\end{align}
It models for instance the interaction between particles with identical electrical charges.

\paragraph{Parabolic-elliptic Keller-Segel models.}\label{subsec:KS}

An important and tricky example covered by this paper is the parabolic-elliptic Keller-Segel PDE, which takes the form \eqref{eq:PDE} with the kernel defined, for some $\chi>0$, by
\begin{align}\label{eq:KS-kernel}
K_{KS}(x) = -\chi \frac{x}{|x|^d}.
\end{align}

The difficulty in this model comes from the fact that the kernel is attractive, in the following sense:
\begin{align}\label{eq:attractive}
x\cdot K(x) < 0 , \quad \text{on the domain of definition of }K.
\end{align}
This leads to important issues that we discuss in more details in \cite{ORT}. Let us just mention as an example that in dimension $2$,  the PDE admits a global solution if and only if $\chi<8\pi$ (see e.g. \citet{Biler} for a recent review). Note that in that case ($d=2$ and $\chi<8\pi$) it is possible to choose a value of the cut-off $A_{T}$ independently of $T$ (see \cite{ORT}).

Theorem \ref{th:rate2}, Corollaries \ref{cor:rateEmpMeas} and \ref{cor:wo-cutoff} give rates of convergence of the particle system to the Keller-Segel PDE, even for solutions that exhibit a blow-up (with the parameters discussed above for $s=d-2$).

As a consequence of Theorem \ref{prop:wpMP}, we also deduce the local-in-time weak well-posedness of the McKean-Vlasov SDE \eqref{eq:McKeanVlasov} for all values of the concentration parameter $\chi$, and global-in-time weak well-posedness for $\chi<8\pi$, which is a novelty (see \cite{ORT} for a thorough discussion and comparison with previous results). We obtain the propagation of chaos result for our particle system to this SDE.

\paragraph{Biot-Savart kernel and the $2d$ Navier-Stokes equation.}

By considering the vorticity field $\xi$ associated to the incompressible two-dimensional Navier-Stokes solution $u$, one gets equation \eqref{eq:PDE} with the Biot-Savart kernel $K_{BS}(x) = \frac{1}{\pi} \frac{x^\perp}{|x|^2}$, where $x^\perp = 
\begin{pmatrix} -x_{2}\\x_{1}\end{pmatrix}$. The original Navier-Stokes solution is then recovered thanks to the formula $u_{t} = K_{BS}\ast \xi_{t}$. 

The Biot-Savart kernel is an example of repulsive kernel, in the sense of \eqref{eq:repulsive}. 
In this case, the Biot-Savart kernel is merely repulsive since $x\cdot K(x) = 0$.

It is well-known that with such kernel, Eq. \eqref{eq:PDE} has a unique global solution, and that $ {\|K\ast \xi_{t}\|_{L^\infty(\R^2)}}$ can be bounded by $C \left(1+ \|\xi_{0}\|_{L^\infty(\R^2)} \right)$ (see \cite{FlandoliOliveraSimon} and references therein). This permits to choose the cut-off value $A_{T}$ from  \eqref{eq:AT'} independently of $T$.

The kernel $K_{BS}$ is covered by our assumption \eqref{HK}, for the same reason as the Coulomb kernel with $d=2$ ($K_{0}(x) = \tfrac{1}{|x|}$  and we can choose $\alpha< \frac{1}{2}$). In that case, we improve within Theorem \ref{th:rate2} and Corollary \ref{cor:weakConv}  the Theorem 1.3 of \citet{FlandoliOliveraSimon}, by allowing $\alpha\in [\frac{1}{4},\frac{1}{2})$ and by providing a rate of convergence.

All the other results of this paper apply. In particular, if the initial condition is regular enough, we obtain a rate $\varrho$ in $L^1\cap L^\infty$ norm (the rate is maximized with $\rK=+\infty$ when $s=d-2$), and this rate is almost $\frac{1}{6}$.

\paragraph{Attractive-repulsive kernels.}

There is at least one other very interesting class of kernels that enters our framework. The attractive-repulsive kernels are attractive in a region of space, i.e. they satisfy \eqref{eq:attractive} on a subdomain $D$ of the domain of definition of $K$, and repulsive (i.e. satisfying \eqref{eq:repulsive}) on the complement of $D$.

The most famous example of such attractive-repulsive kernels might be the Lennard-Jones potential in molecular dynamics: this isotropic potential reads
\begin{align*}
V(x) = V_{0} \left(|x|^{-12} - |x|^{-6}\right) ,
\end{align*}
for some $V_{0}>0$. Then $K(x) = \nabla V(x)$ satisfies the first condition of \eqref{HK} (local integrability) only if the dimension is greater or equal to $14$, which may not be of the greatest physical relevance.

A similar, but less singular potential is proposed by \citet{FlandoliLeocataRicci} to model the adhesion of cells in biology. It can be expressed in general as
$$V(x) = V_{a}\, |x|^{-a} - V_{b}\, |x|^{-b},$$
 with $a,b>0$ and $V_{a}, V_{b}>0$. One can now refer to the discussion on Riesz kernels in Section \ref{subsec:RieszKernel} to determine the values of $a$ and $b$ that ensure the applicability of our results.

\vspace{1cm}

\appendix

\section*{Appendix}
\setcounter{section}{1}
\setcounter{equation}{0}
\label{sec:appendix}
\addcontentsline{toc}{section}{Appendix}

\subsection{Properties of the PDE and of the PDE with cut-off}\label{sec:PDE}

We start this section with some classical embeddings that will be used throughout the article. Then in Subsection \ref{subsec:PDE}, we derive some general inequalities and prove Proposition \ref{prop:local_existence}. In Subsection \ref{subsec:PDE_cutoff}, we prove that the PDE with cut-off \eqref{eq:cutoffPDE} can have at most one mild solution.

\subsubsection{Properties of mild solutions of the PDE}\label{subsec:PDE}

We recall that whenever  \eqref{HK1} and \eqref{HK2} are assumed to hold, we suppose $\rK \geq \max(\pK',\qK')$. For each $T>0$, let us consider the space 
\begin{align*}
\mathcal{X} := \mathcal{C}\left([0,T]; ~\Lsol\right),
\end{align*}
with the associated norm $ \|\cdot \|_{T,\Lsol}$, hereafter simply denoted by $\, \|\cdot\|_{\mathcal{X}}$.
The proof of existence of local solutions relies on the continuity of the following bilinear mapping, defined on $\mathcal{X}\times \mathcal{X}$ as
\begin{align*}
B: (u,v) \mapsto \left(\int_{0}^t \nabla \cdot e^{(t-s)\Delta} \left(u_{s} ~K\ast v_{s}\right)~ds \right)_{t\in[0,T]} .
\end{align*}

\begin{lemma}\label{lem:continuityB}
The bilinear mapping $B$ is continuous from $\mathcal{X}\times \mathcal{X}$ to $\mathcal{X}$.
\end{lemma}

\begin{proof}
We shall prove that there exists $C>0$ such that for any $u,v\in \mathcal{X}$ and any $t\in [0,T]$, 
\begin{align}\label{eq:continuityB}
\|B(u,v)(t)\|_{\Lsol} \leq C \sqrt{t}~\|u\|_{\mathcal{X}} \|v\|_{\mathcal{X}} ,
\end{align}
which suffices to prove the continuity of $B$. 

First, using the property \eqref{eq:heat-op-norm} of the Gaussian kernel with $p=\rK$, observe that for any $t\in [0,T]$,
\begin{align*}
\|B(u,v)(t)\|_{L^\rK(\R^d)} \leq \int_{0}^t \frac{C}{\sqrt{t-s}} \|u_{s} ~K\ast v_{s}\|_{L^\rK(\R^d)} ~ds, 
\end{align*}
and since $v_{s} \in \Lsol$, Lemma \ref{lem:unifbound} yields
\begin{align}\label{eq:BuvL1_2}
\|B(u,v)(t)\|_{L^\rK(\R^d)} &\leq \int_{0}^t \frac{C}{\sqrt{t-s}} \|K\ast v_{s}\|_{L^\infty(\R^d)}~ \|u_{s} \|_{L^\rK(\R^d)} ~ds \nonumber\\
&\leq C \sqrt{t} ~ \|u\|_{\mathcal{X}} \|v\|_{\mathcal{X}}.
\end{align}

On other hand we have, using similarly the property \eqref{eq:heat-op-norm} of the Gaussian kernel with $p=1$ and Lemma \ref{lem:unifbound}, that 
\begin{align}\label{eq:BuvLd}
\|B(u,v)(t)\|_{L^1(\R^d)} &\leq  \int_{0}^t \frac{C}{\sqrt{t-s}} \|u_{s} ~K\ast v_{s}\|_{L^1(\R^d)} ~ds \nonumber\\
&\leq \int_{0}^t \frac{C}{\sqrt{t-s}} \|u_{s}\|_{L^1(\R^d)} \|K\ast v_{s}\|_{L^{\infty}(\R^d)} ~ds
\nonumber\\&\leq  C \sqrt{t} ~ \|u\|_{\mathcal{X}} \|v\|_{\mathcal{X}}.
\end{align}

Therefore, combining Equations \eqref{eq:BuvL1_2}, \eqref{eq:BuvLd} one obtains \eqref{eq:continuityB}.
\end{proof}

The previous property of continuity of $B$ now provides the existence of a local mild solution by a classical argument.

\begin{proof}[Proof of Proposition \ref{prop:local_existence}]
For the existence part, note that for any $T>0$, $\mathcal{X}$ is a Banach space and that our aim is to find $T>0$ and $u\in \mathcal{X}$ such that
\begin{align*}
u_{t} = e^{t\Delta} u_{0} - B(u,u)(t),\quad \forall t\in[0,T].
\end{align*}
In view of Lemma \ref{lem:continuityB}, such a local mild solution is obtained by a standard contraction argument (Banach fixed-point Theorem).

~

We will prove the uniqueness in the (slightly more complicated) case of the cut-off PDE \eqref{eq:cutoffPDE}, for any value of the cut-off $A$ (see Proposition \ref{prop:cutPDE_unique0}). Admitting this result for now, let us observe that it implies uniqueness for the PDE without cut-off. Indeed, if $u^{1}$ and $u^{2}$ are mild solutions to \eqref{eq:PDE} on some interval $[0,T]$, then Lemma \ref{lem:unifbound} implies that for $i=1,2$, 
\begin{align*}
\|K\ast u^{i}\|_{T,L^\infty(\R^d)} &\leq C_{K,d} \|u^{i}\|_{T,\Lsol} <\infty.
\end{align*}
Thus $u^{1}$ and $u^{2}$ are also mild solutions to the cut-off PDE with $A$ larger than the maximum between $C_{K,d} \|u^{1}\|_{T,\Lsol}$ and $C_{K,d} \|u^{2}\|_{T,\Lsol}$. Hence the uniqueness result for the PDE with cut-off implies that $u^{1}$ and $u^{2}$ coincide.
\end{proof}

\begin{corollary}\label{cor:boundIC}
Assume that \eqref{HK1} and \eqref{HK2} hold. Let $C$ be the constant that appears in \eqref{eq:continuityB}. Then for $T>0$ such that 
\begin{align}\label{eq:condT0}
4 C \sqrt{T} ~\|u_0\|_{\Lsol}<1, 
\end{align}
one can define a local mild solution $u$ to \eqref{eq:PDE} up to time $T$ and 
\begin{align}\label{eq:boundu}
\|u\|_{T, \Lsol}\leq \frac{1-\sqrt{1-4 C\sqrt{T}~ \|u_0\|_{\Lsol}}}{2 C\sqrt{T} }.
\end{align}
For instance, assuming that $\|u_0\|_{\Lsol}\neq 0$ and choosing $T$ such that $4 C \sqrt{T}~ \|u_0\|_{\Lsol}=\frac{1}{2}$, one has 
\begin{equation*}
\|u\|_{T, \Lsol}< 4 \|u_0\|_{\Lsol}.
\end{equation*}
\end{corollary}

\begin{proof}
We rely on the bound \eqref{eq:continuityB}, in order to get that for $t\leq T$,
$$\|u\|_{t, \Lsol} \leq \|u_0\|_{\Lsol}+ C \sqrt{t}~ \|u\|^2_{t, \Lsol}.$$
Then by a standard argument (see e.g. Lemma 2.3 in \cite{Nagai11}), choosing $T>0$ which satisfies \eqref{eq:condT0} ensures that \eqref{eq:boundu} holds true.
\end{proof}

\begin{remark}
\label{remark:propertiesEDP}
Here we recall that if $u$ is a local mild solution on $[0,T]$ to \eqref{eq:PDE}, then $u$ preserves mass and sign: if $u_{0}\geq 0$, then $u_{t}\geq 0$ for any $t\in [0,T]$ and $\int_{\R^d}u_t(x) \ dx=\int_{\R^d} u_0(x) \ dx.$ 
The sign preservation is obtained by an argument similar to \cite[Proposition 2.7]{Nagai11} (although using \eqref{HK3} instead of the Poisson kernel). Then, the mass preservation is obtained by a standard integration-by-parts argument (see \cite[Proposition 2.4]{Nagai11} for a similar result in a different functional framework).
\end{remark}

\subsubsection{Properties of mild solutions of the PDE with cut-off}\label{subsec:PDE_cutoff}

We aim to prove the convergence of the mollified empirical measure to the following PDE with cut-off:
\begin{equation}\label{eq:cutoffPDE}
\begin{cases}
&\partial_t \widetilde{u}(t,x) = \Delta\widetilde{u}(t,x) -  \nabla \cdot \left(\widetilde{u}(t,x) F_{A} ( K \ast \widetilde{u}(t,x))\right),\quad t>0,~x\in\R^d\\
&\widetilde{u}(0,x) = u_0(x). 
\end{cases}
\end{equation}
Note that if $F_{A}$ is replaced by the identity function, one recovers \eqref{eq:PDE}. 
\begin{remark}\label{rk:mildsol}
Similarly to Definition \ref{def:defMild}, a mild solution to  \eqref{eq:cutoffPDE} satisfies Definition \ref{def:defMild}$(i)$ 
and solves 
\begin{equation}
\label{eq:mildKS-cutoff}
u_{t} =  e^{t\Delta} u_0 -  \int_0^t \nabla \cdot e^{(t-s)\Delta }  (u_{s} F_{A}(K \ast u_{s}))\ ds, \quad 0 \leq t \leq T.
\end{equation}
\end{remark}

In this section, we consider the cut-off PDE \eqref{eq:cutoffPDE} and its mild solution defined above. Here, $F_{A}$ is given in \eqref{eq:defF0}, but we denote it simply by $F$ for the sake of readability.

Note that due to the boundedness of the reaction term in \eqref{eq:cutoffPDE}, any mild solution will always be global. This global solution will be rigorously obtained as the limit of the particle system \eqref{eq:IPS}.

\begin{proposition}\label{prop:cutPDE_unique0}
Assume that $K$ satisfies Assumptions \eqref{HK1} and \eqref{HK2} and recall that $\rK\geq \max(\pK',\qK')$. Let $u_0\in \Lsol$. Then for any $A>0$ and $F$ defined in \eqref{eq:defF0}, there is at most one mild solution 
to the cut-off PDE \eqref{eq:cutoffPDE}. 
\end{proposition}

\begin{proof}
Assume there are two mild solutions $u^1$ and $u^2$ to \eqref{eq:cutoffPDE}. 
Then,
\begin{align*}
u^1_t - u^2_t &= -\int_0^t \nabla \cdot e^{(t-s)\Delta }  \left\{ u^1_s F( K\ast u^1_s) - u^2_s F(K\ast u^2_s ) \right\}~ds \\
&= -\int_0^t  \nabla \cdot e^{(t-s)\Delta }  \left\{ (u^1_s-u^2_s)F( K\ast u^1_s ) 
+u^2_s (F(K\ast u^1_s) - F( K\ast u^2_s))\right\}~ds .
\end{align*}
Hence there exists $C>0$ (that depends on $A$) such that 
\begin{align}\label{eq:deltarho}
\|u^1_t - u^2_t\|_{L^1(\R^d)} + \|u^1_t - u^2_t\|_{L^{\rK}(\R^d)} &\leq 
C \int_0^t  \frac{1}{\sqrt{t-s}} \left( \|u^1_s-u^2_s\|_{L^1(\R^d)}+ \|u^2_s~ K\ast(u^1_s-u^2_s)\|_{L^1(\R^d)} \right) ~ds \nonumber\\
& +C \int_0^t  \frac{1}{\sqrt{t-s}} \left( \|u^1_s-u^2_s\|_{L^{\rK}(\R^d)}+ \|u^2_s~ K\ast(u^1_s-u^2_s)\|_{L^{\rK}(\R^d)} \right) ~ds \nonumber\\
&\leq C \sqrt{t}~ \|u^1 - u^2\|_{t, \Lsol} \nonumber\\
&+ C \int_0^t \frac{\|u^2_s\|_{L^\rK(\R^d)} + \|u^2_s\|_{L^1(\R^d)}}{\sqrt{t-s}} ~\|K\ast \left(u^1_s - u^2_s\right)\|_{L^\infty(\R^d)} ~ds .
\end{align}
In view of Lemma \ref{lem:unifbound}, one has $\|K\ast (u^1_s - u^2_s)\|_{L^\infty(\R^d)} \leq C \|u^1_s - u^2_s\|_{\Lsol}$, thus 
plugging this upper bound in \eqref{eq:deltarho} gives
\begin{align*}
\|u^1_t - u^2_t\|_{L^1(\R^d)} +\|u^1_t - u^2_t\|_{L^\rK(\R^d)} \leq  \|u^1 - u^2\|_{t, \Lsol}C\sqrt{t} \left(1+   \|u^2\|_{t, \Lsol} \right).
\end{align*}

Hence for $t$ small enough, we deduce that $\|u^1 - u^2\|_{t,\Lsol}=0$. Therefore the uniqueness holds for mild solutions on $[0,t]$. Then by restarting the equation and using the same arguments as above, one gets uniqueness. 
\end{proof}

Finally, we get the following stability estimate in $H^\beta_{\rK}$ for the solution of the PDE.
\begin{proposition}\label{regularity}
Assume that $K$ satisfies Assumptions \eqref{HK1} and \eqref{HK2}. Let $u_0\in H_{\rK}^{\beta}(\R^{d})$ and $0\leq \beta<1$. Then for any $A>0$ and $F$ defined in \eqref{eq:defF0}, 
the unique solution  of the cut-off PDE \eqref{eq:cutoffPDE} verifies
\begin{equation*}
\sup_{t\in[0,T]} \|  u_{t}\|_{\beta, \rK}<\infty.
\end{equation*}
\end{proposition}

\begin{proof} 
Let $u$ be the unique solution of the cut-off PDE \eqref{eq:cutoffPDE}.  
Then using \eqref{eq:defFracLaplacian} and the fractional estimate \eqref{eq:heat-op-norm-frac}, we have 
\begin{align*}
\| u_t  \|_{\beta, \rK}  &\leq  \| e^{t\Delta} u_{0}\|_{\beta, \rK}  +
\int_0^t  \| \nabla \cdot e^{(t-s)\Delta }  \left\{ u_s \, F( K\ast u_s) \right\}\|_{\beta,\rK}
\, ds \\
&=  \| e^{t\Delta} u_{0}\|_{\beta, \rK}  +
\int_0^t  \| \left(I-\Delta\right)^{\frac{\beta}{2}} \nabla \cdot e^{(t-s)\Delta }  \left\{ u_s \, F( K\ast u_s) \right\}\|_{L^\rK(\R^d)}
\, ds \\
&\leq
\| u_{0}\|_{\beta, \rK}  +
  C \int_0^t  \frac{1}{(t-s)^{\frac{1+\beta}{2} }}  \| u_s\, F( K\ast u_s)\|_{L^\rK(\R^d)}
\, ds \\
&\leq
\| u_{0}\|_{\beta, \rK}  +
  C \int_0^t  \frac{1}{(t-s)^{\frac{1+\beta}{2} }}  \| u_s\|_{L^\rK(\R^d)}  \, ds\\
&\leq
\| u_{0}\|_{\beta, \rK}  +
  C  \int_0^t  \frac{1}{(t-s)^{\frac{1+\beta}{2} }}  \|  u_s\|_{\beta, \rK}  \, ds .
\end{align*}
 Since $\beta<1$, we deduce the desired estimate from Gr\"onwall's lemma. 
\end{proof}

\subsection{Time and space estimates of the stochastic convolution integrals}
\label{app:martingale}

In this section we study the moments of the supremum in time of $ \|M^N_{t}\|_{L^p(\R^d)}$, where the stochastic convolution integral $M^N_{t}$ was defined in \eqref{eq:defM}. Such estimates will be used in the proof of Theorem~\ref{th:rate2} for $p=1$ and $p=\infty$ only, but the result is established for any $p$ without any additional effort.

To achieve this, we use a generalization of the Burkholder-Davis-Gundy (BDG) inequality in UMD Banach spaces (see \citet{vNeervenEtAl}). There is a classical trick to apply  BDG-type inequalities to stochastic convolution integrals, however it only leads to a bound on  $\|  \|M^N_{t}\|_{L^{1}(\R^{d})} \|_{L^m(\Omega)} $ for a fixed $t>0$, instead of a bound on $\| \sup_{s\in [0,t] } \|M^N_{s}\|_{L^{1}(\R^{d})} \|_{L^m(\Omega)} $. In order to keep the supremum in time inside the expectation, we will also use the lemma of Garsia, Rodemich and Rumsey \cite{GRR}. Besides, there is an additional difficulty here which is that $L^1$ is not a UMD Banach space, hence the infinite-dimensional version of the BDG inequality cannot be applied directly.

In this section, the kernel $K$ is only assumed to be locally integrable. We keep this very general assumption (which is in particular implied by \eqref{HK1}-\eqref{HK2} or \eqref{HK1sing}-\eqref{HK2sing}), since we believe it might be useful in other contexts.

\begin{proposition}
\label{prop:martingale-bound}
Assume that the initial conditions satisfy \eqref{H} or \eqref{Hsing} and that $K\in L^1_{loc}(\R^d)$. 
Let $m\geq 1$, $p\in[1,+\infty]$ and $\varepsilon >0$. Then there exists $C>0$ such that for any $t\in [0,T]$ and $N\in\N^*$,
\begin{equation*}
\left\| \sup_{s\in [0,t] } \|M^N_{s}\|_{L^{p}(\R^{d})} \right\|_{L^{m}(\Omega)}  \leq C\,  N^{- \frac{1}{2}\left(1-\alpha (d+ \varkappa_{p}) \right) + \varepsilon} ,
\end{equation*}
where $\varkappa_{p} = \max \left(0, d(1-\tfrac{2}{p})\right)$.
\end{proposition}

The proof of Proposition \ref{prop:martingale-bound} relies on the following proposition, which we prove at the end of this section:
\begin{proposition}
\label{prop:martingale-bound-Lp}
For any $p\in [1,\infty ]$, any $m\geq 1$ and any $\delta\in (0,1]$, there exists $C>0$ such that
\begin{align*}
\left\| \|M^N_{t}-M^N_{s}\|_{L^{p}(\R^{d})} \right\|_{L^m(\Omega)}  \leq C\, (t-s)^{\frac{\delta}{2}} \, N^{-\frac{1}{2} \left( 1-\alpha(d + 6\delta+\varkappa_{p} ) \right)},\quad \forall s\leq t\in[0,T],~\forall N\in\N^*,
\end{align*}
where $\varkappa_{p}$ was defined in Proposition \ref{prop:martingale-bound}.
\end{proposition}
When $p\geq 2$, Proposition \ref{prop:martingale-bound-Lp} relies itself on the following result. Recall here the notation $\|\cdot\|_{\gamma}$ for the Hilbert norm $\|\cdot\|_{\gamma,2}$, which was defined in \eqref{eq:Hilbertbeta}.
\begin{proposition}
\label{prop:martingale-bound-Hbeta}
Let $\gamma \in \R$ and $m\geq 1$. For any $\delta\in(0,1]$, there exists $C>0$ such that
\begin{align*}
\left\| \|M^N_{t}-M^N_{s}\|_{\gamma} \right\|_{L^m(\Omega)}  \leq C\, (t-s)^{\frac{\delta}{2}} \, N^{-\frac{1}{2} \left( 1-\alpha(d + 4\delta+2\gamma) \right)} ,\quad \forall s\leq t\in[0,T],~\forall N\in\N^*.
\end{align*}
\end{proposition}

The final ingredient in the proof of Proposition \ref{prop:martingale-bound} is a consequence of Garsia-Rodemich-Rumsey's Lemma \cite{GRR}, given in the following lemma (for $\R$-valued processes, it already appears in \cite[Corollary 4.4]{RTY}, and the extension to Banach spaces is consistent with Garsia-Rodemich-Rumsey's Lemma with no additional difficulty, see e.g. \cite[Theorem A.1]{FrizVictoir}).
	\begin{lemma}\label{lem:GRR}
		Let $E$ be a Banach space and $(Y^n)_{n\ge 1}$ be a sequence of $E$-valued continuous processes on $[0,T]$.
		Let $m\geq 1$ and $\eta>0$ such that $m\eta>1$ and assume that there exists a constant $C_{0} > 0$
		and a sequence $(\delta_{n})_{n\ge 1}$ of positive real numbers such that
		\begin{equation*}
			\left(\EE \Big[ \big\|Y^n_{s} - Y^n_{t} \big\|_{E}^m \Big] \right)^\frac{1}{m}
			 \leq 
			C_{0}  |s-t|^\eta\, \delta_{n}, \quad \forall s,t\in[0,T], ~\forall n \ge 1 .
		\end{equation*}
		Then for any $m_{0}\in (0,m]$, there exists a constant 
		$C$,
		 depending only on $C_{0}$, $m$, $m_{0}$, $\eta$ and $T$, such that $\forall n \ge 1$,
		\begin{equation*}
			\left(\EE \Big[ \sup_{t\in[0,T]} \big\| Y^n_t-Y^n_0 \big\|_{E}^{m_{0}} \Big] \right)^\frac{1}{m_{0}}
			\leq
			C ~ \delta_{n}.
		\end{equation*}
	\end{lemma}

~

We are now ready to prove the main result of this section.
\begin{proof}[Proof of Proposition \ref{prop:martingale-bound}]
	We aim to apply Lemma \ref{lem:GRR} to  $M^N$ in the Banach space $L^p(\R^d)$, for some $p\in[1,+\infty]$. 
	Let  $\varepsilon >0$ and $m_{0}>0$. With the notations of Proposition \ref{prop:martingale-bound-Lp}, 
	let us choose $\delta = \frac{\varepsilon}{3\alpha}$,  
	 $\eta = \tfrac{\delta}{2}$ and $\delta_{N} = N^{-\rho}$ with $\rho = - \frac{1}{2}\left(1-\alpha (d+ \varkappa_{p}) \right)  + 3 \alpha\delta = - \frac{1}{2}\left(1-\alpha (d+ \varkappa_{p}) \right)  + \varepsilon$. Hence, 
	choosing $m \geq 1\vee m_{0}$ large enough so that $ m \eta >1$,
	the inequality in Proposition \ref{prop:martingale-bound-Lp} shows that $M^N$ satisfies the conditions of Lemma \ref{lem:GRR},
	and it follows that, for some constant $C > 0$,
	\begin{align*}
		\left\| \sup_{t\in[0,T]} \big\|M^N_t \big\|_{L^p(\R^d)} \right\|_{L^{m_{0}}(\Omega)}
		\leq
		C ~N^{- \frac{1}{2}\left(1-\alpha (d+ \varkappa_{p}) \right) + \varepsilon},
	\end{align*}
	which is the desired result.
\end{proof}

~

It remains to prove Propositions \ref{prop:martingale-bound-Hbeta} and \ref{prop:martingale-bound-Lp}.

\begin{proof}[Proof of Proposition \ref{prop:martingale-bound-Hbeta}]
Let $s\leq t$ and $x\in \R^d$. First, notice that 
\begin{equation}\label{eq:decompMN}
\begin{split}
|M^N_t(x) - M_s^N(x)| &\leq \left| \frac{1}{N} \sum_{i=1}^N \int_s^t \nabla e^{(t-u)\Delta} V^N(X_u^{i,N}-x) \cdot dW^i_u   \right| \\
&+ \left| \frac{1}{N} \sum_{i=1}^N  \int_0^s \nabla e^{(s-u)\Delta} \left[e^{(t-s)\Delta} V^N(X_u^{i,N}-x)- V^N(X_u^{i,N}-x)\right] \cdot dW^i_u   \right|  \\
&=: |I^N_{s,t}(x)| + |II^N_{s,t}(x)| .
\end{split}
\end{equation}
Thus, one has
\begin{equation}
\label{eq:starthere}
\left\| \|M^N_{t}-M^N_{s}\|_{\gamma} \right\|_{L^m(\Omega)}  \leq \left\| \|I^N_{s,t}\|_{\gamma} \right\|_{L^m(\Omega)}  + \left\| \|II^N_{s,t}\|_{\gamma} \right\|_{L^m(\Omega)}. 
\end{equation}

Let us formulate an important remark here. As the above semigroup acts as a convolution in time within the stochastic integral, $ \left(M^N_{t}\right)_{t\geq 0}$ is not a martingale, and neither are $I^N$ and $II^N$. Thus, we define the processes $\widetilde{I}^N$ and $\widetilde{II}^N$ in the following way: For any fixed $x\in \R^d$ and fixed $0\leq s< t$, set:
\begin{align*}
\forall r\in [s,t], \quad \widetilde{I}_{r}^N &:= \frac{1}{N} \sum_{i=1}^N \int_s^r \nabla e^{(t-u)\Delta} V^N(X_u^{i,N}-x) \cdot dW^i_u \, ,\\
\forall r\in [0,s],\quad \widetilde{II}_{r}^N &:= \frac{1}{N} \sum_{i=1}^N  \int_0^r \nabla e^{(s-u)\Delta} \left[e^{(t-s)\Delta} V^N(X_u^{i,N}-x)- V^N(X_u^{i,N}-x)\right] \cdot dW^i_u \, .
\end{align*} 
Now, $\{\widetilde{I}^N_{r}\}_{r\in [s,t]}$ and $\{\widetilde{II}^N_{r}\}_{r\in [0,s]}$ are martingales that take values in an infinite-dimensional space, and we have $\widetilde{I}^N_{t} = I^N_{s,t}$ and $\widetilde{II}^N_{s} = II^N_{s,t}$. 
Recall that the operators $(I-\Delta)^{\frac{\gamma}{2}},~\gamma\in\R$ were defined in the Notations section, see Equation \eqref{eq:defFracLaplacian}, with the relation $ \|(I-\Delta)^{\frac{\gamma}{2}} f\|_{L^2(\R^d)} =  \|f\|_{\gamma}$. 

~

We aim at evaluating the $L^2(\R^d)$ norm of $(I-\Delta)^{\frac{\gamma}{2}} (M^N_{t}  - M^N_{t}) = (I-\Delta)^{\frac{\gamma}{2}} \widetilde{I}^N_{t} + (I-\Delta)^{\frac{\gamma}{2}} \widetilde{II}^N_{s}$. 
Hence, applying the BDG inequality in UMD spaces \cite[Cor. 3.11]{vNeervenEtAl} to $(I-\Delta)^{\frac{\gamma}{2}} \widetilde{I}^N$ and $(I-\Delta)^{\frac{\gamma}{2}} \widetilde{II}^N$ in $L^2(\R^d)$ (which is UMD) yields
\begin{align}\label{eq:boundINbeta}
\big\|& \|I^N_{s,t}\|_{\gamma} \big\|_{L^m(\Omega)} = \big\| \|\widetilde{I}^N_{t}\|_{\gamma} \big\|_{L^m(\Omega)}\nonumber \\
& \leq C \left( \EE\left[ \bigg\|\left(\frac{1}{N^2} \sum_{i=1}^N \int_s^t  |(I-\Delta)^{\frac{\gamma}{2}} e^{(t-u)\Delta} \nabla V^N (X_u^{i,N}-\cdot)|^2 du\right)^{\frac{1}{2}}\bigg\|^{m}_{L^2(\R^d)}\right]\right)^{1/m} ,
\end{align}
and 
\begin{align*}
& \left\| \|II^N_{s,t}\|_{\gamma} \right\|_{L^m(\Omega)}= \big\| \|\widetilde{II}^N_{s}\|_{\gamma} \big\|_{L^m(\Omega)}\\
 &\leq C \left( \EE\left[ \bigg\|\left(\frac{1}{N^2} \sum_{i=1}^N \int_0^s  \left|(I-\Delta)^{\frac{\gamma}{2}} \nabla e^{(s-u)\Delta} [ e^{(t-s)\Delta} V^N (X_u^{i,N}-\cdot) - V^N (X_u^{i,N}-\cdot)\right|^2 ds\right)^{\frac{1}{2}}\bigg\|^{m}_{L^2(\R^d)}\right]\right)^{1/m}.
\end{align*}
Now as in the proof of Lemma 3.1 of \cite{FlandoliOliveraSimon}, one gets for arbitrary small $\delta>0$ that
\begin{align*}
\left\| \|I^N_{s,t}\|_{\gamma} \right\|_{L^m(\Omega)} \leq C N^{\frac{\alpha}{2} (d+2\delta+2\gamma)-\frac{1}{2}} \left( \int_s^t  \frac{1}{(t-u)^{1-\delta}} du \right)^{\frac{1}{2}} \lesssim (t-s)^{\frac{\delta}{2}} \, N^{\frac{\alpha}{2} (d+2\delta+2\gamma)-\frac{1}{2}},
\end{align*}
while for $II^N_{s,t}$ it comes
\begin{align}\label{pomoc1}
\left\| \|II^N_{s,t}\|_{\gamma} \right\|_{L^m(\Omega)} \leq \frac{C}{N^{\frac{1}{2}}} \left( \int_0^s \frac{1}{(s-u)^{1-\delta}} \|(I-\Delta)^{\frac{\gamma+\delta}{2}}(e^{(t-s)\Delta}V^N- V^N)\|_{L^2(\R^d)}^2  du\right)^{\frac{1}{2}} .
\end{align}
It is easy to obtain that, for $f\in H^1(\R^d)$,
$$\|e^{(t-s)\Delta}f- f\|^2_{L^2(\R^d)}\leq C \|\nabla f\|_{L^2(\R^d)}^2 (t-s).$$
Hence, choosing $f=(I-\Delta)^{\frac{\gamma+\delta}{2}} V^N$ and plugging the result of the previous inequality in \eqref{pomoc1} yields
\begin{align*}
\left\| \|II^N_{s,t}\|_{\gamma} \right\|_{L^m(\Omega)}  &\leq C \frac{(t-s)^{\frac{1}{2}}}{N^{\frac{1}{2}}}  \| (I-\Delta)^{\frac{\gamma+\delta}{2}}\nabla  V^N\|_{L^2(\R^d)} .
\end{align*}
Using now the equivalence of norms described in \cite[Eq. (3) p.59]{TriebelTh} (note that the space $F^s_{p,2}$ in \cite{TriebelTh} is the Bessel space used here), we get
\begin{align*}
\left\| \|II^N_{s,t}\|_{\gamma} \right\|_{L^m(\Omega)}  &\leq C \frac{(t-s)^{\frac{1}{2}}}{N^{\frac{1}{2}}} \| (I-\Delta)^{\frac{\gamma+\delta+1}{2}} V^N\|_{L^2(\R^d)} ,
\end{align*}
and using the estimate on $\|V^N\|_{\gamma+\delta+1}$ from \cite[Lemma 15]{FlandoliLeimbachOlivera}, it finally comes that
\begin{align}\label{eq:boundIINbeta1}
\left\| \|II^N_{s,t}\|_{\gamma} \right\|_{L^m(\Omega)}  &\leq C (t-s)^{\frac{1}{2}} N^{\frac{\alpha}{2}(d+ 2 \delta +2\gamma+2 ) - \frac{1}{2}} .
\end{align}
Although the regularity in $(t-s)$ is good in the previous inequality, we paid a factor $N^\alpha$ which will penalise too much the rest of the computations in Propositions \ref{prop:martingale-bound} and \ref{prop:martingale-bound-Lp}. Hence we also observe that 
\begin{align*}
\|(I-\Delta)^{\frac{\gamma+\delta}{2}}(e^{(t-s)\Delta}V^N- V^N)\|_{L^2(\R^d)} &\leq \|(I-\Delta)^{\frac{\gamma+\delta}{2}}\, e^{(t-s)\Delta}V^N\|_{L^2(\R^d)} + \|(I-\Delta)^{\frac{\gamma+\delta}{2}} V^N\|_{L^2(\R^d)}\\
&\leq 2 \|(I-\Delta)^{\frac{\gamma+\delta}{2}} V^N\|_{L^2(\R^d)} \leq C N^{\frac{\alpha}{2}(d+ 2 \delta +2\gamma)} .
\end{align*}
Thus, plugging this bound in \eqref{pomoc1} also gives
\begin{align}\label{eq:boundIINbeta2}
\left\| \|II^N_{s,t}\|_{\gamma} \right\|_{L^m(\Omega)} \leq C N^{\frac{\alpha}{2}(d+ 2 \delta +2\gamma) - \frac{1}{2}} .
\end{align}
Hence, one can interpolate between \eqref{eq:boundIINbeta1} and \eqref{eq:boundIINbeta2} to obtain that for any $\epsilon\in[0,1]$,
\begin{align*}
\left\| \|II^N_{s,t}\|_{\gamma} \right\|_{L^m(\Omega)} \leq C (t-s)^{\frac{\epsilon}{2}} N^{\frac{\alpha}{2}(d+ 2 \delta +2\gamma)- \frac{1}{2} + \alpha\epsilon} .
\end{align*}
So the bound \eqref{eq:boundINbeta} for $I^N_{s,t}$ and the previous inequality plugged in \eqref{eq:starthere} and applied to $\epsilon=\delta$ yield the desired  inequality.
\end{proof}

\begin{proof}[Proof of Proposition \ref{prop:martingale-bound-Lp}]
This proof will be divided in two according to whether $p\geq 2$ or $p=1$, in that order. The result for $p\in(1,2)$ can be obtained by interpolation between the two previous cases.

~

\hspace{0.5cm}$\bullet$ \emph{First, assume that}  $p\in[2,+\infty]$. \\
Define, for some $\delta>0$ small enough,
\begin{equation*}
\gamma  := 
d \left(\frac{1}{2} - \frac{1}{p}\right) +\frac{\delta}{2}  ,
\end{equation*}
with the convention $\frac{1}{p}=0$ if $p=\infty$. 
In view of the Sobolev embedding of $H^{\gamma}$ into  $L^{p}$ (which holds because $p\geq 2$, see \cite[Theorem 1.66]{BCD}), one has
\begin{align*}
\left\|  \|M^N_{t} - M^N_{s}\|_{L^{p}(\R^{d})} \right\|_{L^m(\Omega)} \leq \left\|  \|M^N_{t} - M^N_{s}\|_{\gamma} \right\|_{L^m(\Omega)}.
\end{align*}
Thus Proposition \ref{prop:martingale-bound-Hbeta} yields the result in the case $p\geq 2$:
\begin{equation*}
\left\|  \|M^N_{t} - M^N_{s}\|_{L^{p}(\R^{d})} \right\|_{L^m(\Omega)}  \leq C\, (t-s)^{\frac{\delta}{2}} \, N^{-\frac{1}{2} \left( 1-\alpha(d + 5\delta+2 d(\frac{1}{2} - \frac{1}{p}) ) \right)} .
\end{equation*}

~

\hspace{0.5cm} $\bullet$ \emph{Assume now that $p=1$.} 

$L^1(\R^d)$ is not a UMD Banach space, so in the case $p=1$, we proceed as follows. 

Recall the decomposition in \eqref{eq:decompMN} of $M^N_{t} - M^N_{s} = I^N_{s,t} + II^N_{s,t}$. Apply Cauchy-Schwarz's inequality with weights $(1+|x|)^{-\frac{d+1}{2}}\times (1+|x|)^{\frac{d+1}{2}}$ to get:
\begin{align}\label{eq:decompI^NII^N}
\|M^N_{t} - M^N_{s}\|_{L^1(\R^d)} &\leq \|I^N_{s,t}\|_{L^1(\R^d)} + \|II^N_{s,t}\|_{L^1(\R^d)}\nonumber\\
&\leq C  \left(  \|(1+|\cdot|)^{\frac{d+1}{2}}\, I^N_{s,t}\|_{L^{2}(\R^{d})} + \|(1+|\cdot|)^{\frac{d+1}{2}}\, II^N_{s,t}\|_{L^{2}(\R^{d})} \right) .
\end{align}
We will treat separately the two terms involving $I^N$ and $II^N$.

~

\emph{Consider first the term $I^N$ in Equation \eqref{eq:decompI^NII^N}.} \\
Using again the BDG inequality (as explained in the proof of Proposition \ref{prop:martingale-bound-Hbeta}),
 \begin{align*}
 \left\| \|(1+|\cdot|)^{\frac{d+1}{2}}\,  I_{s,t}^N\|_{L^{2}(\R^{d})} \right\|_{L^m(\Omega)} &\leq \frac{C}{N}  \left\| \bigg\|\left( (1+|\cdot|)^{d+1} \sum_{i=1}^N \int_s^t  |\nabla e^{(t-u)\Delta} \, V^N (X_u^{i,N}-\cdot)|^2 du\right)^{1/2}\bigg\|_{L^2(\R^d)}\right\|_{L^m(\Omega)} \\
&= \frac{C}{N}  \left\| \left(  \sum_{i=1}^N \int_s^t \int_{\R^d} (1+|x|)^{d+1}\, |\nabla e^{(t-u)\Delta} \, V^N (X_u^{i,N}-x)|^2 \, dx\,  du\right)^{1/2}  \right\|_{L^m(\Omega)}  .
 \end{align*}
By the simple inequality $(1+|a+b|) \leq (1+|a|) (1+|b|)$ and Fubini's theorem, we then have
\begin{align}
\label{eq:martingaleA2}
& \left\| \|(1+|\cdot|)^{\frac{d+1}{2}}\,  I_{s,t}^N\|_{L^{2}(\R^{d})} \right\|_{L^m(\Omega)}  \nonumber\\
&\leq \frac{C}{N}  \left\| \left(  \sum_{i=1}^N \int_s^t \int_{\R^d} \left(1+|X^{i,N}_{s}|\right)^{d+1} \left(1+|X^{i,N}_{s}-x|\right)^{d+1}\, |\nabla e^{(t-u)\Delta} \, V^N (X_u^{i,N}-x)|^2 \, dx\,  du\right)^{1/2}  \right\|_{L^m(\Omega)} \nonumber \\
&\leq \frac{C}{N}  \left\| \left(  \sum_{i=1}^N \int_s^t \left(1+|X^{i,N}_{s}|\right)^{d+1}  \int_{\R^d} \left(1+|y|\right)^{d+1}\, |\nabla e^{(t-u)\Delta} \, V^N (y)|^2 \, dy\,  du\right)^{1/2}  \right\|_{L^m(\Omega)} \nonumber \\
&\leq \frac{C}{N} \left(\int_{s}^t \int_{\R^d}  \left(1+|y|\right)^{d+1} \, \left\vert \nabla
e^{(t-u) \Delta}  V^{N} (y) \right\vert^{2} dy\, du \right)^{\frac{1}{2}}  \left\|  \left(\sum_{i=1}^{N} \left(1+\sup_{u\in[s,t]} |X^{i,N}_{s}|\right)^{d+1} \right)^{\frac{1}{2}}\right\|_{L^m(\Omega)}  ,
\end{align}
performing a simple change of variables in the second inequality. 
Since $X^{i,N}$ is a diffusion with bounded coefficients (uniformly in $N$), a classical argument gives that for any $p>0$, there exists a constant $C>0$ which depends only on $p$ and $T$ such that $\EE \, \sup_{s\in[0,T]} |X^{i,N}_{s}|^p \leq C$. It follows that
\begin{align}
\label{eq:martingaleA2-1}
\left\|  \left(\sum_{i=1}^{N} \left(1+\sup_{u\in[s,t]} |X^{i,N}_{s}|\right)^{d+1} \right)^{\frac{1}{2}}\right\|_{L^m(\Omega)}\leq C\, N^{\frac{1}{2}}  .
\end{align}
Equations \eqref{eq:martingaleA2} and \eqref{eq:martingaleA2-1} now give
\begin{align}
 \left\| \|(1+|\cdot|)^{\frac{d+1}{2}}\, I_{s,t}^N\|_{L^{2}(\R^{d})} \right\|_{L^m(\Omega)} &\leq \frac{C}{\sqrt{N}} \left( \int_s^t \int_{\R^d} (1+|x|)^{d+1}\, |\nabla e^{(t-u)\Delta} \, V^N (x)|^2 \, dx\,  du\right)^{1/2}  \nonumber\\
 &\leq \frac{C}{\sqrt{N}} \left( \int_s^t \int_{\R^d} |\nabla e^{(t-u)\Delta} \, V^N (x)|^2 \, dx\,  du\right)^{1/2}  \label{eq:IN1} \\
 &\quad \quad + \frac{C}{\sqrt{N}} \left( \int_s^t \int_{\R^d} |x|^{d+1}\, |\nabla e^{(t-u)\Delta} \, V^N (x)|^2 \, dx\,  du\right)^{1/2}  . \label{eq:IN2}
\end{align}
Let $\delta>0$. Proceeding as in the first part of this proof with $p=2$ and the Sobolev embedding of $H^0$ in $L^2$ ($H^0 = L^2$), we obtain for the term \eqref{eq:IN1} of the previous sum:
\begin{align*}
\frac{C}{\sqrt{N}} \left( \int_s^t \int_{\R^d} |\nabla e^{(t-u)\Delta} \, V^N (x)|^2 \, dx\,  du\right)^{1/2}  \leq C \, (t-s)^{\frac{\delta}{2}} \, N^{-\frac{1}{2} \left(1-\alpha(d+5\delta)\right)} .
\end{align*}

Consider now the term \eqref{eq:IN2}: We will bound it successively in two different ways, by applying first the gradient to $V^N$ and then leaving it on the heat kernel, and finally interpolating between these two results.

First, the Cauchy-Schwarz inequality yields
\begin{align*}
\int_{\R^d}  |x|^{d+1} \, \left\vert \nabla
e^{(t-u) \Delta}  V^{N} (x) \right\vert^{2} dx  &=  \int_{\R^d}  |x|^{d+1} \, \left\vert \int_{\R^d}  g_{2(t-u)}(x-y) \,  \nabla V^{N} (y) \, dy \right\vert^{2} dx\\
&\leq \int_{\R^d}  |x|^{d+1}  \int_{\R^d}  \left|\nabla V^{N} (y)\right|^2\,  g_{2(t-u)}(x-y) \, dy\,   dx ,
\end{align*}
where we recall that the notation $g$ for the heat kernel was introduced in \eqref{eq:def-heat-kernel}. 
Then by Fubini's theorem and a simple change of variables,
\begin{align}\label{eq:martingaleA2-2}
\int_{\R^d}  |x|^{d+1} \, \left\vert \nabla
e^{(t-u) \Delta}  V^{N} (x) \right\vert^{2} dx &\leq N^{d\alpha +2\alpha} \int_{\R^d}   |\nabla V(y)|^2   \int_{\R^d} |x+\frac{y}{N^\alpha}|^{d+1}\,  g_{2(t-u)}(x)\,  dx\, dy \nonumber\\
 &\leq C N^{d\alpha +2\alpha}  \left(  \int_{\R^d}   |\nabla V(y)|^2\, |\frac{y}{N^\alpha}|^{d+1} \, dy + \int_{\R^d} |x|^{d+1}\, g_{2(t-u)}(x) \, dx  \right)\nonumber\\
 &\leq C N^{d\alpha +2\alpha}  \left( N^{-\alpha(d+1)} + (t-u)^{\frac{d+1}{2}} \right)\nonumber\\
&\leq C N^{d\alpha + 2\alpha} \,  .
\end{align}

Second, another way to treat the above quantity is to leave the derivative on the Gaussian density. Notice that $\nabla g_{2(t-u)}(x)=C \frac{x}{t-u} \, g_{2(t-u)}(x)$. Hence, using the Cauchy-Schwarz inequality,
\begin{align*}
\left\vert \nabla
e^{(t-u) \Delta}  V^{N} (x) \right\vert^{2}&= \frac{C}{(t-u)^2}\left|\int_{\R^d} (x-y)\, g_{2(t-u)}(x-y)\, V^N(y) \, dy\right|^2 \\
&\leq \frac{C}{(t-u)^2}\int_{\R^d} \left|x-y\right|^2 \, V^N(y)^2\,  g_{2(t-u)}(x-y) \, dy .
\end{align*}
After applying Fubini's theorem and the inequality $|x|^{d+1}\lesssim |y|^{d+1} + |x-y|^{d+1}$,
 this gives
\begin{align}\label{eq:bound2}
\int_{\R^d}  |x|^{d+1} \, \left\vert \nabla
e^{(t-u) \Delta}  V^{N} (x) \right\vert^{2} dx &\leq  \frac{C}{(t-u)^2} \int_{\R^d} V^N(y)^2 \int_{\R^d} |x|^{d+1} \, \left|x-y\right|^2 g_{2(t-u)}(x-y)  \, dx\, dy \nonumber\\
&\leq \frac{C}{(t-u)^2} \bigg( \int_{\R^d}  |y|^{d+1}\,  (V^N(y) )^2 \, dy \int_{\R^d} \left|x\right|^2  g_{2(t-u)}(x) \, dx \nonumber\\
&\hspace{2cm} + \int_{\R^d}  (V^N(y) )^2 \, dy \int_{\R^d} \left|x\right|^{d+3}  g_{2(t-u)}(x) \, dx \bigg) \nonumber\\
&\leq \frac{C}{(t-u)^2} \left( N^{-\alpha} \, (t-u) + N^{d\alpha} \, (t-u)^{\frac{d+3}{2}}\right)\nonumber\\
&\leq C\, N^{d\alpha}\, (t-u)^{-1}.
\end{align}
Interpolating between the bounds \eqref{eq:martingaleA2-2} and \eqref{eq:bound2}, respectively with powers $\frac{\delta}{2}$ and $1-\frac{\delta}{2}$, we get
\begin{align*}
\int_{s}^t \int_{\R^d}  |x|^{d+1} \, \left\vert \nabla
e^{(t-u) \Delta}  V^{N} (x) \right\vert^{2} dx\, du &\leq C\, N^{\alpha (d+\delta)} \int_{s}^t  (t-u)^{-(1-\frac{\delta}{2})}\, du \\
&\leq C\, N^{\alpha (d+\delta)}\, (t-s)^{\frac{\delta}{2}}.
\end{align*}
Hence, in view of this bound and the bound on \eqref{eq:IN1}, we have finally obtained that for any $\delta>0$, there exists $C>0$ such that
\begin{align}\label{eq:bound(1+x)IN}
\left\| \|(1+|\cdot|)^{\frac{d+1}{2}}\,  I_{s,t}^N\|_{L^{2}(\R^{d})} \right\|_{L^m(\Omega)} \leq C \, (t-s)^{\frac{\delta}{2}} \, N^{-\frac{1}{2} \left(1-\alpha(d+5\delta)\right)}.
\end{align}

~

\emph{Consider now the term $II^N$ in Equation \eqref{eq:decompI^NII^N}.} \\
Using again the BDG inequality (as explained in the proof of Proposition \ref{prop:martingale-bound-Hbeta}),
\begin{multline*}
\left\| \|(1+|\cdot|)^{\frac{d+1}{2}} II_{s,t}^N\|_{L^{2}(\R^{d})} \right\|_{L^m(\Omega)} \\
\quad \leq \frac{C}{N}  \left\| \left(\sum_{i=1}^N \int_0^s \int_{\R^d} (1+|x|)^{d+1}   \left|\nabla e^{(s-u)\Delta} \left[e^{(t-s)\Delta} V^N(X_u^{i,N}-x)- V^N(X_u^{i,N}-x)\right] \right|^2  dx\, du \right)^{1/2}  \right\|_{L^m(\Omega)}.  
\end{multline*}
Performing the same computations \eqref{eq:martingaleA2}-\eqref{eq:martingaleA2-1} as for the term $I^N$, one gets
\begin{equation}
\label{eq:A2-1}
 \left\| \| (1+|\cdot|)^{\frac{d+1}{2}} \, II_{s,t}^N\|_{L^{2}(\R^{d})} \right\|_{L^m(\Omega)} \leq \frac{C}{\sqrt{N}}  \left(\int_{0}^s \int_{\R^d}  \left(1+|x|\right)^{d+1} \, \left|\nabla e^{(s-u)\Delta} \left[e^{(t-s)\Delta} V^N(x)- V^N(x)\right] \right|^2 dx\, du  \right)^{\frac{1}{2}}  .
\end{equation}
Here again, we will estimate the previous integral in two different ways and then interpolate. 

Notice first that
$$\left|\nabla e^{(s-u)\Delta} \left[e^{(t-s)\Delta} V^N(x)- V^N(x)\right] \right|^2\leq \left|\nabla e^{(t-u)\Delta} V^N(x)\right|^2 +  \left|\nabla e^{(s-u)\Delta} V^N(x) \right|^2.$$
Hence, 
\begin{align*}
\int_{\R^d}  \left(1+|x|\right)^{d+1} \, &\left|\nabla e^{(s-u)\Delta} \left[e^{(t-s)\Delta} V^N(x)- V^N(x)\right] \right|^2 dx \\
&\leq \int_{\R^d}  \left(1+|x|\right)^{d+1} \left[\left|\nabla e^{(t-u)\Delta} V^N(x)\right|^2 +  \left|\nabla e^{(s-u)\Delta} V^N(x) \right|^2\right] dx\\
&=:A + B.
\end{align*} 
The two terms $A$ and $B$ can be treated in the same manner as in \eqref{eq:bound2}. This leads us to the following inequalities:
$$A\leq C\frac{N^{d\alpha }}{(t-u)} ~ \text{ and } ~ B\leq C\frac{N^{d\alpha }}{(s-u)}.$$
Hence, as $s<t$, on the one hand we have
\begin{equation}
\label{eq:boundIIN-1}
\int_{\R^d}  \left(1+|x|\right)^{d+1} \, \left|\nabla e^{(s-u)\Delta} \left[e^{(t-s)\Delta} V^N(x)- V^N(x)\right] \right|^2 dx \leq C\frac{N^{d\alpha }}{(s-u)}.
\end{equation}

On the other hand, applying the gradient on $V^N$, one gets
\begin{align*}
\left|\nabla e^{(s-u)\Delta} \left[e^{(t-s)\Delta} V^N(x)- V^N(x)\right] \right| &= \left| g_{2(s-u)}\ast \left[g_{2(t-s)}\ast \nabla V^N(x)- \nabla V^N(x)\right] \right|\\
&\leq \int_{\R^d} g_{2(s-u)}(x-y) \int_{\R^d} g_{2(t-s)}(z) \left| \nabla V^N(y-z) -\nabla V^N(y) \right| \, dz\, dy \\
&\leq \|\nabla^2 V^N\|_{\infty}  \int_{\R^d} g_{2(s-u)}(x-y) \int_{\R^d} g_{2(t-s)}(z)\, | z| \, dz\, dy .
\end{align*}
Since $\|\nabla^2 V^N\|_{\infty} \lesssim N^{d\alpha +2\alpha}$ (using the regularity of $V$), it comes
\begin{align*}
\left|\nabla e^{(s-u)\Delta} \left[e^{(t-s)\Delta} V^N(x)- V^N(x)\right] \right| &\leq C\, N^{d\alpha +2\alpha} \, \sqrt{t-s}.
\end{align*}
Hence 
\begin{align*}
\int_{\R^d}  \left(1+|x|\right)^{d+1} &\, \left|\nabla e^{(s-u)\Delta} \left[e^{(t-s)\Delta} V^N(x)- V^N(x)\right] \right|^2 dx \\
&\leq C\, N^{d\alpha +2\alpha} \, \sqrt{t-s}  \int_{\R^d}  \left(1+|x|\right)^{d+1} \, \left|\nabla e^{(s-u)\Delta} \left[e^{(t-s)\Delta} V^N(x)- V^N(x)\right] \right| dx \\
&\leq C\, N^{d\alpha +2\alpha} \, \sqrt{t-s} \, \bigg\{ \int_{\R^d}  \left(1+|x|\right)^{d+1} \, \left| e^{(t-u)\Delta} \nabla V^N(x) \right| dx \\
&\hspace{3.3cm} + \int_{\R^d}  \left(1+|x|\right)^{d+1} \, \left| e^{(s-u)\Delta} \nabla V^N(x) \right| dx \bigg\}
\end{align*}
and by a change of variables,
\begin{align*}
\int_{\R^d}  &\left(1+|x|\right)^{d+1} \, \left|\nabla e^{(s-u)\Delta} \left[e^{(t-s)\Delta} V^N(x)- V^N(x)\right] \right|^2 dx \\
&\leq C\, N^{d\alpha +3\alpha} \, \sqrt{t-s} \, \bigg\{ \int_{\R^d}  \left(1+|\frac{x}{N^\alpha}|\right)^{d+1} \, \left| e^{(t-u)\Delta} \nabla V(x) \right| dx + \int_{\R^d}  \left(1+|\frac{x}{N^\alpha}|\right)^{d+1} \, \left| e^{(s-u)\Delta} \nabla V(x) \right| dx \bigg\}\\
&\leq C\, N^{d\alpha +3\alpha} \, \sqrt{t-s} \, \bigg\{ \int_{\R^d}  \left(1+|x|\right)^{d+1} \, \left| e^{(t-u)\Delta} \nabla V(x) \right| dx + \int_{\R^d}  \left(1+|x|\right)^{d+1} \, \left| e^{(s-u)\Delta} \nabla V(x) \right| dx \bigg\} .
\end{align*}
Observe that for any $\sigma>0$, using the inequality $(1+|x|) \leq (1+|y|)\, (1+|x-y|)$, one gets
\begin{align*}
 \int_{\R^d}  \left(1+|x|\right)^{d+1} \, \left| e^{\sigma\Delta} \nabla V(x) \right| dx \leq \int_{\R^d} (1+|y|)^{d+1}\, |\nabla V(y)| \int_{\R^d} (1+|x-y|)^{d+1}\, g_{2\sigma}(x-y)\, dx \, dy ,
\end{align*}
which is bounded by a constant independent of $\sigma$. Thus
\begin{align}\label{eq:boundIIN-2}
\int_{\R^d}  &\left(1+|x|\right)^{d+1} \, \left|\nabla e^{(s-u)\Delta} \left[e^{(t-s)\Delta} V^N(x)- V^N(x)\right] \right|^2 dx \leq  C\, N^{d\alpha +3\alpha} \, \sqrt{t-s}.
\end{align}
Now for $\delta>0$, one can interpolate between \eqref{eq:boundIIN-1} and \eqref{eq:boundIIN-2} to get:
\begin{align*}
\int_{\R^d}  &\left(1+|x|\right)^{d+1} \, \left|\nabla e^{(s-u)\Delta} \left[e^{(t-s)\Delta} V^N(x)- V^N(x)\right] \right|^2 dx \leq  C\, N^{\alpha (d+ 6\delta)} \frac{(t-s)^{\delta}}{(s-u)^{1-2\delta}} .
\end{align*}
Plugging this bound into \eqref{eq:A2-1} now provides
\begin{align}\label{eq:bound(1+x)IIN}
 \left\| \| (1+|\cdot|)^{\frac{d+1}{2}} \, II_{s,t}^N\|_{L^{2}(\R^{d})} \right\|_{L^m(\Omega)} \leq C\, N^{-\frac{1}{2} \left(1-\alpha(d+6\delta)\right)} (t-s)^{\frac{\delta}{2}} .
\end{align}

~

Putting together Equations \eqref{eq:bound(1+x)IN} and \eqref{eq:bound(1+x)IIN}, we deduce from \eqref{eq:decompI^NII^N} that 
\begin{align*}
\|M^N_{t} - M^N_{s}\|_{L^1(\R^d)} \leq C\, N^{-\frac{1}{2} \left(1-\alpha(d+6\delta)\right)} (t-s)^{\frac{\delta}{2}} .
\end{align*}

~

Finally for $p\in[1,2]$, we obtain the same estimate by interpolating between the $L^1$ norm and the $L^2$ norm.
\end{proof}

\subsection{Proof of the boundedness estimate}\label{app:tightness}

\begin{proposition}\label{prop:bound1}  
Let the assumptions of Proposition \ref{prop:martingale-bound} hold.
Let $q \geq 1$ and $\beta \in [0, 1)$. We assume that $\sup_{N\in\N^*} \mathbb{E} \| u_{0}^{N} \|_{\beta,\rK}^{q}< \infty$ and that 
\begin{equation}
\label{eq:alfa_appendix}
0<\alpha<\frac{1}{d+2 \beta + 2 d (\frac{1}{2} - \frac{1}{\rK})\vee 0}.
\end{equation} 
Then
\begin{equation*}
\sup_{N\in\N^*}  \, \mathbb{E}\left[\sup_{t\in[0,T]} \left\Vert  u^N_{t} \right\Vert _{\beta,\rK}^{q}\right]  <\infty.
\end{equation*}
\end{proposition}
\begin{proof}
\textbf{Step 1}. Recall that the operators $(I-\Delta)^\beta,~\beta\in\R$ were defined in the Notations section, see Equation \eqref{eq:defFracLaplacian}, with a clear link with the Sobolev norm $ \|\cdot\|_{\beta,\rK}$. 

Let $F$ stand for the function $F_{A}$ defined in \eqref{eq:defF0}. From \eqref{eq:mildeq} after applying $(\mathrm{I}-\Delta)^{\frac{\beta}{2}}$ and by the triangular inequality we have
\begin{align}
\Big\Vert \left(  \mathrm{I}-\Delta\right)^{\frac{\beta}{2}} u^N_t&\Big\Vert _{L^{\rK}(\R^d) }  
\leq\left\Vert \left(  \mathrm{I}-\Delta\right)  ^{\frac{\beta}{2}}e^{t\Delta}u^N_0\right\Vert _{L^{\rK}(\R^d)} \label{eq:first} \vphantom{\Bigg(}\\
&  + \int_{0}^{t}\left\Vert \left( \mathrm{I} -\Delta\right)  ^{\frac{\beta}{2}}\nabla \cdot
e^{\left(  t-s\right)  \Delta}\left(  V^{N} \ast\left( F(K \ast u^N_s) \mu_{s}^{N}\right)  \right)  \right\Vert _{ L^{\rK}(\R^d)}  ds \label{eq:second} \vphantom{\Bigg(}\\
&  +\bigg\Vert \frac{1}{N}\sum_{i=1}^{N}\int_{0}^{t}\left(  \mathrm{I}-\Delta\right)^{\frac{\beta}{2}}\nabla e^{\left(  t-s\right)  \Delta}\left(  
V^{N} \left(X_{s}^{i,N} - \cdot\right)  \right) \cdot dW_{s}^{i}\bigg\Vert _{L^{\rK}(\R^d)}. \label{eq:third}
\end{align}

\noindent \textbf{Step 2}. By a convolution inequality, we have
$$\Vert \left(  \mathrm{I}-\Delta\right)^{\frac{\beta}{2}} e^{t\Delta} u^N_0\Vert_{ L^{\rK}(\R^d)}  \leq \| e^{t\Delta}\|_{L^p\to L^p}  \Vert \left(  \mathrm{I}-\Delta\right)^{\frac{\beta}{2}} u^N_0\Vert_{ L^{\rK}(\R^d)} \leq C \|u_{0}^N\|_{\beta,\rK}.$$

\medskip
\noindent \textbf{Step 3}. Let us come to the second term \eqref{eq:second}:

\begin{align*}
&   \int_{0}^{t}  \big\Vert \left(\mathrm{I}-\Delta\right)^{\frac{\beta}{2}} \nabla \cdot 
e^{\left(  t-s\right) \Delta} 
\left(  V^{N} \ast\left( F( K \ast u^N_s) \mu_{s}^{N}\right)  \right)  \big\Vert_{L^{\rK}(\R^d) }ds 
\\&    \leq  C \ \int_{0}^{t}\big\Vert  \left(  \mathrm{I}-\Delta\right)^{\frac{\beta}{2}}    \nabla \cdot e^{ (t-s) \Delta}\big\Vert _{L^{\rK}
\rightarrow L^{\rK}  }\ \big\Vert  \left(
 V^{N} \ast\left(  F( K\ast u^N_s) \mu_{s}^{N}\right)
\right)  \big\Vert _{L^{\rK}(\R^d)}ds.
\end{align*}
Using the equivalence of norms described in \cite[Eq. (3) p.59]{TriebelTh} (note that the space $F^s_{p,2}$ in \cite{TriebelTh} is the Bessel space used here), we get
\begin{align*}
\big\Vert   \left( \mathrm{I}-\Delta\right)^{\frac{\beta}{2}}    \nabla \cdot e^{ (t-s) \Delta}\big\Vert _{L^{\rK}
\rightarrow L^{\rK}  } \leq  C\, \big\Vert   \left( \mathrm{I}-\Delta\right)^{\frac{1+\beta}{2}}    e^{ (t-s) \Delta}\big\Vert _{L^{\rK}
\rightarrow L^{\rK}  }.
\end{align*}
Hence in view of Inequality \eqref{eq:heat-op-norm-frac}, we have that
\[
\big\Vert   \left( \mathrm{I}-\Delta\right)^{\frac{\beta}{2}}    \nabla \cdot e^{ (t-s) \Delta}\big\Vert _{L^{\rK}
\rightarrow L^{\rK}  }\ \leq  C \frac{1}{(t-s)^{\frac{(1+\beta)}{2}}}.  
 \]
 Thus,
 \begin{align*}
 \int_{0}^{t}\big\Vert \left(  \mathrm{I}-\Delta\right)^{\frac{\beta}{2}} \nabla \cdot 
e^{\left(  t-s\right) \Delta}\left(  V^{N} \ast \left(
F( K \ast u^N_s) \mu_{s}^{N}\right)  \right)  \big\Vert _{L^{\rK}(\R^d)}ds 
&\leq C \int_{0}^{t}  \frac{1}{(t-s)^{\frac{1+\beta}{2}} }  \|  u_{s}^{N} \|_{L^{\rK}(\R^d)}  ds\\
& \leq C\int_{0}^{t}  \frac{1}{(t-s)^{\frac{1+\beta}{2}} }  \|   \left(  \mathrm{I}-\Delta\right)^{\frac{\beta}{2}} u_{s}^{N} \|_{L^{\rK}(\R^d)}  ds,
 \end{align*}
 using \eqref{eq:ineqBesselnorms} in the last inequality.

\medskip

\noindent\textbf{Step 4}.
Recalling the notation introduced in \eqref{eq:defM}, the third term \eqref{eq:third} is nothing but $\|M^{{N}}_{t}\|_{\beta,\rK} $. Plugging the result of Steps $2$ and $3$ in \eqref{eq:first}-\eqref{eq:second}-\eqref{eq:third},  we can now use the Gr\"onwall lemma for convolution integrals (see e.g. \cite[Lemma 7.1.1]{Henry}). Then we obtain that there exists a constant $C>0$ such that
\begin{equation}\label{eq:suputN}
\sup_{t\in [0,T]} \| u_{t}^N\|_{\beta,\rK} \leq C \left(\|u^N_0\|_{\beta,\rK} + \sup_{t\in [0,T]} \|M^{{N}}_{t}\|_{\beta,\rK} \right) .
\end{equation}

\medskip

\noindent\textbf{Step 5}.
 The embedding for Bessel potential spaces of \cite[p.203]{Triebel} gives that $H^{\beta + d(\frac{1}{2}-\frac{1}{\rK})}(\R^d)$ is continuously embedded into $H^\beta_{\rK}(\R^d)$, thus we obtain
\begin{align*}
 \|M^N_t\|_{\beta,\rK} \leq C \, \|M^N_t\|_{\beta + d(\frac{1}{2}-\frac{1}{\rK}),2}  .
\end{align*}
Now combining Proposition \ref{prop:martingale-bound-Hbeta} with Lemma \ref{lem:GRR} as in the proof of Proposition \ref{prop:martingale-bound}, we get that for any $\varepsilon>0$, there exists $C>0$ such that
$$\left\| \sup_{t\in [0,T]}\|M^N_t\|_{\beta,\rK} \right\|_{L^q(\Omega)}\leq C\,  N^{-\frac{1}{2} \left( 1-\alpha(d + 2\beta + 2d(\frac{1}{2}-\frac{1}{\rK})) \right) +\varepsilon}.$$
In view of the constraint in \eqref{eq:alfa_appendix}, it follows that
\begin{align*}
\left\|\sup_{t\in [0,T]} \|M^N_t\|_{\beta,\rK} \right\|_{L^q(\Omega)}\leq C.
\end{align*}
Taking the $L^q(\Omega)$ norm in \eqref{eq:suputN} and applying the previous inequality, we obtain the desired result.
\end{proof}



~

~

\end{document}